\PassOptionsToPackage{nopatch=footnote}{microtype}
\documentclass[twoside,11pt,a4paper]{article}
\usepackage[utf8]{inputenc}
\usepackage[T1]{fontenc}
\usepackage{amsthm}
\usepackage{macros}
\usepackage{tabularx}
\usepackage{xurl}
\usepackage{authblk}
\usepackage[square,numbers,comma,sort&compress]{natbib}
\usepackage{arxiv_layout}
\definecolor{darkblue}{rgb}{0.0,0.0,0.7}
\usepackage[linktocpage=true,hyperfootnotes=false,colorlinks=true,
citecolor=darkblue,urlcolor=darkblue,linkcolor=blue,backref=page]{hyperref}

\renewcommand*{\backrefalt}[4]{\ifcase #1\or
  (Cited on page:~#2.)\else (Cited on pages:~#2.)\fi}

\newtheorem{theorem}{Theorem}
\newtheorem{lemma}{Lemma}
\newtheorem{proposition}{Proposition}
\newtheorem{corollary}{Corollary}
\newtheorem{definition}{Definition}
\newtheorem{remark}{Remark}
\newtheorem{example}{Example}
\newtheorem{assumption}{Assumption}

\newcommand{\DMNc}{\mathcal N_c}
\newcommand{\DMSc}{\mathcal S_c}
\newcommand{\DMbmu}{\bm\mu}
\newcommand{\DMbnu}{\bm\nu}
\newcommand{\DMbfv}{\bm f}
\newcommand{\DMbg}{\bm g}

\newcommand{\DMdd}{\,\mathrm d}
\newcommand{\pr}{\mathsf{pr}}

\newcommand{\osc}{\mathsf{osc}}

\newcommand{\exc}{\mathsf{exc}}

\newcommand{\DMargmax}{\operatorname*{arg\,max}}
\newcommand{\DMTV}{\mathrm{TV}}

\newcommand{\DMdoi}[1]{\href{https://doi.org/#1}{doi:\nolinkurl{#1}}}
\newcommand{\DMarxiv}[1]{\href{https://arxiv.org/abs/#1}{arXiv:\nolinkurl{#1}}}

\newcommand{\DMMOT}{\mathsf{MOT}_c}
\newcommand{\DMdist}{\operatorname{dist}}

\allowdisplaybreaks[2]
\hypersetup{pdftitle={Mutually c-Conjugate Potentials in Multimarginal Optimal Transport: Dual Representation and Stability},
pdfauthor={Yehya Cheryala; Mokhtar Z. Alaya; Salim Bouzebda},
pdfsubject={Multimarginal optimal transport: duality, normalization and stability}}

\title{Mutually \texorpdfstring{$c$}{c}-Conjugate Potentials in Multimarginal Optimal Transport: Dual Representation and Stability}
\author[1]{Yehya Cheryala}
\author[1]{Mokhtar Z. Alaya\thanks{Corresponding author.}}
\author[1]{Salim Bouzebda}
\affil[1]{\footnotesize
Universit\'e de Technologie de Compi\`egne,\authorcr
LMAC (Laboratoire de Math\'ematiques Appliqu\'ees de Compi\`egne),\authorcr
CS 60 319, 60203 Compi\`egne Cedex, France\authorcr\vspace{0.5em}
\texttt{yehya.cheryala@utc.fr},\quad
\texttt{alayaelm@utc.fr},\quad
\texttt{salim.bouzebda@utc.fr}}
\date{November 16, 2026}

\begin{document}
\maketitle

\begin{abstract}
We establish dual representation and stability results for multimarginal
optimal transport  (MOT) with bounded continuous costs. Our approach uses mutually
$c$-conjugate potentials: each component is the largest function compatible
with the dual constraint when the others are fixed. A single sweep of
successive component updates produces such a family from any real-valued
feasible family while preserving contact points. The resulting potentials are
bounded and upper semicontinuous. On fixed compact metric spaces,
normalization at fixed reference points yields a compact class of continuous
potentials in the product uniform norm. This class is independent of the
marginals and contains a dual maximizer for every marginal tuple. We prove its
Lipschitz dependence on the cost in Hausdorff distance. Under uniform cost
convergence and weak marginal convergence, transport values converge. The
corresponding normalized optimizer sets are upper stable in the product
uniform norm, without uniqueness assumptions. 
We also obtain value bounds in total
variation and, for coordinatewise Lipschitz costs, Wasserstein distance. For
coordinatewise Lipschitz reference costs, a power-growth condition on the dual
objective gives one-sided H\"older estimates for optimizer sets under joint
perturbations. On finite spaces, polyhedral structure yields linear growth and
one-sided Lipschitz estimates with reference-dependent constants. On Polish
spaces, total boundedness of the coordinate cost profiles in the uniform norm
yields a compact normalized representation. For a fixed cost satisfying this
condition, value continuity and upper stability of normalized conjugate
optimizer sets hold under weak marginal convergence.
\end{abstract}

\begin{mscclassification}
49N15; 49K40; 90C46.
\end{mscclassification}


\section{Introduction}\label{section: Introduction}

Optimal transport minimizes the cost of coupling prescribed probability
measures. Its dual formulation assigns a potential to each marginal
space, subject to a joint constraint imposed by the cost. Optimal potentials certify the transport value and connect the
optimization problem to the geometry of optimal couplings. The subject
originates in Monge's formulation in terms of maps~\cite{monge1781memoire}
and Kantorovich's relaxation to couplings~\cite{article}; its foundations
are developed in~\cite{villani2003,villani2008,santambrogio2015}.
Multimarginal optimal transport (MOT) considers three or more prescribed
marginals. It appears in Wasserstein barycenters~\cite{Agueh2011},
matching problems~\cite{carlier2010}, and the applications surveyed
by Pass~\cite{pass2015}. The geometry of its optimal plans can differ
substantially from that of two-marginal transport~\cite{gangbo1998,pass2012,colombo2015}.

We develop a dual representation based on mutually $c$-conjugate
potentials and use it to analyze stability under changes in the
marginals or the cost. Continuity of the transport value does not by
itself describe the behavior of the potentials. Two difficulties
arise even for continuous costs on compact spaces. Adding constants
whose sum is zero preserves the dual constraint and objective, so
potentials must first be normalized in a consistent way. Moreover,
normalization need not make the optimizer unique. We therefore compare
\emph{sets} of normalized mutually $c$-conjugate optimal potentials.
The problem is also
relevant when marginal measures are approximated from data, since
dual representations enter the analysis of empirical transport
values~\cite{hundrieser2022}.

The main setting of the paper is a bounded continuous cost. Let $X_1,\ldots,X_K$ be Polish spaces, with $K\ge3$, and let
$\boldsymbol\mu=(\mu_1,\ldots,\mu_K)$ be probability marginals.
Writing $X=\prod_{k=1}^K X_k$ and denoting by
$\Pi(\boldsymbol\mu)$ the probability measures on $X$ with these
marginals, the primal value is
\[
\mathsf{MOT}_c(\boldsymbol\mu)
=\inf_{\pi\in\Pi(\boldsymbol\mu)}\int_X c(\boldsymbol x)\,d\pi(\boldsymbol x).
\]
The dual problem maximizes
$J_{\boldsymbol\mu}(\boldsymbol f)=\sum_{k=1}^K\int_{X_k}f_k\,d\mu_k$
over Borel measurable, marginally integrable potentials satisfying
\begin{equation}\label{eq:dual-feasibility}
\sum_{k=1}^K f_k(x_k)\le c(\boldsymbol x)
\qquad\text{for every }\boldsymbol x\in X.
\end{equation}
Equality in this constraint defines contact points; equality on an
optimal support gives a splitting family on that support.

Our approach uses \emph{mutually $c$-conjugate} families: each potential
is the largest function compatible with the dual constraint when the
other potentials are fixed. The corresponding operation is called a
partial $c$-transform. Successively updating all components once
converts a real-valued feasible family into a mutually conjugate
family and preserves its existing contact points
(Lemma~\ref{Lemma: sequential-conjugation}). For bounded continuous
costs, these potentials are bounded and upper semicontinuous. This
successive-transform mechanism was used by
Griessler~\cite{griessler2016} for splitting potentials. The construction
of dual maximizers also uses the classical duality framework supplied
by Kellerer's theory~\cite{kellerer}, recalled here on compact and
Polish spaces. Related developments for measurable
or extended-valued costs are given
in~\cite{beiglbock2012general,beiglbock2012generalized}.
Bartz et al.~\cite{bartz2018} construct explicit splitting potentials
for certain costs assembled from pairwise terms. These foundations
are distinguished from the perturbation results below in
Section~\ref{sec:related}.

\paragraph{A common compact representation.}
On fixed compact metric marginal spaces, we normalize the first
$K-1$ potentials to vanish at reference points
$x_k^\circ\in X_k$, chosen independently of the marginals. Let
$\mathcal N_c$ denote the class of mutually $c$-conjugate families
with this normalization. Conjugacy controls their oscillations and
moduli of continuity; normalization gives uniform bounds. The Arzel\`a--Ascoli theorem and closedness
under uniform limits then make $\mathcal N_c$ compact in the product
uniform norm. It also contains a dual maximizer for
\emph{every} marginal tuple:
\begin{equation}\label{dm:eq:normalized-representation}
\mathsf{MOT}_c(\boldsymbol\mu)
=\max_{\boldsymbol f\in\mathcal N_c}
J_{\boldsymbol\mu}(\boldsymbol f).
\end{equation}
Thus changing the marginals changes only the objective on a common
compact class (Theorem~\ref{dm:theorem:compact-attainment}). This
representation is the basis for comparing optimal families without
a uniqueness assumption.

We also quantify how the class itself changes with the cost. For
continuous costs $c$ and $\widetilde c$ on the same compact product,
using the same reference points, we prove
\begin{equation}\label{dm:eq:hausdorff-cost-class}
d_H^\infty(\mathcal N_c,\mathcal N_{\widetilde c})
\le(2^K-1)\|c-\widetilde c\|_\infty,
\end{equation}
where $d_H^\infty$ is the Hausdorff distance induced by the product
uniform norm (Theorem~\ref{dm:theorem:cost-class-stability}).
This estimate allows us to compare dual problems whose admissible
conjugate classes differ. Under uniform convergence of costs and weak
convergence of each marginal, the transport values converge. The
corresponding normalized conjugate optimizer sets are upper stable
in the product uniform norm
(Theorem~\ref{dm:theorem:joint-optimizer-stability}). In particular,
every sequence of such perturbed optimizers has a uniformly convergent
subsequence, and every uniform cluster point solves the limiting dual
problem. For a fixed cost, these conclusions follow directly from
maximization over $\mathcal N_c$
(Theorem~\ref{dm:theorem:optimizer-stability}).

\paragraph{Beyond compact marginal spaces.}
The same representation remains available on Polish spaces under a
condition on the cost sections. We assume that, in each coordinate,
the family of sections is totally bounded in the uniform norm:
finitely many sections approximate all the others uniformly in the
remaining coordinates, at any prescribed accuracy. Under this
condition, $\mathcal N_c$ is compact and contains a dual maximizer
for every marginal tuple
(Proposition~\ref{dm:proposition:noncompact-class}). For a fixed cost,
weak marginal convergence again gives value continuity and upper
stability (Corollary~\ref{dm:corollary:noncompact-stability}).
This sufficient condition requires neither compact marginal spaces
nor a splitting family on an optimal support. A separate construction
on general Polish spaces applies when an optimal support admits a
real-valued splitting family. It yields bounded upper semicontinuous
mutually conjugate maximizers
(Theorem~\ref{Theorem: noncompact c-conjugate}). On compact spaces,
we obtain continuous mutually conjugate maximizers without this
additional assumption (Theorem~\ref{Theorem: compact c-conjugate}).

\paragraph{Quantitative bounds and the role of nonuniqueness.}
On fixed compact spaces, we bound changes in transport values by
uniform changes in the cost and total variation changes in the
marginals. Coordinatewise Lipschitz costs also give Wasserstein
bounds (Proposition~\ref{dm:proposition:value-bounds}). For potentials,
we first control the dual objective gap and then relate that gap to
the uniform distance from the reference optimizer set. For a
coordinatewise Lipschitz reference cost, an additional power-growth
condition yields one-sided H\"older estimates under marginal or
joint cost-and-marginal perturbations
(Theorems~\ref{dm:theorem:growth-stability}
and~\ref{dm:theorem:quant-joint}). On finite marginal spaces, the
normalized conjugate class is a finite union of compact polytopes,
which supplies a linear growth bound without a separate growth
assumption (Theorem~\ref{dm:theorem:finite-growth}). The resulting
one-sided Lipschitz estimates require neither uniqueness nor full
support of the marginals; their constants may depend on the reference
problem (Corollaries~\ref{dm:corollary:finite-lipschitz}
and~\ref{dm:corollary:finite-joint}).

All these optimizer estimates are directed from perturbed solutions
toward the reference optimizer set. The reverse approximation can
fail: a finite three-marginal example has a limiting optimizer that
cannot be approximated by perturbed optimizers
(Example~\ref{dm:example:nonuniqueness}). Thus Lipschitz dependence
of the full classes $\mathcal N_c$ does not imply Hausdorff continuity
of their optimizer subsets. A separate three-marginal model with a
quadratic pairwise cost illustrates conjugacy, splitting, and the
value bounds (Section~\ref{sec:geometric-illustration}).

Figure~\ref{fig:theory-roadmap} summarizes how conjugation and fixed-reference
normalization lead to the common compact representation and the
stability results on compact spaces.

\begin{figure}[tbp]
\centering
\includegraphics[width=\linewidth]{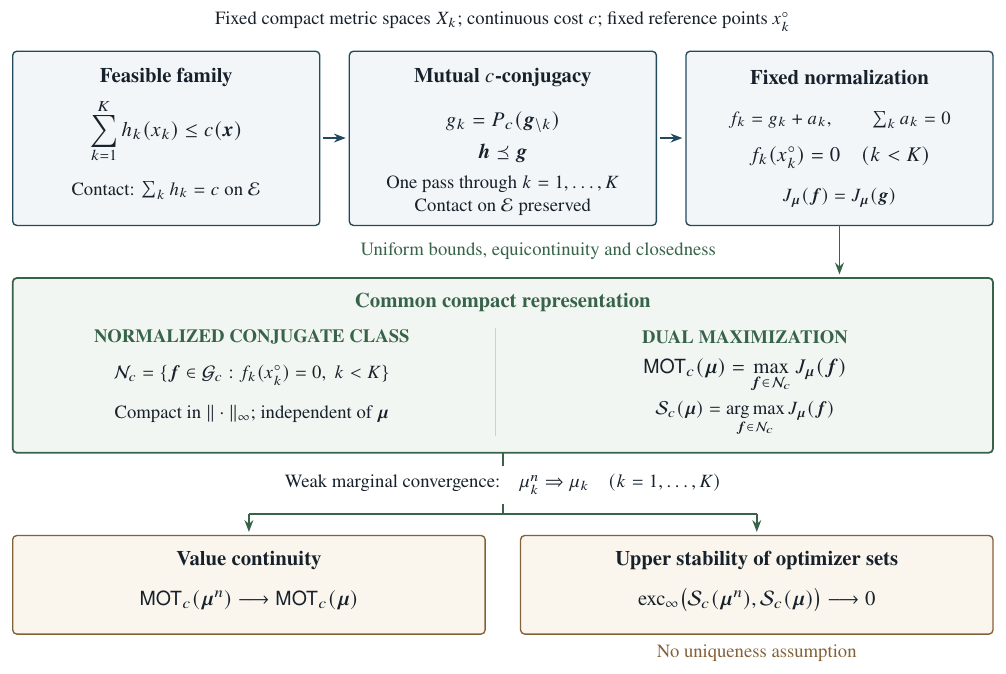}
\caption{Construction and stability of normalized dual potentials on compact spaces, for a fixed continuous cost and fixed reference points. Top: a single conjugation sweep transforms a real-valued feasible family $\boldsymbol h$ into a mutually $c$-conjugate family $\boldsymbol g$, preserving existing contact points $\mathcal E$. A zero-sum shift then gives the normalized family $\boldsymbol f$. Middle: uniform bounds, equicontinuity and closedness yield the common compact class $\mathcal N_c$, independent of the marginals. Classical duality identifies the maximum of $J_{\boldsymbol\mu}$ over this class with $\mathsf{MOT}_c(\boldsymbol\mu)$. Bottom: weak marginal convergence gives value continuity and upper stability of the optimizer sets $\mathcal S_c(\boldsymbol\mu)$, without uniqueness. Conjugation alone does not ensure optimality. Here $\mathcal G_c$ denotes the mutually conjugate class, $\preceq$ componentwise pointwise order, $\Rightarrow$ weak convergence, and $\exc_\infty(A,B)= \sup_{\boldsymbol u\in A}\inf_{\boldsymbol v\in B} \|\boldsymbol u-\boldsymbol v\|_\infty$ is the one-sided excess in the product uniform norm.}
\label{fig:theory-roadmap}
\end{figure}

\paragraph{Organization of the paper.}
Sections~\ref{Preliminaries}--\ref{section: Optimal dual potentials via c-conjugacy}
give the preliminaries, classical duality, and the construction of
mutually conjugate maximizers.
Section~\ref{dm:section:conjugacy} develops the normalized compact
representation, its Polish-space extension, and an explicit illustration.
Sections~\ref{dm:section:qualitative stability}
and~\ref{dm:section:quantitative stability} establish the qualitative
and quantitative perturbation results.
Section~\ref{sec:related} places the results in the literature, and
Section~\ref{section: conclusion} concludes.
Appendix~\ref{section: Appendix} contains supporting proofs and results.


\section{Preliminaries and weak duality}\label{Preliminaries}

\noindent\textbf{Notation and setting. } Fix $K\ge3$. Let $X_1,\ldots,X_K$ be Polish spaces equipped with their Borel $\sigma$-algebras, and let $\mu_k\in\mathcal P(X_k)$ for $k=1,\ldots,K$. Here $\mathcal P(X_k)$ denotes the set of Borel probability measures on $X_k$. We recall 
\begin{equation*} 
X=\prod_{k=1}^K X_k,\qquad \boldsymbol\mu=(\mu_1,\ldots,\mu_K),\qquad \boldsymbol x=(x_1,\ldots,x_K)\in X, 
\end{equation*} 
and use $\boldsymbol f=(f_1,\ldots,f_K)$ for a family of potentials. For each $k$, the notation 
\begin{equation*} 
X_{\setminus k}=\prod_{l\ne k}X_l,\qquad \boldsymbol x_{\setminus k}=(x_l)_{l\ne k},\qquad \boldsymbol f_{\setminus k}=(f_l)_{l\ne k} 
\end{equation*} 
omits the $k$-th component. For $x_k\in X_k$ and $\boldsymbol z_{\setminus k}\in X_{\setminus k}$, $(x_k,\boldsymbol z_{\setminus k})$ denotes the point of $X$ with $k$-th coordinate $x_k$ and remaining coordinates $\boldsymbol z_{\setminus k}$.

We denote by $\mathcal C(X_k)$ and $\mathcal C_b(X_k)$ the spaces of continuous real-valued functions and bounded continuous real-valued functions on $X_k$, respectively. We write $f_k\in L_1(\mu_k)$ when $f_k:X_k\to\mathbb R$ is Borel measurable and $\int_{X_k}|f_k|\,d\mu_k<\infty.$ Throughout, potentials are specified representatives defined at every point of the marginal spaces.

\noindent\textbf{Couplings and transport values. } We retain the notation $\Pi(\boldsymbol\mu)$ for the set of couplings introduced in Section~\ref{section: Introduction}. For a bounded Borel measurable cost $c:X\to\mathbb R$, define 
\begin{equation*} 
I(\pi)=\int_X c(\boldsymbol x)\,d\pi(\boldsymbol x), \qquad \pi\in\Pi(\boldsymbol\mu). 
\end{equation*} 
The corresponding multimarginal transport value is 
\begin{equation*} 
\mathsf{MOT}_c(\boldsymbol\mu) =\inf_{\pi\in\Pi(\boldsymbol\mu)}I(\pi). 
\end{equation*}

\noindent\textbf{Dual potentials and admissibility. } For potentials satisfying $f_k\in L_1(\mu_k)$ for every $k$, define 
\begin{equation*} 
J_{\boldsymbol\mu}(\boldsymbol f) =\sum_{k=1}^K\int_{X_k}f_k\,d\mu_k. 
\end{equation*} 
With the marginals fixed, we abbreviate $J_{\boldsymbol\mu}(\boldsymbol f)$ to $J(\boldsymbol f)$. Subscripts will be retained when different marginal tuples are compared. The admissible dual class is 
\begin{equation*} 
\mathcal F_c =\left\{ \boldsymbol f=(f_1,\ldots,f_K): \begin{array}{l} f_k\in L_1(\mu_k)\quad(k=1,\ldots,K),\\[2pt] \displaystyle\sum_{k=1}^K f_k(x_k)\le c(\boldsymbol x) \quad\text{for every }\boldsymbol x\in X \end{array} \right\}. 
\end{equation*} 
We suppress the dependence of $\mathcal F_c$ on the fixed marginals. The inequality defining $\mathcal F_c$ and all $c$-transforms below are understood pointwise. Thus changing a potential on a marginal null set need not preserve admissibility.

With this notation, the primal and dual Kantorovich problems are 
\begin{align*} 
(\mathsf{MOT}\!-\!\mathsf{KP})&:\quad \inf\{I(\pi):\pi\in\Pi(\boldsymbol\mu)\},\\ (\mathsf{MOT}\!-\!\mathsf{DP})&:\quad \sup\{J(\boldsymbol f):\boldsymbol f\in\mathcal F_c\}. 
\end{align*} 
The following weak duality inequality requires neither compactness of the marginal spaces nor continuity of the cost.

\begin{lemma}[Weak duality]\label{Lemma: weak duality} 
Let $X_1,\ldots,X_K$ be Polish spaces, let $\mu_k\in\mathcal P(X_k)$ for $k=1,\ldots,K$, and let $c:X\to\mathbb R$ be bounded and Borel measurable. Then 
\begin{equation}\label{eq:weak-duality}
\sup_{\boldsymbol f\in\mathcal F_c}J(\boldsymbol f) \le \inf_{\pi\in\Pi(\boldsymbol\mu)}I(\pi).
\end{equation} 
\end{lemma}

\begin{proof} 
Fix $\boldsymbol f\in\mathcal F_c$ and $\pi\in\Pi(\boldsymbol\mu)$. For every $k$, the marginal condition gives 
\begin{equation*} 
\int_X|f_k(x_k)|\,d\pi(\boldsymbol x) =\int_{X_k}|f_k|\,d\mu_k<\infty. 
\end{equation*} 
Thus each $f_k\circ\pr_k$ is $\pi$-integrable. Using the marginal identities and the pointwise dual constraint, we obtain 
\begin{align}
J(\boldsymbol f) &=\sum_{k=1}^K\int_{X_k}f_k\,d\mu_k =\int_X\sum_{k=1}^K f_k(x_k)\,d\pi(\boldsymbol x)
\label{eq:marginal-dual-identity}\\
&\le\int_X c(\boldsymbol x)\,d\pi(\boldsymbol x) =I(\pi).\notag
\end{align} Since this holds for every admissible $\boldsymbol f$ and every coupling $\pi$, taking the supremum over $\boldsymbol f$ and the infimum over $\pi$ proves the claim. 
\end{proof}

\section{Kantorovich duality for MOT} \label{section:duality}

This section gives proofs of classical Kantorovich duality in compact and non-compact settings. The compact proof uses Fenchel--Rockafellar duality, while the extension to Polish spaces relies on tightness and compact approximation. Detailed proofs are given in Appendix~\ref{section: Appendix}.

\subsection{Compact case} \label{section:duality-compact}

Throughout this subsection, the spaces $X_k$ are compact metric spaces, the marginals $\mu_k\in\mathcal P(X_k)$ are fixed, and $c:X\to\mathbb R_+$ is continuous. We work on the Banach space $E=\mathcal C(X)$, equipped with the supremum norm $\|u\|_\infty=\sup_{\boldsymbol x\in X}|u(\boldsymbol x)|$. To prove Theorem~\ref{theorem: duality-compact}, we introduce two convex functionals on $E$. The next two lemmas verify the hypotheses of the Fenchel--Rockafellar theorem stated in Theorem~\ref{SuppRes:Theo:Fenchel}.

\begin{lemma}
 \label{Lemma: Theta}
  The functional $\Theta:E\to\mathbb R\cup\{\infty\}$ defined by
\begin{equation*}
\Theta(u)= 
\begin{cases} 0 & \text{if }u(\boldsymbol x)\ge-c(\boldsymbol x) \text{ for every }\boldsymbol x\in X,\\[1mm] \infty & \text{otherwise} 
\end{cases}
\end{equation*}
is convex and continuous at the constant function $u(\boldsymbol x)=1$ for every $\boldsymbol x\in X$. 
\end{lemma}

\begin{lemma}
 \label{Lemma: Xi}
  The functional $\Xi:E\to\mathbb R\cup\{\infty\}$ defined by
\begin{equation*}
\Xi(u)= 
\begin{cases} \displaystyle\sum_{k=1}^K\int_{X_k}f_k\,d\mu_k & \text{if }u(\boldsymbol x)=\displaystyle\sum_{k=1}^K f_k(x_k) \text{ on }X,\\[2mm] \infty & \text{otherwise}, 
\end{cases}
\end{equation*}
where $f_k\in\mathcal C(X_k)$ for every $k$, is well defined and convex. Moreover, $\Xi(1)<\infty$, where $1$ denotes the constant function on $X$. \end{lemma}

These lemmas yield the compact duality theorem.

\begin{theorem}
 \label{theorem: duality-compact}
  Fix compact metric spaces $X_1,\ldots,X_K$ and marginals $\mu_k\in\mathcal P(X_k)$ for $k=1,\ldots,K$. For every continuous cost $c:X\to\mathbb R_+$, the duality identity
\begin{equation*}
\mathsf{MOT}_c(\boldsymbol\mu) =\inf_{\pi\in\Pi(\boldsymbol\mu)} I(\pi) =\sup_{\boldsymbol f\in \mathcal F_c\cap\prod_{k=1}^K\mathcal C(X_k)}J(\boldsymbol f) =\sup_{\boldsymbol f\in\mathcal F_c}J(\boldsymbol f)
\end{equation*}
holds. 
\end{theorem}

\subsection{Non-compact case}
\label{section:duality-noncompact}

When the marginal spaces are Polish, tightness of the marginals implies tightness of the associated couplings. For fixed marginals, the marginal constraints are also preserved under weak convergence. Together with Prokhorov's theorem, these facts give weak compactness of $\Pi(\boldsymbol\mu)$, as recorded in the next two lemmas.

\begin{lemma}
 \label{Lemma: tight}
  Let $X_1,\ldots,X_K$ be Polish spaces, and let $Q_k\subset\mathcal P(X_k)$ be tight for each $k=1,\ldots,K$. The set $\Pi(Q_1,\ldots,Q_K)$ of probability measures on $X$ whose $k$-th marginal belongs to $Q_k$ for every $k$ is tight in $\mathcal P(X)$. 
  \end{lemma}

\begin{lemma}
 \label{Lemma: compact space}
  Let $X_1,\ldots,X_K$ be Polish spaces, and fix $\mu_k\in\mathcal P(X_k)$ for $k=1,\ldots,K$. Then $\Pi(\boldsymbol\mu)$ is compact in $\mathcal P(X)$ for the topology of weak convergence. \end{lemma}

For bounded continuous costs, compact approximation then extends the duality identity to the full Polish spaces.

\begin{theorem}
 \label{theorem: duality-noncompact}
  Fix Polish spaces $X_1,\ldots,X_K$ and marginals $\mu_k\in\mathcal P(X_k)$ for $k=1,\ldots,K$. For every bounded continuous cost $c:X\to\mathbb R$, the duality identity
\begin{equation*}
\mathsf{MOT}_c(\boldsymbol\mu) =\inf_{\pi\in\Pi(\boldsymbol\mu)} I(\pi) =\sup_{\boldsymbol f\in\mathcal F_c}J(\boldsymbol f)
\end{equation*}
holds. 
\end{theorem}

\begin{remark} 
The nonnegativity assumption on the cost in Theorem~\ref{theorem: duality-compact} is inessential. Indeed, let $c:X\to\mathbb R$ be continuous on the compact product $X$, choose $M\ge-\inf_X c$, and set $\widetilde c=c+M\ge0$. For every $\pi\in\Pi(\boldsymbol\mu)$, $\int_X\widetilde c\,d\pi=\int_X c\,d\pi+M$, so
\begin{equation*}
\inf_{\pi\in\Pi(\boldsymbol\mu)}\int_X\widetilde c\,d\pi =\inf_{\pi\in\Pi(\boldsymbol\mu)}\int_X c\,d\pi+M.
\end{equation*}
The dual value has the same translation property. The map $(f_1,f_2,\ldots,f_K)\longmapsto(f_1+M,f_2,\ldots,f_K)$ 
is a bijection from $\mathcal F_c$ onto $\mathcal F_{\widetilde c}$, with inverse $(g_1,g_2,\ldots,g_K)\mapsto(g_1-M,g_2,\ldots,g_K)$. It adds $M$ to $J$ and preserves continuity of the components. Hence
\begin{equation*}
\sup_{\boldsymbol f\in\mathcal F_{\widetilde c}}J(\boldsymbol f) =\sup_{\boldsymbol f\in\mathcal F_c}J(\boldsymbol f)+M,
\end{equation*}
and the same identity holds when the suprema are restricted to continuous potentials. Applying Theorem~\ref{theorem: duality-compact} to $\widetilde c$ and subtracting $M$ proves the compact duality identity for real-valued continuous costs. This allows us to apply that theorem to compact truncations of the cost in Theorem~\ref{theorem: duality-noncompact}. 
\end{remark}

\paragraph{Proof sketch of Theorem~\ref{theorem: duality-noncompact}.} The proof proceeds by compact truncation, normalization and successive improvement of the truncated potentials, followed by extension to the full spaces. We outline these steps here; the detailed estimates are given in Appendix~\ref{section: Appendix}.

By Lemma~\ref{Lemma: compact space}, $\Pi(\boldsymbol\mu)$ is weakly compact. Since $c\in\mathcal C_b(X)$, the functional $I$ is continuous for weak convergence, so an optimal plan $\pi_*\in\Pi(\boldsymbol\mu)$ exists. Fix $0<\varepsilon<1/K$. Tightness provides compact sets $X_k^0\subset X_k$ with $\mu_k(X_k\setminus X_k^0)\le\varepsilon$. Their product $X^0=\prod_{k=1}^K X_k^0$ satisfies
\begin{equation}\label{5}
\pi_*(X\setminus X^0)\le K\varepsilon.
\end{equation}
The normalized restriction
\begin{equation}\label{uuu}
\pi_*^0=\frac{\mathbf1_{X^0}}{\pi_*(X^0)}\pi_*
\end{equation}
is therefore well defined. Let $\mu_k^0$ be its induced marginals, supported on $X_k^0$. Replacing the part of $\pi_*$ on $X^0$ by a compact optimal plan with these marginals, scaled by $\pi_*(X^0)$, preserves the original marginals. This comparison, together with boundedness of $c$, shows that the full and compact primal values differ by at most $2K\|c\|_\infty\varepsilon$.

On $X^0$, define the truncated dual objective by
\begin{equation*}
J^0(\boldsymbol h) =\sum_{k=1}^K\int_{X_k^0}h_k\,d\mu_k^0.
\end{equation*}
The compact duality theorem provides a continuous admissible family $\widetilde{\boldsymbol f}^{\,0}$ whose value under $J^0$ is within $\varepsilon$ of the compact primal value. Its near-optimality ensures that there is a point $\boldsymbol x^0\in X^0$ such that
\begin{equation*}
\sum_{k=1}^K\widetilde f_k^{\,0}(x_k^0) \ge-\|c\|_\infty-1.
\end{equation*}
Adding constants with zero sum makes the component values at this point equal, without changing feasibility or $J^0$. Keeping the notation $\widetilde{\boldsymbol f}^{\,0}$ for the normalized family, admissibility with all but one coordinate fixed at $\boldsymbol x^0$ gives
\begin{equation}\label{tild}
\widetilde f_k^{\,0}(x_k)\le2\|c\|_\infty+1
\qquad(x_k\in X_k^0,\ k=1,\ldots,K).
\end{equation}

We next improve the potentials successively. At each step, one component is replaced by the infimum of the cost minus the sum of the other components, taken at their current values. This replacement increases the chosen component and preserves admissibility. Continuity is preserved because the infimum is taken over a fixed compact product. After all $K$ components have been updated, the resulting family $\bar{\boldsymbol f}^{\,0}$ satisfies
\begin{equation*}
J^0(\bar{\boldsymbol f}^{\,0}) \ge J^0(\widetilde{\boldsymbol f}^{\,0}), \qquad |\bar f_k^{\,0}|\le C\quad\text{on }X_k^0,
\end{equation*}
where $C>0$ depends only on $K$ and $\|c\|_\infty$. The upper bounds follow from admissibility and the preserved lower bounds at the reference point. The infimum defining each updated component then gives a uniform lower bound as well.

To extend this family, set $M=(K-1)C+\|c\|_\infty$ and define                                                                                                                     
\begin{equation*}
f_k(x_k)= 
\begin{cases} \bar f_k^{\,0}(x_k), & x_k\in X_k^0,\\ -M, & x_k\notin X_k^0.
\end{cases}
\end{equation*}
On $X^0$, feasibility follows from the compact construction. If at least one coordinate lies outside its compact set, then
\begin{equation*}
\sum_{k=1}^K f_k(x_k) \le-M+(K-1)C =-\|c\|_\infty \le c(\boldsymbol x).
\end{equation*}
The extended functions are bounded and Borel measurable, since the compact sets are Borel and the improved potentials are continuous on those sets. Thus $\boldsymbol f\in\mathcal F_c$.

Finally, splitting the integral under $\pi_*$ into $X^0$ and its complement shows that
\begin{equation*}
\big|J(\boldsymbol f)-J^0(\bar{\boldsymbol f}^{\,0})\big| \le K^2(C+M)\varepsilon.
\end{equation*}
Combining this estimate with the compact near-optimality and the primal comparison gives
\begin{equation*}
J(\boldsymbol f) \ge\inf_{\pi\in\Pi(\boldsymbol\mu)}I(\pi) -\bigl(2K\|c\|_\infty+1+K^2(C+M)\bigr)\varepsilon.
\end{equation*}
The constants are independent of $\varepsilon$. Letting $\varepsilon\to0$ and applying weak duality from Lemma~\ref{Lemma: weak duality} proves the claimed identity. This argument follows the classical compact-approximation approach presented by Villani~\cite{villani2003}, with the marginal estimates written explicitly for $K$ coordinates.

\section{Mutually \texorpdfstring{$c$}{c}-conjugate potentials and dual optimizers}
\label{section: Optimal dual potentials via c-conjugacy}

We construct mutually $c$-conjugate families by successively improving the components of a feasible family. The construction is pointwise: it preserves feasibility and every existing equality in the dual constraint. Continuity of the cost then gives regularity of the resulting potentials. On compact metric spaces, normalization and compactness yield dual optimizers. On Polish spaces, we obtain dual optimizers from a real-valued splitting family on an optimal support, whose existence is explicitly assumed. The normalization used below removes an additive ambiguity; it does not imply uniqueness of the optimal family.

\subsection{Successive partial \texorpdfstring{$c$}{c}-transforms}

For the pointwise construction, let $X_1,\ldots,X_K$ be nonempty sets and let $c:X\rightarrow\mathbb R$ be a real-valued cost. No topology or measurability is required at this stage. A real-valued family $\boldsymbol f=(f_1,\ldots,f_K)$ is called \emph{feasible} if $\sum_{k=1}^K f_k(x_k)\le c(\boldsymbol x)$ for every $\boldsymbol x\in X$.

\begin{definition}[Mutually $c$-conjugate families]
 \label{definition: c-conjugate}
 For a real-valued family $\boldsymbol f$, the partial $c$-transform associated with its $k$-th component is
\begin{equation*}
P_c(\boldsymbol f_{\setminus k})(x_k) =\inf_{\boldsymbol x_{\setminus k}\in X_{\setminus k}} \Bigl\{c(\boldsymbol x)-\sum_{l\ne k}f_l(x_l)\Bigr\}, \qquad x_k\in X_k.
\end{equation*}
The family is \emph{mutually $c$-conjugate} if
\begin{equation}\label{eq:mutual-conjugacy}
f_k=P_c(\boldsymbol f_{\setminus k})\quad\text{on }X_k,
\qquad k=1,\ldots,K.
\end{equation} 
\end{definition}

For arbitrary real-valued components, this infimum may equal $-\infty$. It cannot equal $\infty$, since the expression is finite at every point of the nonempty set $X_{\setminus k}$. If the family is feasible, then $f_k\le P_c(\boldsymbol f_{\setminus k})$. This finite lower bound, together with evaluation at any fixed complementary tuple, shows that the transform is real-valued. Feasibility is therefore sufficient to carry out successive transforms without introducing infinite values.

When the spaces and marginals are as in Section~\ref{Preliminaries}, we also require dual admissibility, including integrability.

\begin{definition}[Admissible mutually $c$-conjugate families]
\label{definition: Gc space}
 The class of admissible mutually $c$-conjugate families is
\begin{equation*}
\mathcal G_c =\bigl\{\boldsymbol f\in\mathcal F_c\mid f_k=P_c(\boldsymbol f_{\setminus k})\text{ on }X_k \text{ for every }k=1,\ldots,K\bigr\}.
\end{equation*}
\end{definition}

The successive replacements below follow the construction used by Griessler~\cite[Lemma~2.2]{griessler2016} for splitting families. We formulate the pointwise argument for real-valued feasible families and make explicit that one pass through the components yields mutual $c$-conjugacy.

For a feasible family $\boldsymbol h$, let $T_k\boldsymbol h$ be the family obtained by replacing $h_k$ with $P_c(\boldsymbol h_{\setminus k})$ and leaving the other components unchanged. We write $\boldsymbol h\preceq\boldsymbol q$ when $h_k(x_k)\le q_k(x_k)$ for every $k$ and every $x_k\in X_k$. Starting from $\boldsymbol h$, define
\begin{equation*}
\boldsymbol h^{(0)}=\boldsymbol h,\qquad \boldsymbol h^{(k)}=T_k\boldsymbol h^{(k-1)}\quad(1\le k\le K), \qquad \boldsymbol g=\boldsymbol h^{(K)}.
\end{equation*}
Each update uses the current values of all other components, including those already updated.

\begin{lemma}[Conjugation in one pass] 
\label{Lemma: sequential-conjugation} 
For a real-valued feasible initial family, every family in the above construction is real-valued and feasible, and
\begin{equation*}
\boldsymbol h=\boldsymbol h^{(0)} \preceq\boldsymbol h^{(1)}\preceq\cdots \preceq\boldsymbol h^{(K)}=\boldsymbol g.
\end{equation*}
The final family $\boldsymbol g$ is mutually $c$-conjugate. If the initial family satisfies
\begin{equation}\label{eq:contact-equality}
\sum_{k=1}^K h_k(x_k)=c(\boldsymbol x)
\qquad(\boldsymbol x\in\mathcal E)
\end{equation}
on a set $\mathcal E\subset X$, then $\sum_{k=1}^K g_k(x_k)=c(\boldsymbol x)$ on $\mathcal E$. These conclusions hold for any order in which each component is updated exactly once. \end{lemma}
The essential point is that a component remains conjugate after its update. Subsequent updates increase the other components, so its partial transform can only decrease. Feasibility supplies the opposite inequality, preserving equality. The full argument is given in Appendix~\ref{section: Appendix}. The example \hyperlink{example-four-marginal-sweep}{\emph{A four-marginal
conjugation sweep}} in the same appendix computes all four updates for a discrete family and illustrates preservation of its contact set.

\begin{remark}[Order along a construction and between families]
\label{Remark: nonmonotone-transform}
 The increase along a single construction does not mean that $T_k$ preserves the order between different feasible families. For example, let $K=4$, $X_k=\{0,1\}$ and $c(\boldsymbol x)=\sum_{k=1}^4x_k$. The constant families $\boldsymbol h=(-2,-2,0,0)$ and $\boldsymbol q=(-1,-1,0,0)$ are feasible, with $\boldsymbol h\preceq\boldsymbol q$. Yet
\begin{equation*}
P_c(\boldsymbol h_{\setminus1})(x_1)=x_1+2,\qquad P_c(\boldsymbol q_{\setminus1})(x_1)=x_1+1,
\end{equation*}
so
\begin{equation*}
T_1\boldsymbol h=(x_1+2,-2,0,0),\qquad T_1\boldsymbol q=(x_1+1,-1,0,0).
\end{equation*}
The first updated family has a larger first component and a smaller second component. Thus the two outputs are incomparable. 
\end{remark}

Continuity of the cost also gives regularity of the transformed potentials. For each fixed complementary tuple, the functions being transformed contribute only constants to the coordinate slice of the cost. Their measurability is therefore not needed for the following conclusion.

\begin{lemma}[Upper semicontinuity of partial $c$-transforms] 
\label{Lemma: usc-conjugation}
 Let $X_1,\ldots,X_K$ be nonempty topological spaces and let $c:X\rightarrow\mathbb R$ be continuous. Fix $k\in\{1,\ldots,K\}$, and let $h_l:X_l\rightarrow\mathbb R$, $l\ne k$, be arbitrary functions. If $P_c(\boldsymbol h_{\setminus k})$ is real-valued, it is upper semicontinuous and hence Borel measurable, without any measurability assumption on the $h_l$. In particular, under these topological assumptions the output of Lemma~\ref{Lemma: sequential-conjugation} has upper semicontinuous Borel components. 
\end{lemma}

Mutual $c$-conjugacy also has an order interpretation. For nonempty topological spaces $X_k$ and $c\in\mathcal C_b(X)$, let
\begin{equation*}
\mathcal F_c^b =\left\{\boldsymbol f=(f_1,\ldots,f_K): \begin{array}{l} f_k:X_k\rightarrow\mathbb R\text{ is bounded and Borel measurable for every }k,\\[2pt] \displaystyle\sum_{k=1}^K f_k(x_k)\le c(\boldsymbol x) \quad\text{for every }\boldsymbol x\in X \end{array}\right\}.
\end{equation*}
For two families $\boldsymbol f,\boldsymbol g\in\mathcal F_c^b$, the relation $\boldsymbol f\preceq\boldsymbol g$ means that
\begin{equation*}
g_k(x_k)\ge f_k(x_k)
\qquad\text{for every }k=1,\ldots,K\text{ and every }x_k\in X_k.
\end{equation*}
We then say that $\boldsymbol g$ \emph{dominates} $\boldsymbol f$. The domination is strict if at least one of these inequalities is strict at some point of some marginal space. A family $\boldsymbol f\in\mathcal F_c^b$ is \emph{Pareto maximal} if no other family in $\mathcal F_c^b$ strictly dominates it. Equivalently,
\begin{equation*}
 \boldsymbol g\in\mathcal F_c^b,\quad
\boldsymbol f\preceq\boldsymbol g
\quad\Longrightarrow\quad
\boldsymbol g=\boldsymbol f.   
\end{equation*}
Thus no value of any component can be increased while keeping all remaining values unchanged or larger and preserving membership in $\mathcal F_c^b$. Maximality does not mean that $\boldsymbol f$ dominates every feasible family. Several maximal families may coexist and be incomparable: one can have larger values at some points and smaller values at others. For the general notion of maximality in ordered spaces, see \cite[Definition~4.1]{jahn2011vector}. For Pareto optimality in finite-dimensional vector optimization, see \cite[Section~4.7.3]{boyd2004convex}.

\begin{proposition}[Mutually $c$-conjugate families and Pareto maximality] 
\label{dm:proposition:maximality}
 Under the preceding assumptions, a family $\boldsymbol f\in\mathcal F_c^b$ is Pareto maximal if and only if it is mutually $c$-conjugate. 
\end{proposition}

The partial $c$-transform is the largest possible replacement of a component compatible with feasibility when the other components are fixed. The proposition therefore identifies mutual $c$-conjugacy with the absence of any feasible pointwise improvement of this kind. This property does not imply dual optimality: a different feasible family may have a larger value of $J$ through increases at some points and decreases at others, weighted by the marginal measures.

\subsection{Dual optimizers on compact spaces}

Throughout this subsection, $(X_k,d_k)$ are compact metric spaces, $\mu_k\in\mathcal P(X_k)$ are fixed, and $c\in\mathcal C(X)$. Equip $X$ with the product metric $d_\infty(\boldsymbol x,\boldsymbol y) =\max_{1\le k\le K}d_k(x_k,y_k)$, and define
\begin{equation*}
\omega(t) =\sup_{\substack{\boldsymbol x,\boldsymbol y\in X\\ d_\infty(\boldsymbol x,\boldsymbol y)\le t}} |c(\boldsymbol x)-c(\boldsymbol y)|, \qquad t\ge0.
\end{equation*}
Uniform continuity of $c$ gives $\omega(t)\to0$ as $t\downarrow0$. Partial transforms inherit this common modulus. Together with the normalization below, this permits the use of the Arzel\`a--Ascoli theorem.

\begin{lemma}[Modulus of continuity of partial $c$-transforms] 
\label{Lemma: equicontinuity}
 Let $(X_k,d_k)$ be compact metric spaces, let $c\in\mathcal C(X)$, and let $\omega$ be a modulus of continuity of $c$ for $d_\infty$. Fix $k\in\{1,\ldots,K\}$ and functions $f_l:X_l\rightarrow\mathbb R$, $l\ne k$. If $f_k=P_c(\boldsymbol f_{\setminus k})$ is real-valued, then
\begin{equation*}
|f_k(x_k)-f_k(x_k')| \le\omega(d_k(x_k,x_k')), \qquad x_k,x_k'\in X_k.
\end{equation*}
\end{lemma}

Equicontinuity alone does not control additive constants. Indeed, replacing $f_k$ by $f_k+a_k$, where $\sum_{k=1}^K a_k=0$, leaves the sum of the potentials and the dual objective unchanged. It also preserves mutual $c$-conjugacy, since
\begin{equation*}
P_c\bigl((f_l+a_l)_{l\ne k}\bigr) =P_c(\boldsymbol f_{\setminus k})+a_k.
\end{equation*}
For example, adding a constant to one component and subtracting it from another can make their norms arbitrarily large. To fix this ambiguity, choose reference points $x_k^0\in X_k$ and impose $\widetilde f_k(x_k^0)=0$ for $k<K$, compensating the shifts in the last component. This selects a representative within each family of such translations; different normalized optimizers may still exist.

\begin{lemma}[Uniform bounds after normalization] 
\label{uniform-bounds} 
Let $X_1,\ldots,X_K$ be compact metric spaces and let $c:X\rightarrow\mathbb R$ be continuous. Suppose that $(\boldsymbol f^n)_{n\ge1}\subset\mathcal G_c$ is a maximizing sequence, that is,
\begin{equation*}
J(\boldsymbol f^n)\longrightarrow \sup_{\boldsymbol f\in\mathcal F_c}J(\boldsymbol f).
\end{equation*}
There exist constants $a_k^n\in\mathbb R$ with $\sum_{k=1}^K a_k^n=0$ such that $\widetilde f_k^n=f_k^n+a_k^n$ satisfies $\widetilde f_k^n(x_k^0)=0$ for $k<K$. The translated families remain in $\mathcal G_c$, have the same dual values, and satisfy
\begin{equation*}
\|\widetilde f_k^n\|_\infty\le C_k \qquad(k=1,\ldots,K,\ n\ge1),
\end{equation*}
for finite constants $C_k$ independent of $n$. 
\end{lemma}

To obtain dual optimizers, start with a continuous maximizing sequence provided by compact Kantorovich duality and apply Lemma~\ref{Lemma: sequential-conjugation} to each family. The transformed sequence lies in $\mathcal G_c$ and remains maximizing. The two preceding lemmas give equicontinuity and uniform bounds after normalization. Arzel\`a--Ascoli then provides a uniformly convergent subsequence, whose limit preserves mutual $c$-conjugacy and realizes the dual optimum. The details appear in Appendix~\ref{section: Appendix}.

\begin{theorem}[Existence of dual optimizers on compact spaces]
 \label{Theorem: compact c-conjugate}
  Fix compact metric spaces $X_1,\ldots,X_K$, marginals $\mu_k\in\mathcal P(X_k)$ for $k=1,\ldots,K$, and a continuous cost $c:X\rightarrow\mathbb R$. The dual problem admits an optimal family in $\mathcal G_c$. In particular,
\begin{equation*}
\mathsf{MOT}_c(\boldsymbol\mu) =\max_{\boldsymbol f\in\mathcal F_c}J(\boldsymbol f) =\max_{\boldsymbol f\in\mathcal G_c}J(\boldsymbol f).
\end{equation*}
\end{theorem}

\subsection{Dual optimizers from splitting families on Polish spaces}

The compactness argument above is replaced, on Polish spaces, by a construction starting from a splitting family on the support of an optimal plan. Equality on that support will identify the dual value with the primal optimum. The required existence of a real-valued splitting family is an assumption of the result.

Recall the following definition from Griessler~\cite{griessler2016}.

\begin{definition}[$c$-splitting set]
 A set $\Gamma\subset X$ is a \emph{$c$-splitting set} if there are functions $f_k:X_k\rightarrow[-\infty,\infty)$, $k=1,\ldots,K$, such that
\begin{enumerate}[nosep, leftmargin=*]
\item $\sum_{k=1}^K f_k(x_k)\le c(\boldsymbol x)$  for every $\boldsymbol x\in X,$
\item $\sum_{k=1}^K f_k(x_k)=c(\boldsymbol x)$ for every $\boldsymbol x\in\Gamma. $
\end{enumerate}
Such a family is called a \emph{splitting family} for $\Gamma$. No measurability is required in this definition. \end{definition}

This definition permits the value $-\infty$. For the construction below, we assume that the splitting family is real-valued on every marginal space. Lemma~\ref{Lemma: sequential-conjugation} then applies and preserves the splitting equality.

\begin{proposition}[Mutually $c$-conjugate splitting families] 
\label{Proposition: splitting}
 Let $c:X\rightarrow\mathbb R$ be a real-valued cost, and let $\Gamma\subset X$ be a nonempty $c$-splitting set with a real-valued splitting family $\boldsymbol f$. Updating each component once produces a real-valued family $\boldsymbol g\succeq\boldsymbol f$ such that
\begin{enumerate}[nosep, leftmargin=*]
\item $\sum_{k=1}^K g_k(x_k)\le c(\boldsymbol x)\text{ for every }\boldsymbol x\in X,$
\item $\sum_{k=1}^K g_k(x_k)=c(\boldsymbol x) \text{ for every }\boldsymbol x\in\Gamma,$
\end{enumerate}
and $\boldsymbol g$ is mutually $c$-conjugate. 
\end{proposition}
Boundedness of the cost gives a further property of the constructed family. The following statement does not require a splitting set.

\begin{proposition}[Boundedness of mutually $c$-conjugate families] 
\label{Proposition: bounded-splitting}
 Let $X_1,\ldots,X_K$ be nonempty sets and let $c:X\rightarrow\mathbb R$ be bounded. Every real-valued mutually $c$-conjugate family has bounded components. 
\end{proposition}

\begin{remark}[Regularity and dual admissibility]
 \label{Remark: bounded-conjugate-regularity}
  If, in addition, the $X_k$ are topological spaces and $c$ is continuous, Lemma~\ref{Lemma: usc-conjugation} makes these components upper semicontinuous and Borel measurable. Boundedness then implies their integrability for any Borel probability marginals. Thus the family belongs to $\mathcal G_c$ for every such choice of marginals. The bounds here concern each individual family; uniform bounds across families require an appropriate normalization. 
  \end{remark}

These properties turn splitting data on an optimal support into an admissible dual optimizer. The initial family need not be measurable: regularity and integrability are obtained after the successive transforms.

\begin{theorem}[Existence of dual optimizers under a splitting hypothesis] 
\label{Theorem: noncompact c-conjugate}
 Fix Polish spaces $X_1,\ldots,X_K$, marginals $\mu_k\in\mathcal P(X_k)$ for $k=1,\ldots,K$, and a bounded continuous cost $c:X\rightarrow\mathbb R$. Let $\pi^*\in\Pi(\boldsymbol\mu)$ be an optimal plan and set $\Gamma=\operatorname{supp}(\pi^*)$. Assume that $\Gamma$ admits a real-valued splitting family, with no measurability or conjugacy assumption. Then there is an optimal dual family $\boldsymbol g\in\mathcal G_c$ whose components are bounded and upper semicontinuous, hence Borel measurable. Moreover,
\begin{equation*}
\mathsf{MOT}_c(\boldsymbol\mu) =\max_{\boldsymbol f\in\mathcal F_c}J(\boldsymbol f) =\max_{\boldsymbol f\in\mathcal G_c}J(\boldsymbol f).
\end{equation*}
\end{theorem}

\section{A common compact class of normalized dual potentials}
\label{dm:section:conjugacy}

To compare mutually $c$-conjugate dual optimizers as the marginals vary,
we place them in a common compact class. We first carry out this
construction on compact marginal spaces, then give a sufficient condition
for it to remain valid on noncompact Polish spaces. The construction
combines regularity inherited
from the cost with normalization at fixed reference points. Normalization
removes the freedom to add constants whose sum is zero; it does not
imply uniqueness of the optimal family.

For the normalization argument, $(X_k,d_k)$ is a nonempty Polish space
for each $k=1,\ldots,K$, and $c\in\mathcal C_b(X)$, where
$X=\prod_{k=1}^K X_k$. Set $M=\|c\|_\infty$. We use
$J_{\boldsymbol\mu}$, as defined in Section~\ref{Preliminaries},
to display the dependence of the dual objective on the marginals.
For bounded families, we use the product uniform norm $\|\boldsymbol f\|_\infty=\max_{1\le k\le K}\|f_k\|_\infty.$

\subsection{Normalization and uniform bounds}

By Proposition~\ref{Proposition: bounded-splitting} and
Remark~\ref{Remark: bounded-conjugate-regularity}, every real-valued
mutually $c$-conjugate family has bounded upper semicontinuous
components. They are therefore Borel measurable and integrable against
every probability measure on $X_k$.
Consequently, in the present setting, the class $\mathcal G_c$ of
Definition~\ref{definition: Gc space} consists of all real-valued
mutually $c$-conjugate families and is independent of the marginals.

\begin{definition}[Normalized conjugate class]
\label{dm:definition:mutual}
Fix reference points $x_k^\circ\in X_k$, $k=1,\ldots,K-1$,
independently of the marginal measures. The normalized conjugate class is
\[
\mathcal N_c=
\left\{\boldsymbol f\in\mathcal G_c:
 f_k(x_k^\circ)=0,\quad k=1,\ldots,K-1\right\}.
\]
\end{definition}

There are $K-1$ independent constants in a zero-sum additive shift.
The $K-1$ reference-point conditions fix these constants, while the shift
of the last component is determined by their sum. We now show that this
normalization turns the oscillation bounds supplied by conjugacy into
uniform bounds on the components themselves.

For $k=1,\ldots,K$, define the coordinate variation of the cost by
\[
\Delta_k^c(x_k,y_k)=
\sup_{\boldsymbol z_{\setminus k}\in X_{\setminus k}}
\big|c(x_k,\boldsymbol z_{\setminus k})
-c(y_k,\boldsymbol z_{\setminus k})\big|.
\]
For a bounded family $\boldsymbol f$, both partial transforms below
are finite. The elementary inequality
$|\inf A-\inf B|\le\sup|A-B|$ thus gives
\begin{equation}\label{dm:eq:modulus-estimate}
\left|P_c(\boldsymbol f_{\setminus k})(x_k)
-P_c(\boldsymbol f_{\setminus k})(y_k)\right|
\le\Delta_k^c(x_k,y_k)\le2M.
\end{equation}
In particular, every $\boldsymbol f\in\mathcal G_c$ satisfies
\[
\osc(f_k):=
\sup_{x_k\in X_k}f_k(x_k)-\inf_{x_k\in X_k}f_k(x_k)
\le2M,\qquad k=1,\ldots,K.
\]

\begin{lemma}[Zero-sum normalization and uniform bounds]
\label{dm:lemma:normalization}
Let $\boldsymbol f\in\mathcal G_c$. If $a_1,\ldots,a_K$ are constants
with $\sum_{k=1}^K a_k=0$, then $(f_k+a_k)_{k=1}^K\in\mathcal G_c$
and its dual objective equals that of $\boldsymbol f$ for every
probability marginal tuple. Every family in $\mathcal G_c$ can be
normalized in this way into $\mathcal N_c$. Moreover, each
$\boldsymbol g\in\mathcal N_c$ satisfies
\begin{equation}\label{dm:eq:uniform-bounds}
\|g_k\|_\infty\le2M\quad(k<K),\qquad
\|g_K\|_\infty\le(2K-1)M,\qquad
\osc(g_k)\le2M\quad(1\le k\le K).
\end{equation}
In particular, $\|\boldsymbol g\|_\infty\le(2K-1)M$.
The class $\mathcal N_c$ is nonempty.
\end{lemma}

\begin{proof}
For a zero-sum family of constants and each $x_k\in X_k$,
\[
\begin{aligned}
P_c((f_\ell+a_\ell)_{\ell\ne k})(x_k)
&=\inf_{\boldsymbol x_{\setminus k}}
 \Bigl\{c(\boldsymbol x)-\sum_{\ell\ne k}f_\ell(x_\ell)\Bigr\}
 -\sum_{\ell\ne k}a_\ell\\
&=P_c(\boldsymbol f_{\setminus k})(x_k)+a_k
 =f_k(x_k)+a_k.
\end{aligned}
\]
Thus mutual conjugacy is preserved. The sum of the potentials is
unchanged pointwise, and, since $\mu_k(X_k)=1$ for every $k$,
\[
J_{\boldsymbol\mu}((f_k+a_k)_{k=1}^K)
=J_{\boldsymbol\mu}(\boldsymbol f)+\sum_{k=1}^K a_k
=J_{\boldsymbol\mu}(\boldsymbol f).
\]
Choose
\[
a_k=-f_k(x_k^\circ)\quad(k<K),\qquad
 a_K=-\sum_{k=1}^{K-1}a_k.
\]
Then $g_k=f_k+a_k$ satisfies $g_k(x_k^\circ)=0$ for $k<K$,
so $\boldsymbol g\in\mathcal N_c$.

For any $\boldsymbol g\in\mathcal N_c$,
\eqref{dm:eq:modulus-estimate} gives $\osc(g_k)\le2M$.
For $k<K$, the reference-point condition consequently yields
$|g_k(x_k)|=|g_k(x_k)-g_k(x_k^\circ)|\le2M$.
For the last component, conjugacy gives
\[
g_K(x_K)=\inf_{\boldsymbol x_{\setminus K}}
\Bigl\{c(\boldsymbol x)-\sum_{k=1}^{K-1}g_k(x_k)\Bigr\}
\ge-M-2(K-1)M=-(2K-1)M.
\]
Evaluating the same infimum at the fixed reference points gives
\[
g_K(x_K)\le c(x_1^\circ,\ldots,x_{K-1}^\circ,x_K)\le M.
\]
These two inequalities prove the remaining uniform bound.

Finally, the constant family $(-M,0,\ldots,0)$ is real-valued and
feasible. Lemma~\ref{Lemma: sequential-conjugation} turns it into a
real-valued mutually conjugate family. Its components are bounded
and upper semicontinuous by the preceding results. Normalizing this
family gives an element of $\mathcal N_c$.
\end{proof}

Neither normalization nor these uniform bounds requires compact
marginal spaces. Compactness of $\mathcal N_c$ in the uniform norm,
however, requires additional control of the coordinate variations of
the cost. We give two settings in which this control is available.

\subsection{Compact dual representation on compact spaces}

In this subsection, assume that every $X_k$ is compact. Continuity of
$c$ on the compact product then gives a common modulus of continuity
for the conjugate potentials. Together with the preceding uniform
bounds, this yields relative compactness. Continuity of partial
transforms will give closedness.

\begin{lemma}[Compactness of the normalized class]
\label{dm:lemma:compact-class}
If every $X_k$ is compact, the class $\mathcal N_c$ is a nonempty compact subset of
$\prod_{k=1}^K\mathcal C(X_k)$ with its product uniform norm.
\end{lemma}

\begin{proof}
Nonemptiness follows from Lemma~\ref{dm:lemma:normalization}.
The components of every family in $\mathcal N_c$ are continuous by
Lemma~\ref{Lemma: equicontinuity}.

For each $k$, define
\[
\omega_k^c(r)=
\sup_{\substack{x_k,y_k\in X_k,\ d_k(x_k,y_k)\le r\\
\boldsymbol z_{\setminus k}\in X_{\setminus k}}}
\big|c(x_k,\boldsymbol z_{\setminus k})
-c(y_k,\boldsymbol z_{\setminus k})\big|.
\]
Uniform continuity of $c$ implies $\omega_k^c(r)\to0$ as $r\downarrow0$.
By \eqref{dm:eq:modulus-estimate}, every $\boldsymbol f\in\mathcal N_c$
satisfies
\[
|f_k(x_k)-f_k(y_k)|\le\omega_k^c(d_k(x_k,y_k)).
\]
Thus the component families are equicontinuous and, by
Lemma~\ref{dm:lemma:normalization}, uniformly bounded.
The Arzel\`a--Ascoli theorem gives relative compactness in
$\prod_{k=1}^K\mathcal C(X_k)$.

To prove closedness, we use an estimate that also allows the cost to
vary. For bounded costs $c,\widetilde c$ on the same product and bounded
families $\boldsymbol f,\boldsymbol g$,
\begin{equation}\label{dm:eq:transform-stability}
\|P_c(\boldsymbol f_{\setminus k})
-P_{\widetilde c}(\boldsymbol g_{\setminus k})\|_\infty
\le\|c-\widetilde c\|_\infty
+\sum_{\ell\ne k}\|f_\ell-g_\ell\|_\infty.
\end{equation}
Indeed, for fixed $x_k$, the two functions whose infima define the
transforms differ by
\[
(c-\widetilde c)(\boldsymbol x)
-\sum_{\ell\ne k}(f_\ell-g_\ell)(x_\ell),
\]
whose absolute value is bounded by the right-hand side of
\eqref{dm:eq:transform-stability}. Applying the inequality for infima
and then taking the supremum over $x_k$ proves the estimate.

Now let $\boldsymbol f^n\in\mathcal N_c$ converge uniformly to
$\boldsymbol f\in\prod_{k=1}^K\mathcal C(X_k)$.
Using $f_k^n=P_c(\boldsymbol f^n_{\setminus k})$ and
\eqref{dm:eq:transform-stability} with $\widetilde c=c$, we obtain
\[
\begin{aligned}
\|f_k-P_c(\boldsymbol f_{\setminus k})\|_\infty
&\le\|f_k-f_k^n\|_\infty
 +\|P_c(\boldsymbol f^n_{\setminus k})
       -P_c(\boldsymbol f_{\setminus k})\|_\infty\\
&\le\sum_{\ell=1}^K\|f_\ell^n-f_\ell\|_\infty
\longrightarrow0.
\end{aligned}
\]
Hence $f_k=P_c(\boldsymbol f_{\setminus k})$ for every $k$.
The reference-point conditions also pass to the limit:
$f_k(x_k^\circ)=0$ for $k<K$. Therefore
$\boldsymbol f\in\mathcal N_c$. The class is closed and relatively
compact, and thus compact.
\end{proof}

We can now represent all the dual problems on this same compact class.
Only the objective changes when the marginal measures change.

\begin{theorem}[Dual representation on a common compact class]
\label{dm:theorem:compact-attainment}
Let $(X_k,d_k)$, $k=1,\ldots,K$, be nonempty compact metric spaces,
let $c\in\mathcal C(\prod_{k=1}^K X_k)$, and fix
$x_k^\circ\in X_k$ for $k<K$. Then $\mathcal N_c$ is nonempty and
compact in the product uniform norm. It depends only on the cost,
the spaces, and the chosen reference points. For every
$\boldsymbol\mu\in\prod_{k=1}^K\mathcal P(X_k)$,
\[
\mathsf{MOT}_c(\boldsymbol\mu)
=\max_{\boldsymbol f\in\mathcal N_c}J_{\boldsymbol\mu}(\boldsymbol f).
\]
Every family in $\mathcal N_c$ has continuous components and satisfies
\eqref{dm:eq:uniform-bounds} and \eqref{dm:eq:modulus-estimate},
with $M=\|c\|_\infty$.
\end{theorem}

\begin{proof}
For a given marginal tuple, Theorem~\ref{Theorem: compact c-conjugate}
provides $\boldsymbol f^*\in\mathcal G_c$ with
$J_{\boldsymbol\mu}(\boldsymbol f^*)=\mathsf{MOT}_c(\boldsymbol\mu)$.
By Lemma~\ref{dm:lemma:normalization}, a zero-sum shift produces
$\boldsymbol g\in\mathcal N_c$ with the same objective. Since
$\mathcal N_c\subseteq\mathcal F_c$, weak duality gives
\[
J_{\boldsymbol\mu}(\boldsymbol f)
\le\mathsf{MOT}_c(\boldsymbol\mu)
=J_{\boldsymbol\mu}(\boldsymbol g)
\qquad(\boldsymbol f\in\mathcal N_c).
\]
Thus the maximum over $\mathcal N_c$ is the transport value.
Nonemptiness and compactness follow from
Lemma~\ref{dm:lemma:compact-class}; the estimates follow from
Lemma~\ref{dm:lemma:normalization} and
\eqref{dm:eq:modulus-estimate}. Finally, neither the definition of
$\mathcal N_c$ nor these estimates involves the marginal measures.
\end{proof}

\subsection{An extension to noncompact marginal spaces}
\label{dm:subsection:noncompact-class}

Compactness of the original marginal spaces is sufficient for the
preceding representation, but it is not essential. The relevant
condition can instead be imposed on the variation of the cost.
For each $k$, the function $\Delta_k^c$ is a bounded pseudometric:
it satisfies the triangle inequality, but distinct points may have
zero distance if they give the same cost section.

\begin{assumption}[Total boundedness of coordinate cost profiles]
\label{dm:assumption:cost-totally-bounded}
For every $k$ and every $\varepsilon>0$, there are finitely many
points $z_k^1,\ldots,z_k^{m_k}\in X_k$ such that
\[
\forall x_k\in X_k,\qquad
\min_{1\le j\le m_k}\Delta_k^c(x_k,z_k^j)<\varepsilon.
\]
Equivalently, $(X_k,\Delta_k^c)$ is totally bounded for every $k$.
\end{assumption}

The assumption says that finitely many sections
$c(z_k^j,\boldsymbol x_{\setminus k})$ approximate all sections
$c(x_k,\boldsymbol x_{\setminus k})$ uniformly in the remaining
coordinates. It does not require total boundedness for the original
metric $d_k$. If all marginal spaces are compact and $c$ is continuous,
uniform continuity of $c$ implies this assumption automatically.

\begin{proposition}[A compact normalized class on Polish spaces]
\label{dm:proposition:noncompact-class}
Let $X_1,\ldots,X_K$ be nonempty Polish spaces and let
$c\in\mathcal C_b(\prod_k X_k)$ satisfy
Assumption~\ref{dm:assumption:cost-totally-bounded}.
Fix reference points $x_k^\circ\in X_k$ for $k<K$.
Then $\mathcal N_c$ is a nonempty compact subset of
$\prod_{k=1}^K\mathcal C_b(X_k)$ in the product uniform norm.
It depends only on the spaces, the cost, and the reference points.
Every family in $\mathcal N_c$ satisfies
\eqref{dm:eq:uniform-bounds} and \eqref{dm:eq:modulus-estimate}.
For every probability marginal tuple $\boldsymbol\mu$, the representation
\eqref{dm:eq:normalized-representation} also holds:
\begin{equation*}
\mathsf{MOT}_c(\boldsymbol\mu)
=\max_{\boldsymbol f\in\mathcal N_c}
J_{\boldsymbol\mu}(\boldsymbol f).
\end{equation*}
\end{proposition}

The proof is given in Appendix~\ref{dm:appendix:noncompact-class}.
Its main point is that normalized conjugate potentials are uniformly
bounded and satisfy
$|f_k(x_k)-f_k(y_k)|\le\Delta_k^c(x_k,y_k)$.
Finite nets for $\Delta_k^c$ therefore give finite nets for these
potentials in the uniform norm. Continuity in the original topology
also needs justification: total boundedness of the cost profiles and
continuity of $c$ together imply
$\Delta_k^c(x_k^n,x_k)\to0$ whenever $x_k^n\to x_k$.
Finally, conjugating and normalizing a dual maximizing sequence yields
the maximum in \eqref{dm:eq:normalized-representation}.
This argument does not require a splitting family on an optimal support.

\begin{example}[A bounded cost on $\mathbb R^K$]
\label{dm:example:noncompact-features}
Let $X_k=\mathbb R$ for all $k$ and set
\[
c(\boldsymbol x)=\sum_{1\le k<\ell\le K}
\bigl(\arctan x_k-\arctan x_\ell\bigr)^2.
\]
This cost is continuous and bounded by $\binom K2\pi^2$.
Only $K-1$ terms change when the $k$th coordinate changes. Factoring
the difference of two squares gives
\[
\Delta_k^c(x_k,y_k)
\le2\pi(K-1)|\arctan x_k-\arctan y_k|.
\]
The interval $\arctan(\mathbb R)=(-\pi/2,\pi/2)$ is totally bounded.
A finite net in this interval, with centers in the interval, thus
pulls back to a finite $\Delta_k^c$-net in $\mathbb R$.
Proposition~\ref{dm:proposition:noncompact-class} applies, for example
with $x_k^\circ=0$, to arbitrary probability marginals on $\mathbb R$.
No moment condition is needed because the cost is bounded.
\end{example}

This is a sufficient extension for bounded continuous costs. It does
not assert uniform compactness for arbitrary such costs on Polish
spaces, and it does not cover unbounded quadratic costs.

\subsection{Conjugacy and splitting for a quadratic cost}
\label{sec:geometric-illustration}

An explicit three-marginal model illustrates the conjugacy identities,
their contact set, and the effect of normalization. Set
$X_1=X_2=X_3=[0,1]$ and let
$\boldsymbol\theta=(\theta_1,\theta_2,\theta_3)\in(-1,1)^3$
parameterize the marginals. The cost remains fixed:
\begin{equation}\label{eq:illustration-cost}
 c(x_1,x_2,x_3)=\frac12\sum_{1\le k<\ell\le3}(x_k-x_\ell)^2,
 \qquad Q_k^{\boldsymbol\theta}(t)=t+\theta_k t(1-t),
 \quad 0\le t\le1.
\end{equation}
Each map $Q_k^{\boldsymbol\theta}$ is a strictly increasing bijection
of $[0,1]$, since its derivative is at least $1-|\theta_k|>0$ and
it fixes both endpoints. Let $\lambda$ denote Lebesgue probability
measure on $[0,1]$, and set
$\mu_k^{\boldsymbol\theta}=(Q_k^{\boldsymbol\theta})_\#\lambda$.
Equivalently, $Q_k^{\boldsymbol\theta}(U)$ has distribution
$\mu_k^{\boldsymbol\theta}$ when $U$ is uniform on $[0,1]$.

Define
\[
 A(t)=\frac{t^2}{2}-\frac{t^3}{3},\qquad
 B(t)=\frac{t^2}{2}-t^3+\frac{t^4}{2},\qquad
 b_k=3\theta_k-\sum_{\ell=1}^3\theta_\ell,
\]
and specify the potentials by
\[
 f_k^{\boldsymbol\theta}(Q_k^{\boldsymbol\theta}(t))
 =b_k\bigl(A(t)+\theta_k B(t)\bigr),\qquad k=1,2,3.
\]
They all vanish at zero. In particular, the same reference points
$x_1^\circ=x_2^\circ=0$ give the required normalization for every
parameter. The additional equality $f_3^{\boldsymbol\theta}(0)=0$
is a feature of this model, not an extra normalization condition.

The resulting family is globally feasible, with equality precisely on
\begin{equation*}
 \Gamma_{\boldsymbol\theta}
 =\{(Q_1^{\boldsymbol\theta}(t),Q_2^{\boldsymbol\theta}(t),
 Q_3^{\boldsymbol\theta}(t)):0\le t\le1\}.
\end{equation*}
To verify this claim, define, for each ordered pair $k\ne\ell$,
\[
p_{k\ell}(Q_k^{\boldsymbol\theta}(t))
=(\theta_k-\theta_\ell)\bigl(A(t)+\theta_k B(t)\bigr).
\]
The chain rule gives
\[
p_{k\ell}'(x)
=x-Q_\ell^{\boldsymbol\theta}\bigl((Q_k^{\boldsymbol\theta})^{-1}(x)\bigr).
\]
For fixed $y$, the derivative with respect to $x$ of the pair slack
\[
\frac12(x-y)^2-p_{k\ell}(x)-p_{\ell k}(y)
\]
is $Q_\ell^{\boldsymbol\theta}((Q_k^{\boldsymbol\theta})^{-1}(x))-y$.
This expression is strictly increasing in $x$ and vanishes at
$x=Q_k^{\boldsymbol\theta}((Q_\ell^{\boldsymbol\theta})^{-1}(y))$.
The slack therefore has its unique minimum at the matching quantile,
where direct substitution gives zero. Each pair slack is consequently
nonnegative, with equality exactly at matching quantiles.
Since $f_k^{\boldsymbol\theta}=\sum_{\ell\ne k}p_{k\ell}$, summing
over the three unordered pairs proves global feasibility and the
claimed equality set.

Every value of every coordinate occurs on $\Gamma_{\boldsymbol\theta}$.
For a fixed coordinate, feasibility bounds the partial transform from
below by $f_k^{\boldsymbol\theta}$, while evaluation at the matching
quantiles gives the reverse inequality. Hence
$\boldsymbol f^{\boldsymbol\theta}\in\mathcal N_c$.
Figure~\ref{fig:conjugacy-splitting}(a,b) depicts this relation between
the partial-transform identity and the splitting curve.
The common-quantile coupling
\[
\pi^{\boldsymbol\theta}
=(Q_1^{\boldsymbol\theta},Q_2^{\boldsymbol\theta},
Q_3^{\boldsymbol\theta})_\#\lambda
\]
is the law of the triple obtained from the same uniform random variable
$U$. It has the required marginals and is concentrated on the equality
set, so weak duality proves its optimality. Integrating the cost and
using $\int_0^1t^2(1-t)^2\,dt=1/30$ yields
\begin{equation}\label{eq:illustration-value}
 V(\boldsymbol\theta)=\mathsf{MOT}_c(\boldsymbol\mu^{\boldsymbol\theta})
 =\frac1{60}\sum_{1\le k<\ell\le3}(\theta_k-\theta_\ell)^2.
\end{equation}

Figure~\ref{fig:conjugacy-splitting} brings together the four aspects
of the example at $\boldsymbol\theta^0=(-0.60,0.15,0.55)$: mutual
conjugacy in panel~(a), the splitting curve in panel~(b), successive
conjugation in panel~(c), and fixed-reference normalization in panel~(d).
The finite experiment and the normalization shifts are detailed below.

\begin{figure}[!htbp]
\centering
\includegraphics[width=\linewidth]{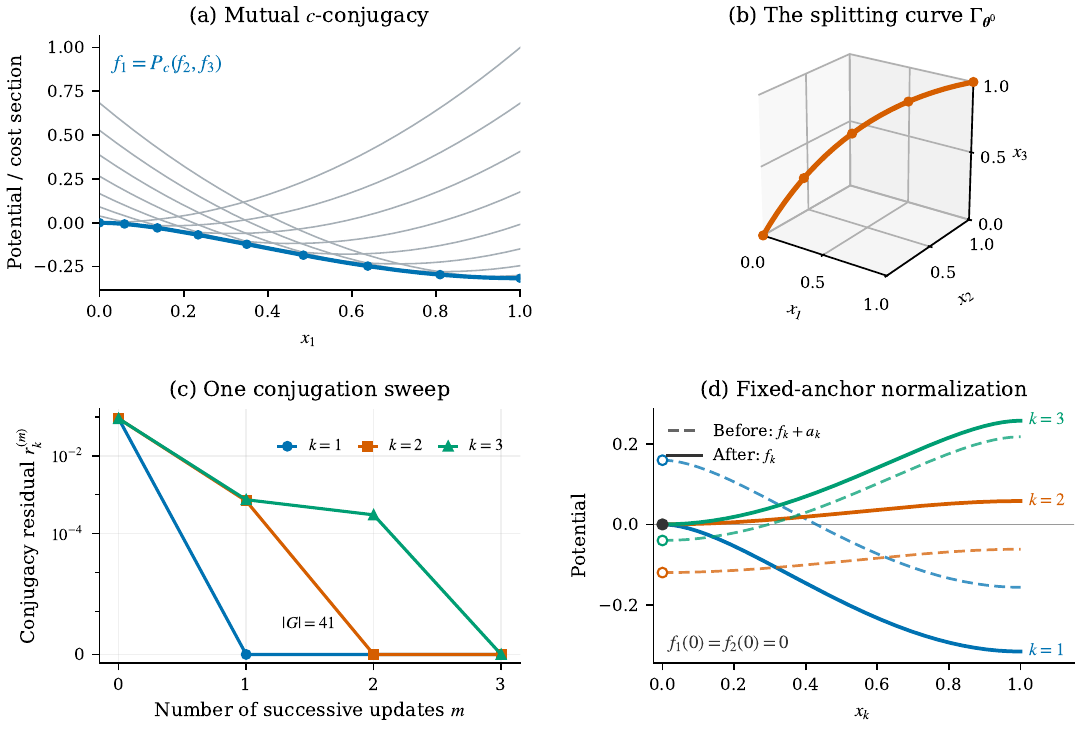}
\caption{Conjugacy, splitting, and normalization for
\eqref{eq:illustration-cost} at $\boldsymbol\theta^0$.
Here $\boldsymbol f=\boldsymbol f^{\boldsymbol\theta^0}$ and
$Q_k=Q_k^{\boldsymbol\theta^0}$.
(a) Gray curves are cost sections
$c(x_1,Q_2(t),Q_3(t))-f_2(Q_2(t))-f_3(Q_3(t))$
at nine quantiles. The blue curve is the analytic infimum over the full
product, not merely the minimum of the nine displayed sections; dots
mark their contact points.
(b) The splitting curve $\Gamma_{\boldsymbol\theta^0}$ supports the
optimal common-quantile coupling.
(c) Conjugacy residuals $r_k^{(m)}$ for the finite sweep on $G^3$,
where $G=\{j/40:j=0,\ldots,40\}$; the residuals are defined in the text.
The vertical scale is linear from zero to $10^{-6}$ and logarithmic
above $10^{-6}$, so numerical zeros and the larger residuals remain
visible on the same axis.
(d) Dashed curves show $q_k=f_k+a_k$, with
$\boldsymbol a=(0.16,-0.12,-0.04)$; solid curves show the normalized
potentials. Only the first two reference-point
conditions are imposed; the third potential also vanishes at zero in
this particular model.}
\label{fig:conjugacy-splitting}
\end{figure}

For the finite conjugation experiment in
Figure~\ref{fig:conjugacy-splitting}(c), use the common grid
$G=\{j/40:j=0,\ldots,40\}$ and start from
\[
h_k(x)=f_k^{\boldsymbol\theta^0}(x)-\eta_k x(1-x),
\qquad \boldsymbol\eta=(0.18,0.12,0.09).
\]
The 41-point grid is an illustrative discretization: it includes both
endpoints and keeps the full product $G^3$ small enough for direct
minimization. The one-pass identity does not depend on this particular
grid size. The initial family is feasible because, on $[0,1]^3$,
\[
c(\boldsymbol x)-\sum_{k=1}^3h_k(x_k)
= c(\boldsymbol x)-\sum_{k=1}^3 f_k^{\boldsymbol\theta^0}(x_k)
  +\sum_{k=1}^3\eta_k x_k(1-x_k)\ge0.
\]
It retains equality at $(0,0,0)$ and $(1,1,1)$.

Write $c_G=c|_{G^3}$ and let $\boldsymbol h^{(m)}$, $m=0,1,2,3$,
denote the families after $m$ successive updates. Each update minimizes
over the full complementary grid and uses the current components,
including those already updated. The plotted residual is
\[
r_k^{(m)}=
\bigl\|P_{c_G}(\boldsymbol h^{(m)}_{\setminus k})
-h_k^{(m)}\bigr\|_{\infty,G},
\qquad \|u\|_{\infty,G}=\max_{x\in G}|u(x)|.
\]
Feasibility and the two contact points are preserved, and the final
residuals are below $10^{-14}$ in the numerical calculation.
These computations illustrate Lemma~\ref{Lemma: sequential-conjugation};
one pass produces mutual conjugacy, but need not maximize the dual
objective for prescribed marginals.

The normalization in Figure~\ref{fig:conjugacy-splitting}(d) starts from
$q_k=f_k^{\boldsymbol\theta^0}+a_k$, with
$\boldsymbol a=(0.16,-0.12,-0.04)$. Since $\sum_k a_k=0$, this family
has the same total potential, contact set, and dual objective as
$\boldsymbol f^{\boldsymbol\theta^0}$. Subtracting $q_1(0)$ and $q_2(0)$
from the first two components and compensating in the third recovers
$\boldsymbol f^{\boldsymbol\theta^0}$. Thus normalization changes the
individual representatives while preserving their sum.

\FloatBarrier


\section{Qualitative stability of values and normalized optimizers}
\label{dm:section:qualitative stability}

The common compact representation from
Section~\ref{dm:section:conjugacy} allows us to compare dual problems
as their data vary. We first establish bounds for the transport value
and upper stability of normalized optimizer sets at a fixed cost.
We then allow the cost to vary and compare the corresponding normalized
conjugate classes. A finite example shows why the resulting upper
stability of optimizer sets need not imply Hausdorff convergence.

Unless explicitly stated otherwise, the marginal spaces $(X_k,d_k)$
are nonempty compact metric spaces, $c\in\mathcal C(X)$, and
$M=\|c\|_\infty$. The normalization points $x_k^\circ$, $k<K$,
remain fixed throughout. Corollary~\ref{dm:corollary:noncompact-stability}
extends the fixed-cost qualitative conclusions to Polish spaces under
Assumption~\ref{dm:assumption:cost-totally-bounded}.

\subsection{Value bounds and fixed-cost upper stability}
\label{dm:subsection:value-stability}

By Theorem~\ref{dm:theorem:compact-attainment}, every marginal tuple
has a dual maximizer in the same compact class $\DMNc$. We write
\[
\DMSc(\DMbmu)
=\DMargmax_{\DMbfv\in\DMNc}J_{\DMbmu}(\DMbfv)
\]
for its set of normalized conjugate optimizers. This set is nonempty
and compact, since $J_{\DMbmu}$ is continuous on $\DMNc$ in the
product uniform norm.

Recall that $\|\DMbfv\|_\infty=\max_k\|f_k\|_\infty$.
For a nonempty set $A$ of potentials, define
\[
\DMdist_\infty(\DMbfv,A)
=\inf_{\DMbg\in A}\|\DMbfv-\DMbg\|_\infty.
\]
For nonempty compact sets $A,B\subset\prod_{k=1}^K\mathcal C(X_k)$,
their one-sided excess is
\[
\exc_\infty(A,B)
=\sup_{\DMbfv\in A}\DMdist_\infty(\DMbfv,B).
\]
Thus a small excess from a perturbed optimizer set to the limiting
optimizer set means that every perturbed optimizer is close to some
limiting optimizer. It does not require every limiting optimizer to
be approximated in this way.

We use the total variation distance
\[
d_{\DMTV}(\mu,\nu)
=\sup_{A\text{ Borel}}|\mu(A)-\nu(A)|
=\frac12\|\mu-\nu\|_{\mathrm{TV}},
\]
where $\|\mu-\nu\|_{\mathrm{TV}}$ is the total variation norm of the
signed measure $\mu-\nu$.

The next proposition gives a supporting inequality for the transport
value as a function of the marginals, together with bounds for cost
and marginal perturbations. These estimates will also enter the
subsequent analysis of optimal potentials.

\begin{proposition}[Supporting inequality and quantitative bounds]
\label{dm:proposition:value-bounds}
Under the standing compact-space assumptions, the following statements hold.
\begin{enumerate}[label=\textup{(\roman*)}]
\item If $\DMbfv\in\DMSc(\DMbmu)$, then, for every probability
marginal tuple $\DMbnu$ on the same spaces,
\begin{equation}\label{dm:eq:subgradient}
\DMMOT(\DMbnu)
\ge\DMMOT(\DMbmu)+\sum_k\int f_k\,\DMdd(\nu_k-\mu_k).
\end{equation}
In particular, $\DMMOT$ is convex as a function of the marginal tuple.
\item For bounded continuous costs $c,\widetilde c$ and probability
marginal tuples $\DMbmu,\DMbnu$,
\begin{align}
|\DMMOT(\DMbmu)-\mathsf{MOT}_{\widetilde c}(\DMbmu)|
&\le\|c-\widetilde c\|_\infty,
\label{dm:eq:cost-lipschitz}\\
|\DMMOT(\DMbmu)-\DMMOT(\DMbnu)|
&\le2M\sum_k d_{\DMTV}(\mu_k,\nu_k).
\label{dm:eq:tv-bound}
\end{align}
\item Suppose, in addition, that $c$ is coordinatewise Lipschitz:
\begin{equation}\label{dm:eq:coordinate-lipschitz}
|c(x_k,\boldsymbol z_{\setminus k})-c(y_k,\boldsymbol z_{\setminus k})|
\le L_k d_k(x_k,y_k)
\quad\text{for every }k,x_k,y_k,\boldsymbol z_{\setminus k}.
\end{equation}
Then
\begin{equation}\label{dm:eq:wasserstein-bound}
|\DMMOT(\DMbmu)-\DMMOT(\DMbnu)|
\le\sum_k L_k W_1(\mu_k,\nu_k),
\end{equation}
where $W_1$ is the 1-Wasserstein distance associated with $d_k$ on
$X_k$. The required first moments are finite because the marginal
spaces are compact.
\end{enumerate}
\end{proposition}

\begin{proof}
\emph{(i) Supporting inequality and convexity.}
Let $\DMbfv\in\DMSc(\DMbmu)$. Its components are bounded and
continuous, and its feasibility inequality holds on all of $X$.
It is therefore admissible for every marginal tuple $\DMbnu$.
Testing the dual problem for $\DMbnu$ with this admissible family
and using linearity of integration gives
\[
\begin{aligned}
\DMMOT(\DMbnu)
&\ge J_{\DMbnu}(\DMbfv)\\
&=J_{\DMbmu}(\DMbfv)
  +\sum_k\int_{X_k}f_k\,\DMdd(\nu_k-\mu_k)\\
&=\DMMOT(\DMbmu)
  +\sum_k\int_{X_k}f_k\,\DMdd(\nu_k-\mu_k),
\end{aligned}
\]
Here the first equality follows from
$\int f_k\,\DMdd\nu_k=\int f_k\,\DMdd\mu_k
+\int f_k\,\DMdd(\nu_k-\mu_k)$ for each $k$;
the last uses the optimality of $\DMbfv$ for $\DMbmu$.
This proves \eqref{dm:eq:subgradient}.

For $t\in[0,1]$, let
$\boldsymbol\mu^t=t\DMbmu+(1-t)\DMbnu$, with the convex combination
taken componentwise. The dual representation applies to
$\boldsymbol\mu^t$ on the same class $\DMNc$, which is independent
of the marginals. For every $\DMbfv\in\DMNc$, linearity gives
$J_{\boldsymbol\mu^t}(\DMbfv)
=tJ_{\DMbmu}(\DMbfv)+(1-t)J_{\DMbnu}(\DMbfv)$. Hence
\[
\begin{aligned}
\DMMOT(\boldsymbol\mu^t)
&=\max_{\DMbfv\in\DMNc}
 \bigl[tJ_{\DMbmu}(\DMbfv)+(1-t)J_{\DMbnu}(\DMbfv)\bigr]\\
&\le t\max_{\DMbfv\in\DMNc}J_{\DMbmu}(\DMbfv)
 +(1-t)\max_{\DMbfv\in\DMNc}J_{\DMbnu}(\DMbfv)\\
&=t\DMMOT(\DMbmu)+(1-t)\DMMOT(\DMbnu).
\end{aligned}
\]
Thus the transport value is convex in the marginal tuple.

\medskip
\noindent\emph{(ii) Cost and total variation bounds.}
Set $\delta=\|c-\widetilde c\|_\infty$ and let $\pi_c$ be an
optimal coupling for $c$ with marginals $\DMbmu$. The same coupling
is admissible for $\widetilde c$, so
\[
\mathsf{MOT}_{\widetilde c}(\DMbmu)
\le\int_X\widetilde c\,d\pi_c
=\int_Xc\,d\pi_c+\int_X(\widetilde c-c)\,d\pi_c
\le\DMMOT(\DMbmu)+\delta.
\]
Interchanging the costs proves \eqref{dm:eq:cost-lipschitz}.

For a bounded Borel function $h$ and probability measures $\mu,\nu$,
\[
\left|\int h\,\DMdd(\mu-\nu)\right|
\le\operatorname{osc}(h)\,d_{\DMTV}(\mu,\nu).
\]
Indeed, for $m=(\sup h+\inf h)/2$, the zero total mass of
$\mu-\nu$ and the bound $\|h-m\|_\infty\le\operatorname{osc}(h)/2$
give
\[
\begin{aligned}
\left|\int h\,\DMdd(\mu-\nu)\right|
&=\left|\int(h-m)\,\DMdd(\mu-\nu)\right|\\
&\le\frac{\operatorname{osc}(h)}{2}\|\mu-\nu\|_{\mathrm{TV}}
=\operatorname{osc}(h)\,d_{\DMTV}(\mu,\nu).
\end{aligned}
\]
Choose $\boldsymbol f^\mu\in\DMSc(\DMbmu)$.
Applying \eqref{dm:eq:subgradient} and then this integral estimate,
we obtain
\[
\begin{aligned}
\DMMOT(\DMbmu)-\DMMOT(\DMbnu)
&\le\sum_k\int_{X_k}f_k^\mu\,\DMdd(\mu_k-\nu_k)\\
&\le\sum_k\operatorname{osc}(f_k^\mu)\,d_{\DMTV}(\mu_k,\nu_k)
\le2M\sum_k d_{\DMTV}(\mu_k,\nu_k),
\end{aligned}
\]
where \eqref{dm:eq:uniform-bounds} supplies the oscillation bound.
Choosing an optimizer at $\DMbnu$ and exchanging the marginal tuples
gives the reverse inequality, proving \eqref{dm:eq:tv-bound}.

\medskip
\noindent\emph{(iii) Wasserstein bound.}
For every $\DMbfv\in\DMNc$, mutual conjugacy and
\eqref{dm:eq:modulus-estimate} imply
\[
\begin{aligned}
|f_k(x_k)-f_k(y_k)|
&=|P_c(\boldsymbol f_{\setminus k})(x_k)
    -P_c(\boldsymbol f_{\setminus k})(y_k)|\\
&\le\Delta_k^c(x_k,y_k)\le L_k d_k(x_k,y_k).
\end{aligned}
\]
Thus $f_k$ is $L_k$-Lipschitz. For any coupling
$q_k\in\Pi(\mu_k,\nu_k)$,
\[
\begin{aligned}
\left|\int_{X_k}f_k\,\DMdd(\mu_k-\nu_k)\right|
&=\left|\int_{X_k\times X_k}
   \bigl(f_k(x_k)-f_k(y_k)\bigr)\,\DMdd q_k(x_k,y_k)\right|\\
&\le L_k\int_{X_k\times X_k}d_k(x_k,y_k)\,\DMdd q_k(x_k,y_k).
\end{aligned}
\]
Taking the infimum over $q_k$ gives
$\left|\int f_k\,\DMdd(\mu_k-\nu_k)\right|
\le L_kW_1(\mu_k,\nu_k)$.
With $\boldsymbol f^\mu\in\DMSc(\DMbmu)$, the supporting
inequality therefore yields
\[
\DMMOT(\DMbmu)-\DMMOT(\DMbnu)
\le\sum_k\int f_k^\mu\,\DMdd(\mu_k-\nu_k)
\le\sum_k L_kW_1(\mu_k,\nu_k).
\]
Exchanging the marginal tuples proves \eqref{dm:eq:wasserstein-bound}.
\end{proof}

We next consider weak convergence of the marginals. The compactness
of $\DMNc$ provides convergent subsequences of normalized optimizers;
continuity of the dual objective identifies their limits.

\begin{theorem}[Upper stability of the normalized optimizer set]
\label{dm:theorem:optimizer-stability}
Let all $X_k$ be nonempty compact metric spaces, let
$c\in\mathcal C(X)$, and keep the normalization points fixed.
Suppose that $\mu_k^n\rightharpoonup\mu_k$ for every $k$, in the sense
of Definition~\ref{SuppRes:Def:WeakConv}. Then
\begin{equation}\label{dm:eq:value-weak-stability}
\DMMOT(\DMbmu^n)\longrightarrow\DMMOT(\DMbmu)
\end{equation}
and
\begin{equation}\label{dm:eq:excess-stability}
\exc_\infty\bigl(\DMSc(\DMbmu^n),\DMSc(\DMbmu)\bigr)
\longrightarrow0.
\end{equation}
Moreover, every sequence $\DMbfv^n\in\DMSc(\DMbmu^n)$ is relatively
compact in $\prod_k\mathcal C(X_k)$ with the product uniform norm,
and each uniform cluster point belongs to $\DMSc(\DMbmu)$.
If $\DMSc(\DMbmu)$ is a singleton, every such sequence converges
uniformly to its unique element.
\end{theorem}

\begin{proof}
By Theorem~\ref{dm:theorem:compact-attainment}, all the optimizer
sets are nonempty subsets of the same compact class $\DMNc$.

\emph{Limits of optimizers.}
Take $\DMbfv^n\in\DMSc(\DMbmu^n)$. Compactness gives a subsequence
$\DMbfv^{n_j}\to\DMbfv\in\DMNc$ in the product uniform norm.
For each coordinate,
\[
\begin{aligned}
\left|\int_{X_k}f_k^{n_j}\,\DMdd\mu_k^{n_j}
       -\int_{X_k}f_k\,\DMdd\mu_k\right|
&\le\|f_k^{n_j}-f_k\|_\infty
 +\left|\int_{X_k}f_k\,\DMdd(\mu_k^{n_j}-\mu_k)\right|\\
&\longrightarrow0.
\end{aligned}
\]
Here the first term vanishes by uniform convergence, and the second
by weak convergence, since $f_k$ is bounded and continuous. Hence
$J_{\DMbmu^{n_j}}(\DMbfv^{n_j})\to J_{\DMbmu}(\DMbfv)$.

For any fixed $\DMbg\in\DMNc$, optimality gives
$J_{\DMbmu^{n_j}}(\DMbfv^{n_j})\ge J_{\DMbmu^{n_j}}(\DMbg)$.
The right-hand side tends to $J_{\DMbmu}(\DMbg)$ by weak convergence.
Passing to the limit shows that
$J_{\DMbmu}(\DMbfv)\ge J_{\DMbmu}(\DMbg)$ for all
$\DMbg\in\DMNc$. Thus $\DMbfv\in\DMSc(\DMbmu)$.
This argument applies to every uniformly convergent subsequence of
optimizers, so every uniform cluster point is optimal at $\DMbmu$.

\medskip
\noindent\emph{Convergence of values.}
Along the preceding subsequence, optimality and the dual representation
give
\[
\DMMOT(\DMbmu^{n_j})
=J_{\DMbmu^{n_j}}(\DMbfv^{n_j})
\longrightarrow J_{\DMbmu}(\DMbfv)
=\DMMOT(\DMbmu).
\]
The same reasoning applies to every subsequence of the marginal
tuples. Every subsequence of transport values therefore has a further
subsequence converging to $\DMMOT(\DMbmu)$, which proves
\eqref{dm:eq:value-weak-stability} for the whole sequence.

\medskip
\noindent\emph{Upper stability.}
If \eqref{dm:eq:excess-stability} failed, there would be
$\varepsilon>0$, a subsequence $(n_j)$, and optimizers
$\DMbfv^{n_j}\in\DMSc(\DMbmu^{n_j})$ such that
\[
\DMdist_\infty\bigl(\DMbfv^{n_j},\DMSc(\DMbmu)\bigr)
\ge\varepsilon.
\]
Such optimizers can be chosen because each optimizer set is compact
and distance to the limiting set is continuous. A further subsequence
converges uniformly to some $\DMbfv\in\DMSc(\DMbmu)$ by the first
part of the proof. But then
\[
\DMdist_\infty\bigl(\DMbfv^{n_j},\DMSc(\DMbmu)\bigr)
\le\|\DMbfv^{n_j}-\DMbfv\|_\infty\longrightarrow0,
\]
a contradiction. Finally, if $\DMSc(\DMbmu)=\{\DMbfv^*\}$,
every choice $\DMbfv^n\in\DMSc(\DMbmu^n)$ satisfies
\[
\|\DMbfv^n-\DMbfv^*\|_\infty
=\DMdist_\infty\bigl(\DMbfv^n,\DMSc(\DMbmu)\bigr)
\le\exc_\infty\bigl(\DMSc(\DMbmu^n),\DMSc(\DMbmu)\bigr)
\longrightarrow0.
\]
\end{proof}

The same mechanism applies on Polish spaces whenever
Proposition~\ref{dm:proposition:noncompact-class} supplies a common
compact class of bounded continuous potentials. The next corollary
records this extension for a fixed cost.

\begin{corollary}[Fixed-cost stability on Polish spaces]
\label{dm:corollary:noncompact-stability}
Let the spaces, cost, and reference points satisfy the hypotheses of
Proposition~\ref{dm:proposition:noncompact-class}. Keep the cost and
reference points fixed, and suppose that
$\mu_k^n\rightharpoonup\mu_k$ for every $k$.
Define $\mathcal S_c(\boldsymbol\nu)
=\operatorname*{argmax}_{\boldsymbol f\in\mathcal N_c}
J_{\boldsymbol\nu}(\boldsymbol f)$ for each marginal tuple
$\boldsymbol\nu$. Then these optimizer sets are nonempty and compact
in $\prod_k\mathcal C_b(X_k)$, and the convergences
\eqref{dm:eq:value-weak-stability} and \eqref{dm:eq:excess-stability}
also hold:
\begin{align*}
\mathsf{MOT}_c(\boldsymbol\mu^n)
&\longrightarrow\mathsf{MOT}_c(\boldsymbol\mu) \text{ and  }
\operatorname{exc}_\infty
\bigl(\mathcal S_c(\boldsymbol\mu^n),
      \mathcal S_c(\boldsymbol\mu)\bigr)
\longrightarrow0.
\end{align*}
Here the excess is defined using the product uniform norm on
$\prod_k\mathcal C_b(X_k)$.
Every sequence $\boldsymbol f^n\in\mathcal S_c(\boldsymbol\mu^n)$
has a uniformly convergent subsequence, and every cluster point is
in $\mathcal S_c(\boldsymbol\mu)$. If the latter set is a singleton,
every such optimizer sequence converges uniformly to its element.
\end{corollary}

\begin{proof}
Proposition~\ref{dm:proposition:noncompact-class} provides the same
nonempty compact class $\mathcal N_c\subset\prod_k\mathcal C_b(X_k)$
for all marginal tuples and the representation
\eqref{dm:eq:normalized-representation}. Since the objective is
continuous in the uniform norm, every optimizer set is nonempty and
compact.

If $\boldsymbol f^n\to\boldsymbol f$ uniformly in this class, then
\[
\left|\int f_k^n\,d\mu_k^n-\int f_k\,d\mu_k\right|
\le\|f_k^n-f_k\|_\infty
 +\left|\int f_k\,d\mu_k^n-\int f_k\,d\mu_k\right|
\longrightarrow0.
\]
The last term tends to zero by weak convergence, because
$f_k\in\mathcal C_b(X_k)$.
Let $V_n=\mathsf{MOT}_c(\boldsymbol\mu^n)$ and
$V=\mathsf{MOT}_c(\boldsymbol\mu)$.
Testing the $n$th problem with a fixed maximizer for the limit problem
gives $\liminf_n V_n\ge V$.
Conversely, choose maximizers $\boldsymbol f^n$ and a subsequence
along which $V_n$ tends to its upper limit. Compactness gives a further
subsequence with $\boldsymbol f^n\to\boldsymbol f\in\mathcal N_c$
uniformly. The preceding estimate gives
$\limsup_n V_n=J_{\boldsymbol\mu}(\boldsymbol f)\le V$.
This proves value convergence. For any uniformly convergent subsequence
of maximizers, the same objective estimate identifies the limiting
objective with $V$. Thus every uniform cluster point is optimal for
the limiting marginals.

If the excess in \eqref{dm:eq:excess-stability} did not
tend to zero, a subsequence of maximizers could be chosen at distance
at least some $\varepsilon>0$ from $\mathcal S_c(\boldsymbol\mu)$.
A further uniformly convergent subsequence would have a limit in that
set, a contradiction. When the limit optimizer set is a singleton,
the excess bounds the distance of every perturbed maximizer to its
unique element, proving the last assertion.
\end{proof}

\subsection{Perturbations of the cost and marginal measures}
\label{dm:subsection:joint-perturbations}

We now return to compact marginal spaces and allow the cost and
marginals to vary simultaneously. In this setting, the normalized
conjugate class also varies with the cost. We first control this
variation in Hausdorff distance and then use the resulting estimate to
prove upper stability of the optimizer sets.

Throughout this subsection, the nonempty compact metric spaces $X_k$
and the reference points $x_k^\circ\in X_k$, $k<K$, remain fixed.
For $a\in\mathcal C(X)$, let $\mathcal N_a$ be the class in
Definition~\ref{dm:definition:mutual} associated with the cost $a$ and
these reference points. By Lemma~\ref{dm:lemma:compact-class}, every
such class is nonempty and compact. For nonempty compact sets
$A,B\subset\prod_{k=1}^K\mathcal C(X_k)$, the Hausdorff distance
induced by the product uniform norm is
\[
d_H^\infty(A,B)
=\max\left\{\exc_\infty(A,B),\exc_\infty(B,A)\right\}.
\]

\begin{theorem}[Lipschitz dependence of the normalized conjugate class
on the cost]
\label{dm:theorem:cost-class-stability}
Suppose that the $X_k$ are nonempty compact metric spaces and that
$c,\widetilde c\in\mathcal C(X)$. With the same fixed normalization
points,
\begin{equation*}
 d_H^\infty(\mathcal N_c,\mathcal N_{\widetilde c})
 \le (2^K-1)\|c-\widetilde c\|_\infty.
\end{equation*}
\end{theorem}

\begin{proof}
Set $\delta=\|c-\widetilde c\|_\infty$ and fix
$\boldsymbol g\in\mathcal N_{\widetilde c}$. The family
$\boldsymbol h^{(0)}=(g_1-\delta,g_2,\ldots,g_K)$ is globally
$c$-feasible, since
\[
\sum_{k=1}^K h_k^{(0)}(x_k)
=\sum_{k=1}^K g_k(x_k)-\delta
\le\widetilde c(\boldsymbol x)-\delta
\le c(\boldsymbol x)
\qquad(\boldsymbol x\in X).
\]
Let $T_k^c$ denote the replacement of component $k$ by its partial
$c$-transform. Starting from $\boldsymbol h^{(0)}$, perform one
successive sweep in the order $1,\ldots,K$:
\[
\boldsymbol h^{(k)}=T_k^c\boldsymbol h^{(k-1)},
\qquad k=1,\ldots,K,
\qquad \boldsymbol u=\boldsymbol h^{(K)}.
\]
Lemma~\ref{Lemma: sequential-conjugation} shows that $\boldsymbol u$
is real-valued, feasible, and mutually $c$-conjugate. Its components
are continuous because the spaces are compact and $c$ is continuous;
thus $\boldsymbol u\in\mathcal G_c$.

We next estimate the changes in the components. Mutual conjugacy of
$\boldsymbol g$ gives
$g_k=P_{\widetilde c}(\boldsymbol g_{\setminus k})$ for every $k$.
At the first update, the other components still agree with those of
$\boldsymbol g$, so the transform estimate
\eqref{dm:eq:transform-stability} yields
$\|u_1-g_1\|_\infty\le\delta$.
When component $k\ge2$ is updated, the components with indices $j<k$
have their final values $u_j$, while those with indices $j>k$ still
equal $g_j$. Applying the same estimate gives
\[
\|u_k-g_k\|_\infty
\le\delta+\sum_{j<k}\|u_j-g_j\|_\infty.
\]
Induction therefore yields
\[
\|u_k-g_k\|_\infty\le2^{k-1}\delta,
\qquad k=1,\ldots,K.
\]

Normalize $\boldsymbol u$ at the fixed reference points by setting
\[
\widehat u_k=u_k-u_k(x_k^\circ)\quad(k<K),\qquad
\widehat u_K=u_K+\sum_{k<K}u_k(x_k^\circ).
\]
Lemma~\ref{dm:lemma:normalization} gives
$\widehat{\boldsymbol u}\in\mathcal N_c$.
Since $g_k(x_k^\circ)=0$ for $k<K$, we have
$|u_k(x_k^\circ)|\le2^{k-1}\delta$, and hence
\[
\begin{aligned}
\|\widehat u_k-g_k\|_\infty
&\le2^k\delta\qquad(k<K),\\
\|\widehat u_K-g_K\|_\infty
&\le2^{K-1}\delta+\sum_{k<K}2^{k-1}\delta
=(2^K-1)\delta.
\end{aligned}
\]
It follows that
$\DMdist_\infty(\boldsymbol g,\mathcal N_c)
\le(2^K-1)\delta$.
Taking the supremum over $\boldsymbol g\in\mathcal N_{\widetilde c}$
bounds $\exc_\infty(\mathcal N_{\widetilde c},\mathcal N_c)$.
Interchanging $c$ and $\widetilde c$ gives the reverse excess bound
and proves~\eqref{dm:eq:hausdorff-cost-class}.
\end{proof}

For $a\in\mathcal C(X)$ and a marginal tuple $\boldsymbol\lambda$,
write
\[
\mathcal S_a(\boldsymbol\lambda)
=\operatorname*{arg\,max}_{\boldsymbol f\in\mathcal N_a}
J_{\boldsymbol\lambda}(\boldsymbol f).
\]
Each optimizer set is nonempty and compact by
Theorem~\ref{dm:theorem:compact-attainment}.
The preceding estimate allows us to compare optimizers for varying
costs with elements of a single compact class. This yields the
following joint stability result.

\begin{theorem}[Joint upper stability of normalized dual optimizers]
\label{dm:theorem:joint-optimizer-stability}
Let $c_n,c\in\mathcal C(X)$ and
$\boldsymbol\mu^n,\boldsymbol\mu\in\prod_k\mathcal P(X_k)$.
Suppose that
\[
\|c_n-c\|_\infty\longrightarrow0,
\qquad
\mu_k^n\rightharpoonup\mu_k
\quad(k=1,\ldots,K).
\]
Then
\begin{equation}\label{dm:eq:joint-value-stability}
\mathsf{MOT}_{c_n}(\boldsymbol\mu^n)
\longrightarrow\mathsf{MOT}_{c}(\boldsymbol\mu)
\end{equation}
and
\begin{equation}\label{dm:eq:joint-upper-stability}
\exc_\infty\bigl(
\mathcal S_{c_n}(\boldsymbol\mu^n),
\mathcal S_c(\boldsymbol\mu)\bigr)\longrightarrow0.
\end{equation}
In particular, every sequence
$\boldsymbol f^n\in\mathcal S_{c_n}(\boldsymbol\mu^n)$
has a uniformly convergent subsequence, and every uniform cluster
point belongs to $\mathcal S_c(\boldsymbol\mu)$. If the latter set
is a singleton, every such sequence converges uniformly to its
unique element.
\end{theorem}

\begin{proof}
Set $\delta_n=\|c_n-c\|_\infty$, so that $\delta_n\to0$.
We first prove convergence of the values. Applying
\eqref{dm:eq:cost-lipschitz} at the fixed marginal tuple
$\boldsymbol\mu^n$ gives
\[
\left|\mathsf{MOT}_{c_n}(\boldsymbol\mu^n)
      -\mathsf{MOT}_c(\boldsymbol\mu^n)\right|
\le\delta_n.
\]
Theorem~\ref{dm:theorem:optimizer-stability}, applied to the fixed
cost $c$, also gives
$\mathsf{MOT}_c(\boldsymbol\mu^n)
\to\mathsf{MOT}_c(\boldsymbol\mu)$.
Consequently,
\[
\left|\mathsf{MOT}_{c_n}(\boldsymbol\mu^n)
      -\mathsf{MOT}_c(\boldsymbol\mu)\right|
\le\delta_n+
\left|\mathsf{MOT}_c(\boldsymbol\mu^n)
      -\mathsf{MOT}_c(\boldsymbol\mu)\right|
\longrightarrow0,
\]
which proves~\eqref{dm:eq:joint-value-stability}.

Now choose an arbitrary sequence
$\boldsymbol f^n\in\mathcal S_{c_n}(\boldsymbol\mu^n)$.
Theorem~\ref{dm:theorem:cost-class-stability} gives
\[
d_H^\infty(\mathcal N_{c_n},\mathcal N_c)
\le(2^K-1)\delta_n.
\]
Since $\mathcal N_c$ is compact, the distance from
$\boldsymbol f^n$ to this class is attained. We can therefore choose
$\boldsymbol h^n\in\mathcal N_c$ such that
\[
\|\boldsymbol f^n-\boldsymbol h^n\|_\infty
\le(2^K-1)\delta_n\longrightarrow0.
\]
A subsequence of $(\boldsymbol h^n)$ converges uniformly to some
$\boldsymbol h\in\mathcal N_c$. Along the same subsequence,
\[
\|\boldsymbol f^{n_j}-\boldsymbol h\|_\infty
\le\|\boldsymbol f^{n_j}-\boldsymbol h^{n_j}\|_\infty
   +\|\boldsymbol h^{n_j}-\boldsymbol h\|_\infty
\longrightarrow0.
\]
Thus every optimizer sequence has a uniformly convergent subsequence
whose limit lies in $\mathcal N_c$.

To show that this limit is optimal, observe that
\[
\begin{aligned}
\big|J_{\boldsymbol\mu^{n_j}}(\boldsymbol f^{n_j})
      -J_{\boldsymbol\mu}(\boldsymbol h)\big|
&\le K\|\boldsymbol f^{n_j}-\boldsymbol h\|_\infty
 +\sum_{k=1}^K
 \bigg|\int_{X_k}h_k\,d(\mu_k^{n_j}-\mu_k)\bigg|
 \longrightarrow0.
\end{aligned}
\]
Here the first term vanishes by uniform convergence, and each
integral tends to zero by weak convergence because $h_k$ is bounded
and continuous. Optimality of $\boldsymbol f^{n_j}$ and the
normalized dual representation give
\[
J_{\boldsymbol\mu^{n_j}}(\boldsymbol f^{n_j})
=\mathsf{MOT}_{c_{n_j}}(\boldsymbol\mu^{n_j}).
\]
Using~\eqref{dm:eq:joint-value-stability}, we obtain
$J_{\boldsymbol\mu}(\boldsymbol h)
=\mathsf{MOT}_c(\boldsymbol\mu)$.
Since $\boldsymbol h\in\mathcal N_c$, this proves
$\boldsymbol h\in\mathcal S_c(\boldsymbol\mu)$.
The same argument applies to every subsequence, so every uniform
cluster point of normalized optimizers is optimal for the limiting
problem.

It remains to prove~\eqref{dm:eq:joint-upper-stability}.
If this convergence failed, there would exist $\varepsilon_0>0$
and a subsequence $(n_j)$ such that
\[
\exc_\infty\bigl(
\mathcal S_{c_{n_j}}(\boldsymbol\mu^{n_j}),
\mathcal S_c(\boldsymbol\mu)\bigr)\ge\varepsilon_0
\qquad\text{for every }j.
\]
The distance to $\mathcal S_c(\boldsymbol\mu)$ is continuous,
and each perturbed optimizer set is compact. Hence we can choose
$\boldsymbol f^{n_j}\in
\mathcal S_{c_{n_j}}(\boldsymbol\mu^{n_j})$ with $\DMdist_\infty\bigl(
\boldsymbol f^{n_j},\mathcal S_c(\boldsymbol\mu)\bigr)
\ge\varepsilon_0.$
The preceding argument provides a further subsequence converging
uniformly to some $\boldsymbol h\in\mathcal S_c(\boldsymbol\mu)$.
But then
\[
\DMdist_\infty\bigl(
\boldsymbol f^{n_j},\mathcal S_c(\boldsymbol\mu)\bigr)
\le\|\boldsymbol f^{n_j}-\boldsymbol h\|_\infty
\longrightarrow0,
\]
a contradiction. This proves the asserted upper stability.

Finally, if $\mathcal S_c(\boldsymbol\mu)=\{\boldsymbol f^*\}$,
then every $\boldsymbol f^n\in
\mathcal S_{c_n}(\boldsymbol\mu^n)$ satisfies
\[
\|\boldsymbol f^n-\boldsymbol f^*\|_\infty
=\DMdist_\infty\bigl(
\boldsymbol f^n,\mathcal S_c(\boldsymbol\mu)\bigr)
\le\exc_\infty\bigl(
\mathcal S_{c_n}(\boldsymbol\mu^n),
\mathcal S_c(\boldsymbol\mu)\bigr)
\longrightarrow0.
\]
\end{proof}

\subsection{A finite example: upper stability without Hausdorff continuity}
\label{dm:subsection:nonuniqueness-example}

The normalized conjugate classes depend Lipschitz continuously on
the cost in Hausdorff distance. The optimizer sets, however, need
not converge in this distance when the marginals vary. The following
three-marginal example shows why the stability conclusions above are
one-sided: a limiting optimizer can remain separated from every
optimizer of the perturbed problems.

\begin{example}[Nonuniqueness and failure of lower stability]
\label{dm:example:nonuniqueness}
Let $K=3$, let $X_1=X_2=X_3=\{0,1\}$ with the discrete metric,
and set
\[
c(x_1,x_2,x_3)=|x_1-x_2|+|x_1-x_3|+|x_2-x_3|.
\]
The cost is zero at $(0,0,0)$ and $(1,1,1)$ and equals $2$ at
every other triple. Choose $x_1^\circ=x_2^\circ=0$, fix
$p\in(0,1)$, and let $\mu_k=\operatorname{Bernoulli}(p)$ for
each $k$. Consider the normalized families
\[
\begin{array}{c|cc|cc|cc}
&f_1(0)&f_1(1)&f_2(0)&f_2(1)&f_3(0)&f_3(1)\\
\hline
\boldsymbol f^-&0&-2&0&1&0&1\\
\boldsymbol f^+&0&2&0&-1&0&-1
\end{array}
\]
Both are globally feasible: their component sums vanish at the two
diagonal triples and are at most $2$ at every other triple.
Equality at both diagonal triples provides a contact point for
every value of every coordinate. Each family is therefore mutually
$c$-conjugate and belongs to $\mathcal N_c$.
The diagonal coupling
$(1-p)\delta_{(0,0,0)}+p\delta_{(1,1,1)}$ has cost zero, and
both dual families have objective zero. Weak duality then gives
$\boldsymbol f^-,\boldsymbol f^+\in
\mathcal S_c(\boldsymbol\mu)$.

For $0<\varepsilon<1-p$, perturb only the first marginal:
\[
\nu_1^\varepsilon=\operatorname{Bernoulli}(p+\varepsilon),
\qquad
\nu_2^\varepsilon=\nu_3^\varepsilon
=\operatorname{Bernoulli}(p).
\]
The coupling
\[
\pi^\varepsilon
=(1-p-\varepsilon)\delta_{(0,0,0)}
+\varepsilon\delta_{(1,0,0)}
+p\delta_{(1,1,1)}
\]
has these marginals and cost $2\varepsilon$.
Since $J_{\boldsymbol\nu^\varepsilon}(\boldsymbol f^+)
=2\varepsilon$, weak duality yields $\mathsf{MOT}_c(\boldsymbol\nu^\varepsilon)=2\varepsilon.$

Take any $\boldsymbol g\in
\mathcal S_c(\boldsymbol\nu^\varepsilon)$.
The slack $c-\sum_k g_k$ is nonnegative and has integral zero
under $\pi^\varepsilon$. All three atoms of this coupling have
positive mass, so equality holds at $(0,0,0)$ and $(1,0,0)$.
Using $g_1(0)=g_2(0)=0$, we obtain $g_3(0)=0, g_1(1)=2.$
Because $f^-_1(1)=-2$, it follows that
$\|\boldsymbol f^--\boldsymbol g\|_\infty\ge4$ for every
perturbed optimizer $\boldsymbol g$. Therefore,
\[
\exc_\infty\bigl(
\mathcal S_c(\boldsymbol\mu),
\mathcal S_c(\boldsymbol\nu^\varepsilon)\bigr)\ge4
\qquad(0<\varepsilon<1-p).
\]
Although $\boldsymbol\nu^\varepsilon\to\boldsymbol\mu$ in
$W_1$ coordinatewise, this excess does not tend to zero.
In the opposite direction,
Theorem~\ref{dm:theorem:optimizer-stability} ensures that
\[
\exc_\infty\bigl(
\mathcal S_c(\boldsymbol\nu^\varepsilon),
\mathcal S_c(\boldsymbol\mu)\bigr)\longrightarrow0.
\]
Thus upper stability cannot generally be strengthened to Hausdorff
continuity, even on finite three-marginal spaces.
\end{example}

\section{Quantitative stability of normalized optimizers}
\label{dm:section:quantitative stability}

Section~\ref{dm:section:qualitative stability} establishes upper
stability of normalized optimizer sets without prescribing a rate.
To obtain quantitative estimates, we first bound the loss of dual
objective when an optimizer for perturbed marginals is evaluated at
a reference problem. A growth condition then converts this objective
gap into a distance from the reference optimizer set. On finite
marginal spaces, we prove a linear growth bound from a polyhedral
decomposition of the normalized conjugate class. Finally, we combine
these estimates with the dependence of that class on the cost.

Throughout this section, the nonempty compact metric spaces and the
normalization points remain fixed. We retain the notation of
Section~\ref{dm:section:qualitative stability}. The estimates are
directed from perturbed optimizers toward the reference optimizer
set; their growth constants may depend on the reference problem.

\subsection{A dual growth modulus and quantitative optimizer stability}

Fix a continuous cost $c$ and a reference marginal tuple $\DMbmu$.
For $r\ge0$, define
\begin{equation}\label{dm:eq:growth-modulus}
\rho_{\DMbmu}(r)
=\DMMOT(\DMbmu)
-\max\Bigl\{J_{\DMbmu}(\DMbfv):\DMbfv\in\DMNc,\ 
 \DMdist_\infty(\DMbfv,\DMSc(\DMbmu))\ge r\Bigr\},
\end{equation}
with $\rho_{\DMbmu}(r)=+\infty$ if the constraint set is empty.
This modulus measures the smallest objective gap among normalized
conjugate families at distance at least $r$ from the optimal set.
Its dependence on the fixed cost is suppressed in the notation.

The constraint sets decrease as $r$ increases, so $\rho_{\DMbmu}$
is nondecreasing, and $\rho_{\DMbmu}(0)=0$ by the dual representation.
For $r>0$, any nonempty constraint set is a compact subset of $\DMNc$
disjoint from $\DMSc(\DMbmu)$. The continuous objective therefore
has a maximum strictly below $\DMMOT(\DMbmu)$ on this set, giving
$\rho_{\DMbmu}(r)>0$.

\begin{theorem}[Growth-modulus stability of optimal potentials]
\label{dm:theorem:growth-stability}
Assume the coordinatewise Lipschitz condition
\eqref{dm:eq:coordinate-lipschitz}. For marginal tuples
$\DMbmu,\DMbnu$, set
\[
\mathsf D(\DMbmu,\DMbnu)
=\sum_{k=1}^K L_kW_1(\mu_k,\nu_k).
\]
Then every $\DMbg\in\DMSc(\DMbnu)$ satisfies
\begin{equation}\label{dm:eq:suboptimality-2D}
0\le\DMMOT(\DMbmu)-J_{\DMbmu}(\DMbg)
\le2\mathsf D(\DMbmu,\DMbnu).
\end{equation}
Consequently, for every $r>0$,
\begin{equation}\label{dm:eq:growth-implication}
2\mathsf D(\DMbmu,\DMbnu)<\rho_{\DMbmu}(r)
\quad\Longrightarrow\quad
\DMdist_\infty\bigl(\DMbg,\DMSc(\DMbmu)\bigr)<r.
\end{equation}
If there exist $\kappa>0$ and $p\ge1$ such that
\begin{equation}\label{dm:eq:power-growth}
\DMMOT(\DMbmu)-J_{\DMbmu}(\DMbfv)
\ge\kappa\,\DMdist_\infty
\bigl(\DMbfv,\DMSc(\DMbmu)\bigr)^p
\qquad(\DMbfv\in\DMNc),
\end{equation}
then
\begin{equation}\label{dm:eq:holder-optimizer}
\DMdist_\infty\bigl(\DMbg,\DMSc(\DMbmu)\bigr)
\le\left(\frac{2\mathsf D(\DMbmu,\DMbnu)}{\kappa}\right)^{1/p}.
\end{equation}
\end{theorem}

\begin{proof}
Since $\DMbg\in\DMSc(\DMbnu)$, it belongs to $\DMNc$ and
$J_{\DMbnu}(\DMbg)=\DMMOT(\DMbnu)$. Its objective gap at
$\DMbmu$ is nonnegative by dual admissibility. Moreover,
Proposition~\ref{dm:proposition:value-bounds} gives
\[
|\DMMOT(\DMbmu)-\DMMOT(\DMbnu)|
\le\mathsf D(\DMbmu,\DMbnu).
\]
Every component $g_k$ is $L_k$-Lipschitz, so the same coupling
estimate used in that proposition yields
\[
|J_{\DMbmu}(\DMbg)-J_{\DMbnu}(\DMbg)|
\le\mathsf D(\DMbmu,\DMbnu).
\]
Combining the two estimates, we obtain
\[
\begin{aligned}
\DMMOT(\DMbmu)-J_{\DMbmu}(\DMbg)
&=\DMMOT(\DMbmu)-\DMMOT(\DMbnu)
 +J_{\DMbnu}(\DMbg)-J_{\DMbmu}(\DMbg)\\
&\le2\mathsf D(\DMbmu,\DMbnu),
\end{aligned}
\]
which proves \eqref{dm:eq:suboptimality-2D}.

If $\DMdist_\infty(\DMbg,\DMSc(\DMbmu))\ge r$, then
$\DMbg$ belongs to the constraint set in
\eqref{dm:eq:growth-modulus}, and hence
\[
\DMMOT(\DMbmu)-J_{\DMbmu}(\DMbg)\ge\rho_{\DMbmu}(r).
\]
This contradicts \eqref{dm:eq:suboptimality-2D} whenever
$2\mathsf D(\DMbmu,\DMbnu)<\rho_{\DMbmu}(r)$, proving
\eqref{dm:eq:growth-implication}. If the constraint set is empty,
the distance conclusion holds automatically.

Finally, under \eqref{dm:eq:power-growth}, applying the growth
bound to $\DMbg\in\DMNc$ gives
\[
\kappa\,\DMdist_\infty\bigl(\DMbg,\DMSc(\DMbmu)\bigr)^p
\le\DMMOT(\DMbmu)-J_{\DMbmu}(\DMbg)
\le2\mathsf D(\DMbmu,\DMbnu).
\]
Taking the $p$th root after division by $\kappa$ proves
\eqref{dm:eq:holder-optimizer}.
\end{proof}

\subsection{Automatic linear growth on finite marginal spaces}
\label{dm:subsection:finite-growth}

The power-growth condition~\eqref{dm:eq:power-growth} requires a
separate assumption on general compact metric spaces. When the marginal
spaces are finite, the normalized conjugate class is a finite union of
compact polytopes. This structure yields linear growth at every fixed
reference problem, including when the marginals do not have full support
or the dual optimizer is not unique.

\begin{theorem}[Finite-space polyhedral growth]
\label{dm:theorem:finite-growth}
Suppose that each $X_k$ is a nonempty finite metric space and that
$c:X\to\mathbb R$ is any cost. Fix
$\boldsymbol\mu\in\prod_k\mathcal P(X_k)$ and the reference points
of Definition~\ref{dm:definition:mutual}. Then there exists a constant
$\kappa_{c,\boldsymbol\mu}>0$ such that
\begin{equation}\label{dm:eq:finite-linear-growth}
\mathsf{MOT}_{c}(\boldsymbol\mu)
-J_{\boldsymbol\mu}(\boldsymbol f)
\ge\kappa_{c,\boldsymbol\mu}
\DMdist_\infty\bigl(
\boldsymbol f,\mathcal S_c(\boldsymbol\mu)\bigr),
\qquad\boldsymbol f\in\mathcal N_c.
\end{equation}
No uniqueness or full-support assumption on the marginals is needed.
\end{theorem}

\begin{proof}
Identify $\prod_k\mathcal C(X_k)$ with $\mathbb R^d$, where
$d=\sum_k|X_k|$, equipped with the product supremum norm. A family
$\boldsymbol f$ is $c$-feasible if and only if $\sum_{k=1}^K f_k(x_k)\le c(\boldsymbol x)$ for all $\boldsymbol x\in X.$

For a feasible family, the conjugacy identity
$f_k(x_k)=P_c(\boldsymbol f_{\setminus k})(x_k)$ holds precisely
when equality is attained in~\eqref{eq:dual-feasibility}
at some tuple with $k$-th coordinate $x_k$. This follows because the
infimum defining the partial transform is a minimum over a finite set.

For each pair $(k,x_k)$, select one tuple at which this equality is to
hold. There are finitely many possible collections of such witnesses;
index them by $\sigma$. Let $P_\sigma$ consist of the families satisfying
the feasibility inequalities~\eqref{eq:dual-feasibility},
the selected equalities, and the normalization equations
$f_k(x_k^\circ)=0$ for $k<K$. Every member of $P_\sigma$ is therefore
normalized and mutually $c$-conjugate. Conversely, every member of
$\mathcal N_c$ belongs to at least one such set. Hence
\[
\mathcal N_c=\bigcup_{\sigma:\,P_\sigma\ne\varnothing}P_\sigma.
\]
Each $P_\sigma$ is a polyhedron, and the uniform bounds in
Lemma~\ref{dm:lemma:normalization} make every nonempty $P_\sigma$
a compact polytope.

Let $\mathcal V$ denote the finite union of the vertex sets of these
polytopes, and write $V=\mathsf{MOT}_{c}(\boldsymbol\mu)$. The compact
dual representation and the linearity of $J_{\boldsymbol\mu}$ give
\[
V=\max_{\boldsymbol f\in\mathcal N_c}J_{\boldsymbol\mu}(\boldsymbol f)
 =\max_{\boldsymbol v\in\mathcal V}J_{\boldsymbol\mu}(\boldsymbol v).
\]
If every vertex in $\mathcal V$ is optimal, every point of
$\mathcal N_c$ is optimal as well. In this case, both sides
of~\eqref{dm:eq:finite-linear-growth} vanish, and any positive constant
$\kappa_{c,\boldsymbol\mu}$ satisfies the claim.

Otherwise, set
\[
\gamma=\min_{\substack{\boldsymbol v\in\mathcal V\\
J_{\boldsymbol\mu}(\boldsymbol v)<V}}
\bigl[V-J_{\boldsymbol\mu}(\boldsymbol v)\bigr]>0,
\qquad
D_c=\operatorname{diam}_\infty(\mathcal N_c)>0.
\]
The diameter is positive because $\mathcal N_c$ contains both an
optimal and a nonoptimal vertex. We shall prove the estimate with
$\kappa_{c,\boldsymbol\mu}=\gamma/D_c$.

Fix $\boldsymbol f\in P_\sigma$ and write it as a convex combination
of vertices of this polytope:
\[
\boldsymbol f=\sum_{\ell=1}^m\lambda_\ell\boldsymbol v^\ell,
\qquad
\lambda_\ell\ge0,\qquad \sum_{\ell=1}^m\lambda_\ell=1.
\]
Denote the total weight of the nonoptimal vertices by
\[
t=\sum_{\ell:\,J_{\boldsymbol\mu}(\boldsymbol v^\ell)<V}
\lambda_\ell.
\]
Linearity then yields
\begin{equation}\label{dm:eq:finite-vertex-gap}
V-J_{\boldsymbol\mu}(\boldsymbol f)
=\sum_{\ell=1}^m\lambda_\ell
\bigl[V-J_{\boldsymbol\mu}(\boldsymbol v^\ell)\bigr]
\ge\gamma t.
\end{equation}

Suppose first that $P_\sigma$ contains an optimal vertex
$\boldsymbol v^*$. Replace every nonoptimal vertex in the above convex
combination by $\boldsymbol v^*$. The resulting point
$\boldsymbol f^*\in P_\sigma$ is optimal and satisfies
\[
\|\boldsymbol f-\boldsymbol f^*\|_\infty
\le
\sum_{\ell:\,J_{\boldsymbol\mu}(\boldsymbol v^\ell)<V}
\lambda_\ell\|\boldsymbol v^\ell-\boldsymbol v^*\|_\infty
\le D_c t.
\]
Thus $\DMdist_\infty(\boldsymbol f,\mathcal S_c(\boldsymbol\mu))
\le D_c t$, and~\eqref{dm:eq:finite-vertex-gap} proves the desired
estimate.

If $P_\sigma$ has no optimal vertex, then $t=1$, so
$V-J_{\boldsymbol\mu}(\boldsymbol f)\ge\gamma$.
Since $\mathcal S_c(\boldsymbol\mu)$ is nonempty and contained in
$\mathcal N_c$, we also have
$\DMdist_\infty(\boldsymbol f,\mathcal S_c(\boldsymbol\mu))\le D_c$.
This proves~\eqref{dm:eq:finite-linear-growth} with the same constant
$\gamma/D_c$ and completes the proof.
\end{proof}

Every cost on a product of finite metric spaces is coordinatewise
Lipschitz. Combining this observation with the preceding theorem gives
a Lipschitz bound for the one-sided excess of the optimizer sets, with
a constant determined by the reference problem.

\begin{corollary}[Finite-space Lipschitz stability at a reference problem]
\label{dm:corollary:finite-lipschitz}
Under the hypotheses of Theorem~\ref{dm:theorem:finite-growth}, define
\[
L_k=\sup_{\substack{x_k\ne y_k\\
\boldsymbol z_{\setminus k}\in X_{\setminus k}}}
\frac{|c(x_k,\boldsymbol z_{\setminus k})
-c(y_k,\boldsymbol z_{\setminus k})|}{d_k(x_k,y_k)},
\]
with $L_k=0$ when $X_k$ is a singleton. Then $L_k<\infty$ and, for every
$\boldsymbol\nu\in\prod_k\mathcal P(X_k)$,
\[
\exc_\infty\bigl(
\mathcal S_c(\boldsymbol\nu),
\mathcal S_c(\boldsymbol\mu)\bigr)
\le\frac{2}{\kappa_{c,\boldsymbol\mu}}
\sum_{k=1}^K L_k W_1(\mu_k,\nu_k).
\]
\end{corollary}

\begin{proof}
For a nonsingleton $X_k$, the supremum defining $L_k$ is taken over
finitely many ratios with positive denominators. Hence $L_k$ is finite,
and the coordinatewise Lipschitz condition
\eqref{dm:eq:coordinate-lipschitz} holds with these constants.
Theorem~\ref{dm:theorem:finite-growth} supplies the power-growth
condition~\eqref{dm:eq:power-growth} with $p=1$ and
$\kappa=\kappa_{c,\boldsymbol\mu}$.

For any $\boldsymbol g\in\mathcal S_c(\boldsymbol\nu)$,
Theorem~\ref{dm:theorem:growth-stability} therefore gives
\[
\DMdist_\infty\bigl(\boldsymbol g,\mathcal S_c(\boldsymbol\mu)\bigr)
\le\frac{2}{\kappa_{c,\boldsymbol\mu}}
\mathsf D(\boldsymbol\mu,\boldsymbol\nu)
=\frac{2}{\kappa_{c,\boldsymbol\mu}}
\sum_{k=1}^K L_kW_1(\mu_k,\nu_k).
\]
Taking the supremum over
$\boldsymbol g\in\mathcal S_c(\boldsymbol\nu)$ proves the stated
one-sided excess estimate.
\end{proof}

\subsection{Quantitative stability under simultaneous perturbations}
\label{dm:subsection:quant-joint}

When the cost also changes, a perturbed optimizer need not belong to
the reference class $\mathcal N_c$. The estimate from
Theorem~\ref{dm:theorem:cost-class-stability} first approximates it
by an element of $\mathcal N_c$. We then control the objective gap
of this approximating family and apply the growth condition at the
reference problem.

For a reference pair $(c,\boldsymbol\mu)$ and a perturbed pair
$(\widetilde c,\boldsymbol\nu)$, write
\[
\delta=\|c-\widetilde c\|_\infty,
\qquad
D=\sum_{k=1}^K L_kW_1(\mu_k,\nu_k),
\qquad
H_K=2^K-1,
\]
where the constants $L_k$ concern the reference cost $c$.

\begin{theorem}[Quantitative joint stability of normalized optimizers]
\label{dm:theorem:quant-joint}
Let the $X_k$ be nonempty compact metric spaces, let
$c,\widetilde c\in\mathcal C(X)$, and fix the same normalization
points for both costs. Suppose that $c$ satisfies
\eqref{dm:eq:coordinate-lipschitz} with constants $L_k$.
Let $\boldsymbol\mu,\boldsymbol\nu$ be marginal tuples, and assume
the power-growth condition \eqref{dm:eq:power-growth} at
$(c,\boldsymbol\mu)$ with constants $\kappa>0$ and $p\ge1$.
Then
\[
\exc_\infty\bigl(
\mathcal S_{\widetilde c}(\boldsymbol\nu),
\mathcal S_c(\boldsymbol\mu)\bigr)
\le H_K\delta+
\left(\frac{2D+(1+KH_K)\delta}{\kappa}\right)^{1/p}.
\]
\end{theorem}

\begin{proof}
Fix $\boldsymbol g\in\mathcal S_{\widetilde c}(\boldsymbol\nu)$.
Theorem~\ref{dm:theorem:cost-class-stability} and compactness of
$\mathcal N_c$ provide $\boldsymbol h\in\mathcal N_c$ such that
\begin{equation}\label{dm:eq:joint-h-approx}
\|\boldsymbol g-\boldsymbol h\|_\infty\le H_K\delta.
\end{equation}
Because $\boldsymbol h$ is mutually $c$-conjugate,
\eqref{dm:eq:modulus-estimate} and the coordinatewise Lipschitz
condition on $c$ imply that $h_k$ is $L_k$-Lipschitz. Consequently,
\[
|J_{\boldsymbol\nu}(\boldsymbol h)
 -J_{\boldsymbol\mu}(\boldsymbol h)|\le D.
\]
The uniform approximation also gives
\[
|J_{\boldsymbol\nu}(\boldsymbol g)
 -J_{\boldsymbol\nu}(\boldsymbol h)|\le KH_K\delta.
\]
Finally, the value bounds \eqref{dm:eq:cost-lipschitz} and
\eqref{dm:eq:wasserstein-bound}, applied using the reference cost
$c$, yield
\[
|\mathsf{MOT}_c(\boldsymbol\mu)
 -\mathsf{MOT}_{\widetilde c}(\boldsymbol\nu)|\le D+\delta.
\]
Thus only the reference cost needs the stated Lipschitz regularity.

Since $J_{\boldsymbol\nu}(\boldsymbol g)
=\mathsf{MOT}_{\widetilde c}(\boldsymbol\nu)$, these estimates give
\[
\begin{aligned}
0&\le\mathsf{MOT}_c(\boldsymbol\mu)
       -J_{\boldsymbol\mu}(\boldsymbol h)\\
&\le|\mathsf{MOT}_c(\boldsymbol\mu)
      -\mathsf{MOT}_{\widetilde c}(\boldsymbol\nu)|
 +|J_{\boldsymbol\nu}(\boldsymbol g)
      -J_{\boldsymbol\nu}(\boldsymbol h)| +|J_{\boldsymbol\nu}(\boldsymbol h)
      -J_{\boldsymbol\mu}(\boldsymbol h)|\\
&\le2D+(1+KH_K)\delta.
\end{aligned}
\]
Applying the power-growth condition to $\boldsymbol h\in\mathcal N_c$
therefore yields
\[
\DMdist_\infty\bigl(\boldsymbol h,\mathcal S_c(\boldsymbol\mu)\bigr)
\le\left(\frac{2D+(1+KH_K)\delta}{\kappa}\right)^{1/p}.
\]
Together with \eqref{dm:eq:joint-h-approx}, this implies
\[
\begin{aligned}
\DMdist_\infty\bigl(\boldsymbol g,\mathcal S_c(\boldsymbol\mu)\bigr)
&\le\|\boldsymbol g-\boldsymbol h\|_\infty
 +\DMdist_\infty\bigl(\boldsymbol h,\mathcal S_c(\boldsymbol\mu)\bigr)\\
&\le H_K\delta+
 \left(\frac{2D+(1+KH_K)\delta}{\kappa}\right)^{1/p}.
\end{aligned}
\]
Taking the supremum over
$\boldsymbol g\in\mathcal S_{\widetilde c}(\boldsymbol\nu)$
proves the assertion.
\end{proof}

\begin{corollary}[Joint Lipschitz stability on finite spaces]
\label{dm:corollary:finite-joint}
Suppose that all $X_k$ are nonempty finite metric spaces.
For every fixed reference pair $(c,\boldsymbol\mu)$, there exists
$\kappa_{c,\boldsymbol\mu}>0$ such that, for every cost
$\widetilde c:X\to\mathbb R$ and marginal tuple $\boldsymbol\nu$,
\[
\begin{aligned}
\exc_\infty\bigl(
\mathcal S_{\widetilde c}(\boldsymbol\nu),
\mathcal S_c(\boldsymbol\mu)\bigr)
&\le H_K\|c-\widetilde c\|_\infty +
\frac{2\sum_k L_kW_1(\mu_k,\nu_k)
 +(1+KH_K)\|c-\widetilde c\|_\infty}
 {\kappa_{c,\boldsymbol\mu}},
\end{aligned}
\]
where $L_k$ are the coordinatewise Lipschitz constants of the
reference cost $c$.
\end{corollary}

\begin{proof}
Theorem~\ref{dm:theorem:finite-growth} supplies
\eqref{dm:eq:power-growth} with $p=1$ and
$\kappa=\kappa_{c,\boldsymbol\mu}$. Coordinatewise Lipschitz
regularity is automatic on finite metric spaces, so
Theorem~\ref{dm:theorem:quant-joint} applies and gives the stated bound.
\end{proof}

\subsection{Illustration of value and objective-gap bounds}
\label{sec:stability-illustration}

We illustrate the value and objective-gap estimates with the
three-marginal model of Section~\ref{sec:geometric-illustration}.
Keep the spaces, cost, and normalization points fixed, and perturb
the marginal parameters along
\[
\boldsymbol\theta^\varepsilon
=\boldsymbol\theta^0+\varepsilon(0.55,-0.35,0.20),
\qquad 0\le\varepsilon\le0.5.
\]
Write $\boldsymbol\mu^\varepsilon
=\boldsymbol\mu^{\boldsymbol\theta^\varepsilon}$,
$\boldsymbol f^\varepsilon
=\boldsymbol f^{\boldsymbol\theta^\varepsilon}$, and
$V_\varepsilon=V(\boldsymbol\theta^\varepsilon)$ for the explicit
marginals, normalized optimizers, and transport values.
The cost has coordinatewise Lipschitz constants $L_k=2$.
Since the quantile perturbation is
$(\theta_k^\varepsilon-\theta_k^0)t(1-t)$, the one-dimensional
quantile formula gives
\[
W_1(\mu_k^\varepsilon,\mu_k^0)
=|\theta_k^\varepsilon-\theta_k^0|
  \int_0^1t(1-t)\,dt
=\frac{|\theta_k^\varepsilon-\theta_k^0|}{6}.
\]
Consequently,
\[
D_\varepsilon
=\sum_k L_kW_1(\mu_k^\varepsilon,\mu_k^0)
=\frac{11}{30}\varepsilon.
\]

Figure~\ref{fig:parametric-stability}(a,b) compares
$|V_\varepsilon-V_0|$ with $D_\varepsilon$ and the reference
objective gap $V_0-J_{\boldsymbol\mu^0}(\boldsymbol f^\varepsilon)$
with $2D_\varepsilon$. The two bounds follow from
\eqref{dm:eq:wasserstein-bound} and
\eqref{dm:eq:suboptimality-2D}, respectively.
Panel~(c) records the maximum difference between the explicit
normalized potentials on a common uniform grid of 4001 points.
This comparison concerns the selected representatives; it does not
compute the excess between the complete optimizer sets.
The curves therefore illustrate the analytical estimates without
establishing a general convergence rate for optimal potentials.

\begin{figure}[!htbp]
\centering
\includegraphics[width=\linewidth]{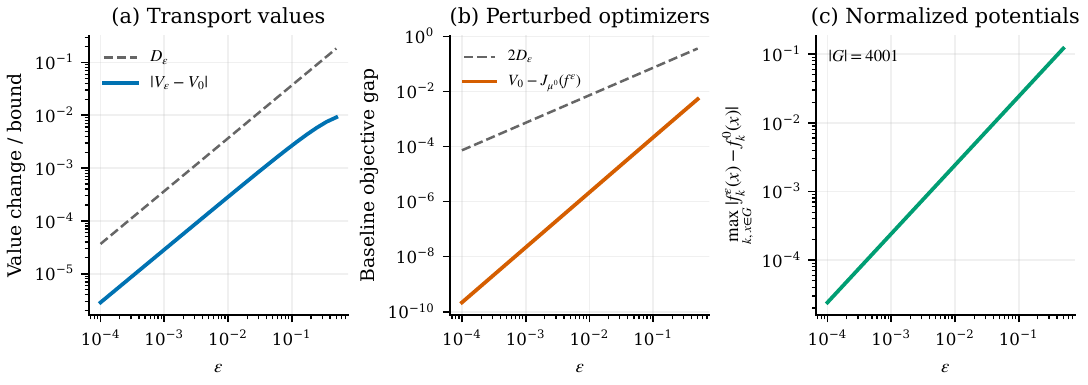}
\caption{Marginal perturbations with fixed spaces, cost, and anchors.
(a) The value change $|V_\varepsilon-V_0|$ and its upper bound
$D_\varepsilon$ from \eqref{dm:eq:wasserstein-bound}.
(b) The reference objective gap
$V_0-J_{\boldsymbol\mu^0}(\boldsymbol f^\varepsilon)$ and its
upper bound $2D_\varepsilon$ from
\eqref{dm:eq:suboptimality-2D}; the objective integrals are evaluated
by numerical quadrature of the explicit potentials.
(c) The maximum difference between the selected normalized
potentials on a fixed uniform grid $G$ of 4001 points.
The continuous values are given by \eqref{eq:illustration-value};
the discretized POT checks are reported separately in the
reproduction package.}
\label{fig:parametric-stability}
\end{figure}

The figure uses 25 logarithmically spaced positive values of
$\varepsilon$ between $10^{-4}$ and $0.5$.
For a separate numerical check, we discretize each marginal at
$n=9,17,33,65,129$ midpoint quantiles and use
$\varepsilon=0,0.1,0.5$.
POT~\cite{flamary2021pot} computes the three pairwise quadratic
transport costs and the marginal $W_1$ distances.
For this particular cost, half the sum of the pairwise optimal
costs is a lower bound for the three-marginal value. The
common-quantile coupling attains all three pairwise optima and
therefore attains this lower bound. Three full $9^3$-variable
linear programs, solved with SciPy/HiGHS, provide an additional check.

All recorded discrete primal--dual and solver comparisons agree
within $10^{-12}$. The quadrature discrepancy between the
discretized value and the continuous value is recorded separately.
The ancillary reproduction files accompanying this preprint
provide the formulas, scripts, raw tables, and software versions.
The reproduction code was prepared with assistance from ChatGPT and
Codex (OpenAI); the numerical checks are described above.
These computations concern unregularized transport and illustrate
the estimates for this explicit family.


\section{Related works}
\label{sec:related}

\paragraph{Duality and splitting potentials.}
Theorems~\ref{theorem: duality-compact} and
\ref{theorem: duality-noncompact} recall classical duality within
Kellerer's framework for marginal problems~\cite{kellerer}.
Broader two-marginal results for Borel costs and relaxed primal
formulations appear in~\cite{beiglbock2012general}; generalized dual
attainment with finitely additive objects is treated
in~\cite{beiglbock2012generalized}.

Griessler~\cite{griessler2016} connects multimarginal cyclical
monotonicity, splitting sets, and optimality, using successive partial
$c$-transforms to construct splitting potentials.
Lemma~\ref{Lemma: sequential-conjugation} records the pointwise
properties needed here: finiteness, feasibility, mutual conjugacy after
one sweep, and preservation of contact points. Bartz, Bauschke and
Wang~\cite{bartz2018} give explicit constructions for special pairwise
costs under conditions on two-coordinate projections. Our construction
starts from a feasible or splitting family; it does not infer splitting
from geometric conditions. In particular,
Theorem~\ref{Theorem: noncompact c-conjugate} assumes a real-valued
splitting family on an optimal support.

\paragraph{Compact normalization and set-valued stability.}
The common normalized class $\mathcal N_c$ makes the compact
maximum-theorem principle~\cite{aliprantis2006,berge1963topological}
applicable to all marginal tuples. The representation is the central
object; value continuity and upper stability then follow from this
classical principle. Theorem~\ref{dm:theorem:cost-class-stability}
further gives an explicit Hausdorff bound between the classes for
different costs, with the same reference points. This yields joint
upper stability under cost and marginal perturbations, whereas
Example~\ref{dm:example:nonuniqueness} rules out general lower stability.
On Polish spaces, total boundedness of the coordinate cost profiles
provides the same compact representation and fixed-cost stability
(Proposition~\ref{dm:proposition:noncompact-class} and
Corollary~\ref{dm:corollary:noncompact-stability}), without a
support-splitting assumption.

\paragraph{Quantitative perturbations and growth estimates.}
The value bounds use dual feasibility and regularity inherited from
the cost. For potentials, a dual growth modulus converts objective
gaps into distances from the reference optimizer set, following
error-bound principles of variational
analysis~\cite{rockafellar1998variational}. Power growth gives explicit
H\"older rates. On finite spaces, a vertex decomposition of
$\mathcal N_c$ proves linear growth, a mechanism related to classical
polyhedral error bounds~\cite{hoffman1952approximate}. The resulting
one-sided Lipschitz constants may depend on the reference cost and
marginals.

Other quantitative transport results concern different objects or
assumptions. Delalande and M\'erigot~\cite{delalande2023quantitative}
treat Brenier potentials and quadratic transport maps under assumptions
on a fixed source; Gallou\"et, M\'erigot and
Thibert~\cite{gallouet2026strong} treat maps and plans under strong
$c$-concavity of Kantorovich potentials; Han and
Zhu~\cite{han2026twomode} consider multimarginal maps for a barycentric
quadratic cost. Here the estimates concern entire sets of normalized
dual potentials, allow nonunique optimizers, and require no special
Euclidean geometry for qualitative upper stability.


\section{Conclusion}\label{section: conclusion}

The common compact class $\mathcal N_c$ represents every marginal dual
problem by normalized mutually $c$-conjugate potentials. On fixed compact
metric spaces, it gives continuous dual maximizers, value continuity,
and upper stability without uniqueness. With fixed reference points,
the classes vary Lipschitz continuously in Hausdorff distance with the
cost, enabling joint cost-and-marginal comparisons. Their optimizer
subsets need not be Hausdorff continuous, as the finite counterexample
shows.

For coordinatewise Lipschitz reference costs, objective-gap bounds and
power growth yield quantitative optimizer estimates. On finite spaces,
polyhedral structure supplies linear growth and directed Lipschitz
bounds with reference-dependent constants. On Polish spaces, total
boundedness of coordinate cost profiles gives a compact normalized
representation and fixed-cost stability. Separately, a real-valued
splitting family on an optimal support yields bounded upper
semicontinuous conjugate maximizers. Both constructions use classical
duality, but their additional hypotheses serve different purposes.

The Polish-space extension does not cover arbitrary bounded continuous
costs, and the joint and quantitative optimizer results retain their
compact-space assumptions. Useful directions are verifiable growth
criteria beyond finite spaces and conditions ensuring lower stability
of optimizer sets.


\appendix
\section{Appendix}\label{section: Appendix}

The first part of this appendix gives detailed proofs of the results
on duality, mutual $c$-conjugacy, and existence of dual optimizers,
following the order of the main text. It includes a four-marginal
example of successive conjugation and the proof of the compact
normalized representation under the cost-profile condition on Polish
spaces. The second part recalls the auxiliary results and definitions
used in these arguments.

\subsection{Selected proofs for the main results}
\label{section: Additional proofs}

\subsubsection{Convex-analytic preliminaries and compact duality}

\subsubsection*{Proof of Lemma~\ref{Lemma: Theta}}

Let \(u_1,u_2\in E\) and \(t\in[0,1]\). If both values
\(\Theta(u_1)\) and \(\Theta(u_2)\) are finite, then
\(u_1,u_2\ge-c\). Their convex combination also satisfies this inequality,
so
\[
\Theta(tu_1+(1-t)u_2)
=0=t\Theta(u_1)+(1-t)\Theta(u_2).
\]
If either value is infinite, the convexity inequality is immediate for
\(0<t<1\); at the endpoints it is an equality. Thus \(\Theta\) is convex.

To prove continuity at the constant function \(1\), note that \(c\ge0\)
and hence \(\Theta(1)=0\). Whenever \(\|u-1\|_\infty<1\), we have
\(u>0\ge-c\), so \(\Theta(u)=0\). The functional is therefore constant
on a neighbourhood of \(1\), which proves continuity there.
\hfill\(\Box\)

\subsubsection*{Proof of Lemma~\ref{Lemma: Xi}}
We first check that the value assigned to \(\Xi(u)\) does not depend on
the representation of \(u\) as a sum of coordinate functions. Suppose that $
\sum_{k=1}^K f_k(x_k)
=
\sum_{k=1}^K \widetilde f_k(x_k) \text{ for all } \boldsymbol x \in X$,
with \(f_k,\widetilde f_k\in\mathcal C(X_k)\). Set $
d_k=\widetilde f_k-f_k$.
Then $
\sum_{k=1}^K d_k(x_k)=0
\text{ for all }\boldsymbol x\in X $.
Fixing all variables except \(x_1\) shows that \(d_1\) is constant on
\(X_1\). Applying the same argument to every coordinate gives constants
\(s_1,\dots,s_K\in\mathbb R\) such that $
\widetilde f_k-f_k=s_k,
\text{ for } k=1,\dots,K$.
Since \(\sum_{k=1}^K d_k(x_k)=0\), we also have $
\sum_{k=1}^K s_k=0$.
Therefore, using that each \(\mu_k\) is a probability measure,
\[
\begin{aligned}
\sum_{k=1}^K \int_{X_k} \widetilde f_k\,d\mu_k
&=\sum_{k=1}^K \int_{X_k} f_k\,d\mu_k
  +\sum_{k=1}^K s_k\mu_k(X_k)\\
&=\sum_{k=1}^K \int_{X_k} f_k\,d\mu_k
  +\sum_{k=1}^K s_k
 =\sum_{k=1}^K \int_{X_k} f_k\,d\mu_k.
\end{aligned}
\]
Hence \(\Xi\) is well defined.

Next, we prove convexity. Let \(u_1,u_2\in E\) and \(t\in[0,1]\). If
\(\Xi(u_1)=+\infty\) or \(\Xi(u_2)=+\infty\), the convexity inequality is
immediate for \(t\in(0,1)\), while the cases \(t=0\) and \(t=1\) are trivial.
Otherwise, there exist \(h_k,g_k\in\mathcal C(X_k)\) such that $
u_1(\boldsymbol x )=\sum_{k=1}^K h_k(x_k) \text{ and }
u_2(\boldsymbol x )=\sum_{k=1}^K g_k(x_k)$.
Then $
tu_1+(1-t)u_2
=
\sum_{k=1}^K \bigl(th_k+(1-t)g_k\bigr)$.
Therefore,
\begin{eqnarray*}
{\Xi(tu_1+(1-t)u_2)}
&=&\sum_{k=1}^K
\int_{X_k}\bigl(th_k+(1-t)g_k\bigr)\,d\mu_k \\
&=&
t\sum_{k=1}^K \int_{X_k} h_k\,d\mu_k
+
(1-t)\sum_{k=1}^K \int_{X_k} g_k\,d\mu_k \\
&=&
t\Xi(u_1)+(1-t)\Xi(u_2).
\end{eqnarray*}
Thus \(\Xi\) is convex.

Finally, choose \(f_1\equiv 1\) and
\(f_k\equiv 0\) for every \(k\ge 2\). Then $
\Xi(1)=\int_{X_1}1\,d\mu_1=1<+\infty$.
Thus \(\Xi(1)<+\infty\), as required.
\hfill\(\Box\)

\subsubsection*{Proof of Theorem~\ref{theorem: duality-compact}}

The proof has two parts: we verify the hypotheses of
Fenchel--Rockafellar duality and then compute the two convex conjugates.

\paragraph{Application of Fenchel--Rockafellar duality.} By
Lemmas~\ref{Lemma: Theta} and~\ref{Lemma: Xi},
the functionals $\Theta$ and $\Xi$ are
proper and convex, $\Theta$ is continuous
at $u\equiv1$, and $\Xi(1)<+\infty$.
It remains to verify the finiteness of
$\inf_{u\in E}(\Theta(u)+\Xi(u))$.
Let $\pi_0=\mu_1\otimes\cdots\otimes\mu_K$.
For every $u$ such that
$\Theta(u)+\Xi(u)<+\infty$, we have
$u\ge-c$ and
$u=\sum_{k=1}^K f_k(x_k)$ for some
$f_k\in \mathcal C(X_k)$. Hence
\[
\Xi(u)
=\int_{X}
u\,d\pi_0
\ge-\int c\,d\pi_0
\ge-\|c\|_\infty.
\]
Since $c\ge0$, the function $u\equiv0$
satisfies $\Theta(0)=\Xi(0)=0$.
Consequently, $
-\|c\|_\infty
\le
\inf_{u\in E}\{\Theta(u)+\Xi(u)\}
\le0$.
Thus all the hypotheses of
Theorem~\ref{SuppRes:Theo:Fenchel} are satisfied, and 
\[
\inf_{u\in E}\{\Theta(u)+\Xi(u)\}
=
\max_{\pi\in E^*}
\{-\Theta^*(-\pi)-\Xi^*(\pi)\}.
\]

We recall from the Riesz representation theorem (see Theorem~\ref{SuppRes:Theo:Riez}) that the dual space \(E^{\ast}\) of $E$ can be
identified isometrically with the space 
\(\mathcal{M}(X)\) of finite
signed Radon measures on the product, endowed with the total variation norm. Hence, to each \(L\in E^{\ast}\) corresponds a unique
\(\pi\in \mathcal{M}(X)\) such that $
L(u)=\int_{X} u\, d\pi$.
We now identify both sides of the duality relation. On the left, changing
variables from \(\boldsymbol f\) to \(\boldsymbol g=-\boldsymbol f\) gives
\begin{align*}
&\inf\Big\{
\sum_{k=1}^{K} \int_{X_k} f_k\,d\mu_k
\;:\;
f_k\in \mathcal C(X_k),\ 
\sum_{k=1}^{K} f_k(x_k)\ge -c(\boldsymbol x)
\Big\}
 \\
&\qquad =
-\sup\Big\{
J(\boldsymbol g)
\;:\;
\boldsymbol g\in \mathcal F_c\cap \prod_{k=1}^{K} \mathcal C(X_k)
\Big\}.
\end{align*}

\paragraph{Legendre--Fenchel transform of \(\Theta\).}
For \(\pi\in \mathcal{M}(X)\), 
\[
\Theta^{\ast}(\pi)
=
\sup_{u\in E}
\Big\{
\int u\, d\pi - \Theta(u)
\Big\}
=
\sup_{\substack{u\in E\\ u\ge -c}}
\int u\, d\pi.
\]
Hence $
\Theta^{\ast}(-\pi)
=
\sup_{\substack{u\in E\\ u\le c}}
\int u\, d\pi$.
If \(\pi\) is not a nonnegative measure, then there exists
\(v\in \mathcal{C}_{b}(X)\), \(v\le 0\), such that
\(\int v\, d\pi>0\). Choosing \(u=t v\) and letting \(t\longrightarrow +\infty\)
gives \(\Theta^{\ast}(-\pi)=+\infty\). If instead \(\pi\ge 0\), the supremum is attained at \(u=c\), and we obtain $\Theta^{\ast}(-\pi)=\int c\, d\pi$. Thus $$
\Theta^{\ast}(-\pi)
=
\begin{cases}
\displaystyle \int c\, d\pi, & \pi\in \mathcal{M}_{+}(X),\\[0.3em]
+\infty, & \text{otherwise}.
\end{cases}$$

\paragraph{Legendre--Fenchel transform of \(\Xi\).}
For \(\Xi^*\), the relevant constraint set is the affine space of signed
measures with the prescribed marginals: 
\[
\mathcal M_{\boldsymbol \mu}
=\left\{\pi\in\mathcal M(X):
(\pr_k)_{\#}\pi=\mu_k\text{ for every }k\right\}.
\]
For $\pi\in E^*$, the definition of $\Xi$ gives
\begin{align*}
\Xi^*(\pi)
=\sup_{f_k\in \mathcal C(X_k)}
\left\{\int\sum_k f_k(x_k)\,d\pi
-\sum_k\int f_k\,d\mu_k\right\}
=\sup_{f_k\in \mathcal C(X_k)}
\sum_k\int f_k\,d\big((\pr_k)_{\#}\pi-\mu_k\big).
\end{align*}
If all marginal differences vanish, every expression in the supremum is
zero. Otherwise, one of these differences is a nonzero signed measure. By
the Riesz representation theorem, there is a continuous test function with
a nonzero integral against that measure. Taking all other test functions
to be zero, choosing the appropriate sign and scaling this function shows
that the supremum is $+\infty$.
Consequently, $$
\Xi^*(\pi)=
\begin{cases}
0,&\pi\in\mathcal M_{\boldsymbol \mu},\\
+\infty,&\text{otherwise}.
\end{cases}$$
In particular, $\Xi^*$ alone does not impose positivity.
For example, on $\{0,1\}\times\{0,1\}\times\{0\}$, the signed array
$\left(\begin{smallmatrix}3/4&-1/4\\-1/4&3/4\end{smallmatrix}\right)$
has uniform first and second marginals and third marginal $\delta_0$,
but is not a coupling.

\paragraph{Identification of the transport problem.}
The conjugate $\Xi^*$ imposes the marginal constraints, while
$\Theta^*(-\pi)$ imposes positivity and supplies the cost term.
Since every measure in $\mathcal M_{\boldsymbol \mu}$ has total mass one,
$\mathcal M_{\boldsymbol \mu}\cap\mathcal M_+=\Pi(\boldsymbol \mu)$.
Thus the combined expression is $$
\Theta^*(-\pi)+\Xi^*(\pi)=
\begin{cases}
\displaystyle\int c\,d\pi,&\pi\in\Pi(\boldsymbol \mu),\\
+\infty,&\text{otherwise}.
\end{cases}$$
Therefore, 
\[
\max_{\pi\in E^{\ast}}
\Big\{
-\Theta^{\ast}(-\pi)-\Xi^{\ast}(\pi)
\Big\}
=
\max_{\pi\in\Pi(\mu_{1},\ldots,\mu_{K})}
\left(
-\int c\, d\pi
\right)
=
-\inf_{\pi\in\Pi(\mu_{1},\ldots,\mu_{K})}
\int c\, d\pi.
\]
Comparing the two sides and changing signs, we obtain
\[
\inf_{\pi\in\Pi(\boldsymbol\mu)} I(\pi)
=
\sup_{\boldsymbol f\in
\mathcal F_c\cap\prod_{k=1}^K\mathcal C_b(X_k)}
J(\boldsymbol f).
\]
Since $
\mathcal F_c\cap\prod_{k=1}^K\mathcal C_b(X_k)
\subseteq \mathcal F_c$,
we have $
\inf_{\pi\in\Pi(\boldsymbol\mu)} I(\pi)
\le
\sup_{\boldsymbol f\in\mathcal F_c}J(\boldsymbol f)$.
The reverse inequality follows from
Lemma~\ref{Lemma: weak duality}. Consequently,
$$
\mathsf{MOT}_c(\boldsymbol\mu)
=
\sup_{\boldsymbol f\in\mathcal F_c}
J(\boldsymbol f),
$$
which completes the proof of
Theorem~\ref{theorem: duality-compact}.
\hfill\(\Box\)

\subsubsection{Tightness and duality on Polish spaces}

\subsubsection*{Proof of Lemma~\ref{Lemma: tight}}
Let \(\varepsilon>0\). Since \(Q_k\) is tight (see Definition~\ref{SuppRes:Def:Tightness}) for each
\(k=1,\ldots,K\), there exists a compact set
\(\mathcal K_{k,\varepsilon}\subset X_k\), independent of the choice of
\(\mu_k\in Q_k\), such that $
\mu_k(X_k\setminus \mathcal K_{k,\varepsilon})
\le \frac{\varepsilon}{K}
\text{ for all } \mu_k\in Q_k$.
Let \(\pi\in \Pi(Q_1,\ldots,Q_K)\), and denote by \(\pi_k\) its \(k\)-th
marginal. Then \(\pi_k\in Q_k\) for every \(k\). Hence $
\pi_k(X_k\setminus \mathcal K_{k,\varepsilon})
\le \frac{\varepsilon}{K}$.
The complement of the product is the union of the coordinate cylinders
$\{\boldsymbol x:x_k\notin\mathcal K_{k,\varepsilon}\}$. The union bound therefore gives
$$
\begin{aligned}
\pi\big(
X
\setminus
\mathcal K_{1,\varepsilon}\times\cdots\times \mathcal K_{K,\varepsilon}
\big)
&\le
\sum_{k=1}^K
\pi_k(X_k\setminus \mathcal K_{k,\varepsilon}) 
\le
\sum_{k=1}^K \frac{\varepsilon}{K}
=
\varepsilon .
\end{aligned}
$$
Since the finite product $
\mathcal K_{1,\varepsilon}\times\cdots\times \mathcal K_{K,\varepsilon}$
is compact in \(X\), this proves that
\(\Pi(Q_1,\ldots,Q_K)\) is tight.
\hfill\(\Box\)

\subsubsection*{Proof of Lemma~\ref{Lemma: compact space}}
Since each \(X_k\) is Polish, every probability measure on \(X_k\) is tight.
In particular, the singleton \(\{\mu_k\}\) is tight in \(\mathcal P(X_k)\) for
each \(k=1,\ldots,K\). This is Ulam's lemma (see \cite{Krylov2002}, Chap.~1). By Lemma~\ref{Lemma: tight}, the set
\(\Pi(\boldsymbol \mu)\) is tight in
\(\mathcal P(X)\).
Moreover, the finite product \(X\) is Polish.
Therefore, by Prokhorov's theorem, \(\Pi(\boldsymbol \mu)\) is relatively
compact for the topology of weak convergence (see Theorem~\ref{SuppRes:Theo:Tightness}).

To prove closedness, let
\((\pi_n)_{n\ge1}\subset \Pi(\boldsymbol \mu)\) and assume that
\(\pi_n\) converges weakly to some
\(\pi\in\mathcal P(X)\) (the convergence is understood in the sense of
Definition~\ref{SuppRes:Def:WeakConv}). Fix
\(k\in\{1,\ldots,K\}\) and let \(\varphi\in \mathcal C_b(X_k)\). Since the \(k\)-th
marginal of \(\pi_n\) is \(\mu_k\), we have $
\int_{X}
\varphi(x_k)\,d\pi_n(\boldsymbol x )
=
\int_{X_k}\varphi\,d\mu_k$.
The function $
\boldsymbol x\longmapsto \varphi(x_k)$
is bounded and continuous on \(X\). Hence, by weak
convergence, $
\int_{X}
\varphi(x_k)\,d\pi_n
\longrightarrow
\int_{X}
\varphi(x_k)\,d\pi$.
Therefore,
$$
\int_{X}
\varphi(x_k)\,d\pi
=
\int_{X_k}\varphi\,d\mu_k
\qquad
\text{for all } \varphi\in \mathcal C_b(X_k).
$$
This means that the \(k\)-th marginal of \(\pi\) is \(\mu_k\). Since \(k\) was
arbitrary, all marginals of \(\pi\) are \(\boldsymbol \mu\), and thus $
\pi\in\Pi(\boldsymbol \mu)$.
Since $X$ is Polish,
the topology of weak convergence on its
space of probability measures is metrizable.
Thus the sequential closedness just proved implies closedness. Combining
this with relative compactness proves that \(\Pi(\boldsymbol \mu)\) is compact.
\hfill\(\Box\)

\subsubsection*{Proof of Theorem~\ref{theorem: duality-noncompact}}

We approximate the primal problem on compact subsets and transfer
near-optimal dual families back to the full spaces. The normalization and
successive improvement below provide bounds independent of the compact
subsets, which are needed to control the extension error.

\medskip
\noindent\emph{Step 0: Preliminary bounds and existence of an optimal plan.}
We define $
\|c\|_\infty
=
\sup_{\boldsymbol x \in X}
\big|c(\boldsymbol x )\big|$.
For every probability measure $\pi$ on $X$ and every measurable set
$A\subset X$, we shall use the bound $
\left|\int_A c\, d\pi\right|
\le
\|c\|_\infty\,\pi(A)$.

Since $c$ is bounded, the value of the Kantorovich problem $
\inf_{\pi \in \Pi(\boldsymbol \mu)} I(\pi)$
is finite. Moreover, this infimum is attained. Indeed, let
$(\pi_n)_{n\ge1}\subset \Pi(\boldsymbol \mu)$ be a minimizing sequence, that is,
$
I(\pi_n)\longrightarrow
\inf_{\pi\in\Pi(\boldsymbol \mu)} I(\pi).
$
By Lemma~\ref{Lemma: compact space}, the set $\Pi(\boldsymbol \mu)$ is compact for the weak
topology. Hence, after passing to a subsequence
if necessary, there exists
$\pi_*\in\Pi(\boldsymbol \mu)$ such that
$
\pi_n \rightharpoonup \pi_* .
$
Since $c\in \mathcal C_b(X)$, the map $
\pi\longmapsto I(\pi)=\int c\,d\pi$
is continuous with respect to weak convergence. Therefore, $
I(\pi_*)
=
\lim_{n\longrightarrow\infty} I(\pi_n)
=
\inf_{\pi\in\Pi(\boldsymbol \mu)} I(\pi)$.
We fix this optimal plan $\pi_*\in\Pi(\boldsymbol\mu)$ throughout the
rest of the proof.

\medskip
\noindent\emph{Step 1: Compact truncation and comparison of the primal values.}
Fix $\varepsilon>0$ with $K\varepsilon<1$. The finite product $X$ is
Polish (see \cite{bourbaki1998}, p.~195), and Ulam's lemma gives tightness
of each $\mu_k$. Hence, for each 
$k=1,\dots,K$, there exists a compact subset $X_k^0 \subset X_k$ such that $\mu_k(X_k \setminus X_k^0) \le \varepsilon$. Set $X^0=\prod_{k=1}^K X_k^0$ and
$X_{\setminus k}^0=\prod_{l\ne k}X_l^0$. Consequently, the optimal plan $\pi_*$ satisfies
\begin{align*}
\pi_*\!\left( X 
\setminus X^0 \right)
\le K\varepsilon .
\end{align*}
Define a new probability measure $\pi^0_*$ supported on $X^0$ by
\begin{equation*}
\pi^0_*
=
\frac{
\mathbf{1}_{\,X^0}\,
}{
\pi_*(X^0)} \pi_*.
\end{equation*}
The denominator is positive by \eqref{5}. Define the induced marginals by $\mu_k^0 = (\pr_k)_\# \pi_*^0$
for all $k=1,\dots,K$. We denote by 
\[
\Pi^0(\boldsymbol \mu^0)
=
\big\{
\pi^0 \in \mathcal P(X^0)
:\; (\pr_k)_\#\pi^0 = \mu_k^0 \ \text{for all } k
\big\}
\]
the set of admissible transport plans on the compact product
$X^0$, where $\boldsymbol \mu^0 = (\mu_1^0,\dots,\mu_K^0)$. Define the truncated functional
\begin{align*}
I^0(\pi^0)
=
\int_{X^0}
c(\boldsymbol x )\, d\pi^0(\boldsymbol x ).
\end{align*}
Let $\widetilde{\pi}^{\,0} \in \Pi^0(\boldsymbol \mu^0)$ satisfy $I^0(\widetilde{\pi}^{\,0})
=
\inf_{\pi^0 \in \Pi^0(\boldsymbol \mu^0)} I^0(\pi^0)$. Existence of this compact optimizer follows from compactness of the
coupling set and continuity of the cost. Replace the restriction of
$\pi_*$ to $X^0$ by the measure
$\pi_*(X^0)\widetilde{\pi}^{\,0}$, and leave its restriction to the
complement unchanged: $
\widetilde{\pi}
=
\pi_*(X^0)\, \widetilde{\pi}^{\,0}
\;+\;
\mathbf{1}_{(X^0)^c}\, \pi_*$.
The verification that
$\widetilde{\pi}\in\Pi(\boldsymbol\mu)$
is given in
Remark~\ref{remark: Theorem duality-noncompact}. From the definition of $\widetilde{\pi}$, we have 
\[
I(\widetilde{\pi})=
\int_{X} c \, d\widetilde{\pi} =
\pi_*(X^0)
\int_{X^0} c \, d\widetilde{\pi}^{\,0}
\;+\;
\int_{(X^0)^c} c \, d\pi_*,
\]
that is, 
\[
I(\widetilde{\pi})
=
\pi_*(X^0)\, I^0(\widetilde{\pi}^{\,0})
\;+\;
\int_{(X^0)^c}
c(\boldsymbol x )\, d\pi_*(\boldsymbol x ).
\]
Subtracting the compact objective gives
\begin{align*}
I(\widetilde{\pi})
- I^0(\widetilde{\pi}^{\,0})
&=
\big(\pi_*(X^0)-1\big)\,
I^0(\widetilde{\pi}^{\,0})
\;+\;
\int_{(X^0)^c} c\, d\pi_* .
\end{align*}
Taking absolute values yields
\begin{align*}
\big|I(\widetilde{\pi}) - I^0(\widetilde{\pi}^{\,0})\big|
&\le
\big(1-\pi_*(X^0)\big)\,\big|I^0(\widetilde{\pi}^{\,0})\big|
\;+\;
\Big|\int_{(X^0)^c} c\, d\pi_*\Big|.
\end{align*}
Since $|\int_A c\,d\pi|\le \|c\|_\infty\,\pi(A)$ for any probability measure $\pi$ and any measurable set $A$, 
we have in particular
\begin{align*}
\Big|I^0(\widetilde{\pi}^{\,0})\Big| 
=
\Big|\int_{X^0} c\,d\widetilde{\pi}^{\,0}\Big|
\le \|c\|_\infty
\qquad\text{and}\qquad 
\Big|\int_{(X^0)^c} c\, d\pi_*\Big|
\le \|c\|_\infty\,\pi_*((X^0)^c).
\end{align*}
Combining these bounds with \eqref{5} yields
\begin{align*}
\Big|I(\widetilde{\pi}) - I^0(\widetilde{\pi}^{\,0})\Big|
&\le
\|c\|_\infty\,\pi_*((X^0)^c)
+\|c\|_\infty\,\pi_*((X^0)^c)
\le 2K\|c\|_\infty\,\varepsilon.
\end{align*}
In particular, $ I(\widetilde{\pi})\le I^0(\widetilde{\pi}^{\,0}) + 2K\|c\|_\infty\,\varepsilon.$
Thus,
$$
\inf_{\pi\in\Pi(\boldsymbol \mu)} I(\pi)
\;\le\;
\inf_{\pi^0\in\Pi^0(\boldsymbol \mu^0)} I^0(\pi^0)
\;+\;
2K\|c\|_\infty\,\varepsilon.
$$
For completeness, the converse comparison follows from the normalized
restriction itself. Since \(\pi_*^0\in\Pi^0(\boldsymbol\mu^0)\) and
\[
I(\pi_*)=\pi_*(X^0)I^0(\pi_*^0)
          +\int_{(X^0)^c}c\,d\pi_*,
\]
the same estimate gives
\[
\inf_{\pi^0\in\Pi^0(\boldsymbol\mu^0)}I^0(\pi^0)
\le I^0(\pi_*^0)
\le\inf_{\pi\in\Pi(\boldsymbol\mu)}I(\pi)
   +2K\|c\|_\infty\varepsilon.
\]
Thus the full and truncated primal values differ by at most
\(2K\|c\|_\infty\varepsilon\). In the remainder of the proof, we use the
first comparison, which bounds the full primal value from above.

\medskip
\noindent\emph{Step 2: Choice and normalization of a near-optimal truncated dual family.}
For a family of $K$ functions $\boldsymbol f^0 = (f_1^0,\dots,f_K^0)$ with 
$f_k^0 \in L_1(\mu_k^0)$, we define the truncated dual functional $
J^0(\boldsymbol f^0)
=
\sum_{k=1}^K \int_{X_k^0} f_k^0 \, d\mu_k^0$.
By Theorem~\ref{theorem: duality-compact}, applied to the restriction of
$c$ to $X^0$ after a constant shift of the cost if necessary,
$\inf I^0=\sup J^0$. For the chosen $\varepsilon$, there is therefore
an admissible family 
$\widetilde{\boldsymbol f}^{\,0} \in \prod_{k=1}^K \mathcal C_b(X_k^0)$ such that $J^0(\widetilde{\boldsymbol f}^{\,0})
>
\sup J^0 - \varepsilon$. We shall extend a suitable modification of
$\widetilde{\boldsymbol f}^{\,0}$ to a family $\boldsymbol f$ whose value
approximates the full dual supremum, with objective $
J(\boldsymbol f )
=
\sum_{k=1}^K \int_{X_k} f_k \, d\mu_k$.
Here admissibility means that $
\widetilde{f}_1^{\,0}(x_1) + \cdots + \widetilde f_K^{\,0}(x_K)
\;\le\;
c(\boldsymbol x )$
holds at every 
\(
\boldsymbol x \in X^0.
\)
Without loss of generality, we assume that $\varepsilon\le 1$. 
Since $c$ is bounded, we have $
-\|c\|_\infty \le c(\boldsymbol x) \text{ for all }\boldsymbol x \in X^0$.
Consider the constant family $
(f_1^0,\ldots,f_K^0)
=
(-\|c\|_\infty,0,\ldots,0)$.
Then, for every $\boldsymbol x \in X^0$, we have $
\sum_{k=1}^K f_k^0(x_k)
=
-\|c\|_\infty
\le
c(\boldsymbol x)$.
Thus, this family is admissible for the truncated dual problem. Its dual value is
\[
J^0(\boldsymbol f ^0)
=
\int_{X_1^0} -\|c\|_\infty\,d\mu_1^0
+
\sum_{k=2}^K \int_{X_k^0} 0\,d\mu_k^0.
\]
Since $\mu_1^0$ is a probability measure, this gives $
J^0(\boldsymbol f^0)
=
-\|c\|_\infty$.
This admissible family supplies the lower bound $
\sup J^0 \ge -\|c\|_\infty$.
Consequently,
\begin{align*}
J^0(\widetilde{\boldsymbol f}^{\,0})
\ge \sup J^0 - \varepsilon 
\ge -\|c\|_\infty-\varepsilon 
\ge -\|c\|_\infty-1.
\end{align*}
For any $\pi^0 \in \Pi^0(\boldsymbol \mu^0)$, we can write
\begin{align*}
J^0(\widetilde{\boldsymbol f}^{\,0})
=\sum_{k=1}^K \int_{X_k^0} \widetilde f_k^{\,0}\, d\mu_k^0
=\int_{X^0}
\big(\widetilde{f}_1^{\,0}(x_1)+\cdots+\widetilde f_K^{\,0}(x_K)\big)
\, d\pi^0(\boldsymbol x ).
\end{align*}
The continuous sum of potentials attains its maximum on $X^0$, and this
maximum is at least the displayed integral. Hence there is a point
$(x_1^0,\dots,x_K^0)\in X^0$ such that
$\widetilde{f}_1^{\,0}(x_1^0)+\cdots+\widetilde f_K^{\,0}(x_K^0)
\ge -\|c\|_\infty-1$. Define $a_k=\widetilde f_k^{\,0}(x_k^0) \text{ for } k=1,\dots,K$, and let $m=\frac{1}{K}\sum_{k=1}^K a_k$. Since $\widetilde{f}_1^{\,0}(x_1^0)+\cdots+\widetilde f_K^{\,0}(x_K^0)
\ge -\|c\|_\infty-1$, it follows that $m\ge -(\|c\|_\infty+1)/K$.
Set $s_k=m-a_k$ and add $s_k$ to the $k$-th potential. These constants
have zero sum, because $
\sum_{k=1}^K s_k=\sum_{k=1}^K (m-a_k)=Km-\sum_{k=1}^K a_k = 0$.
Thus the admissibility condition is preserved:
\begin{equation*}
\sum_{k=1}^K (\widetilde f_k^{\,0}(x_k)+s_k)
=
\sum_{k=1}^K \widetilde f_k^{\,0}(x_k)
+\sum_{k=1}^K s_k
=
\sum_{k=1}^K \widetilde f_k^{\,0}(x_k)
\le c(\boldsymbol x ).
\end{equation*}
Similarly, the dual functional is unchanged:
\begin{equation*}
J^0(\widetilde f_1^{\,0}+s_1,\dots,\widetilde f_K^{\,0}+s_K)
=
\sum_{k=1}^K\int_{X_k^0}\widetilde f_k^{\,0}\, d\mu_k^0
+\sum_{k=1}^K s_k
=
J^0(\widetilde{\boldsymbol f}^{\,0}).
\end{equation*}
Finally, evaluating $\widetilde{\boldsymbol f}^{\,0} $ at the reference point 
$(x_1^0,\dots,x_K^0)$ gives $
\widetilde f_k^{\,0}(x_k^0)+(m-a_k)
=
m
\text{ for all } k=1,\dots,K$.
After adding the constants $s_k$, we still denote the new functions by $\widetilde f_k^{\,0}$. Combined with $m\ge -(\|c\|_\infty+1)/K$, this yields the uniform lower bound
\begin{equation}\label{rr}
\widetilde f_k^{\,0}(x_k^0)\ge -(\|c\|_\infty+1)/K,\qquad k=1,\dots,K.
\end{equation}
Evaluating admissibility with all other coordinates fixed at the
reference point now gives, for every $x_k\in X_k^0$,
\begin{align*}
\widetilde f_k^{\,0}(x_k)
&\leq
c(x_1^0,\dots,x_{k-1}^0,\,x_k,\,x_{k+1}^0,\dots,x_K^0)
-
\sum_{l\ne k}\widetilde{f}_l^{\,0}(x_l^0)
\\
&\leq
c(x_1^0,\dots,x_{k-1}^0,\,x_k,\,x_{k+1}^0,\dots,x_K^0)
+
\frac{(K-1)(\|c\|_\infty+1)}{K}.
\end{align*}
Therefore,
\begin{align*}
\widetilde f_k^{\,0}(x_k)
\;\le\; 2\|c\|_\infty +1 .
\end{align*}

\medskip
\noindent\emph{Step 3: Improvement of the truncated potentials and uniform bounds.}
Starting from the normalized family, we update each component once. Set $
(f_1^{(0)},\ldots,f_K^{(0)})
=
(\widetilde f_1^{\,0},\ldots,\widetilde f_K^{\,0})=\widetilde{\boldsymbol f}^{\,0}$. Here and throughout this construction, the superscript in parentheses denotes the iteration index. Thus \(f_k^{(j)}\) denotes the \(k\)-th potential after \(j\) improvement steps.
Assume that \((f_1^{(k-1)},\ldots,f_K^{(k-1)})\) has been constructed and is
admissible on \(X^0\). We define
\[
f_k^{(k)}(x_k)
=
\inf_{\boldsymbol x_{\setminus k}\in X_{\setminus k}^0}
\Big\{
c(\boldsymbol x)
-
\sum_{l\ne k}f_l^{(k-1)}(x_l)
\Big\},
\]
and we set $
f_l^{(k)}=f_l^{(k-1)}
\text{ for }l\ne k$.
By admissibility of \((f_1^{(k-1)},\ldots,f_K^{(k-1)})\), we have $
f_k^{(k-1)}\le f_k^{(k)}
\text{ on }X_k^0$.
Hence $
J^0(f_1^{(k)},\ldots,f_K^{(k)})
\ge
J^0(f_1^{(k-1)},\ldots,f_K^{(k-1)})$.
Moreover, by the definition of \(f_k^{(k)}\), the new family
\((f_1^{(k)},\ldots,f_K^{(k)})\) is still admissible on \(X^0\). Indeed, for all
\(\boldsymbol x \in X^0\), $
f_k^{(k)}(x_k)
\le
c(\boldsymbol x)
-
\sum_{l\ne k}f_l^{(k)}(x_l)$,
and therefore $
\sum_{l=1}^K f_l^{(k)}(x_l)\le c(\boldsymbol x )$.
After performing this construction for \(k=1,\ldots,K\), we set $
\bar f_k^{\,0}=f_k^{(K)}
\text{ for } k=1,\ldots,K$.
Then \( \bar{\boldsymbol f}^{\,0}\) is admissible on \(X^0\), and $
J^0( \bar{\boldsymbol f}^{\,0})
\ge
J^0(\widetilde{\boldsymbol f}^{\,0})$.

To obtain bounds independent of $\varepsilon$, first consider the
reference points. The initial lower bound \eqref{rr} reads
\[
f_l^{(0)}(x_l^0)
=
\widetilde f_l^{\,0}(x_l^0)
\ge
-\frac{\|c\|_\infty+1}{K},
\qquad l=1,\ldots,K.
\]
At step $k$, only the $k$-th component changes, and the update satisfies $
f_k^{(k-1)}\le f_k^{(k)}
\text{ on }X_k^0$.
Every other component remains unchanged. The initial lower bound is
therefore preserved at each reference point throughout the construction:
\[
f_l^{(k-1)}(x_l^0)
\ge
-\frac{\|c\|_\infty+1}{K},
\qquad l=1,\ldots,K,\quad k=1,\ldots,K.
\]
Therefore, by evaluating the infimum defining \(f_k^{(k)}\) at
\(\boldsymbol x_{\setminus k}=(x_l^0)_{l\ne k}\), we obtain, for every \(x_k\in X_k^0\),
\[
f_k^{(k)}(x_k)
\le
c(x_1^0,\ldots,x_{k-1}^0,x_k,x_{k+1}^0,\ldots,x_K^0)
-
\sum_{l\ne k}f_l^{(k-1)}(x_l^0)
\le
2\|c\|_\infty+1.
\]
Thus every updated component satisfies the same upper bound.

For the lower bound, we use this upper bound inductively. At the initial step, the functions
\(\widetilde f_l^{\,0}=f_l^{(0)}\) satisfy the upper bound \eqref{tild}. Moreover,
at each improvement step, the newly defined function also satisfies the same
upper bound. Hence, for every \(k=1,\ldots,K\), all the functions
\(f_l^{(k-1)}\), \(l\ne k\), which appear in the infimum are bounded from above
by \(2\|c\|_\infty+1\). Therefore, for every choice of \(\boldsymbol x_{\setminus k}\),
\[
c(\boldsymbol x)
-
\sum_{l\ne k}f_l^{(k-1)}(x_l)
\ge
-\|c\|_\infty-(K-1)(2\|c\|_\infty+1).
\]
Taking the infimum gives $
f_k^{(k)}(x_k)
\ge
-(2K-1)\|c\|_\infty-(K-1)$.
Each component has been updated once by the end of the construction.
Combining its upper and lower bounds, we can choose a constant \(C>0\),
depending only on \(K\) and \(\|c\|_\infty\), such that $
|\bar f_k^{\,0}(x_k)|\le C,
\text{ for all } x_k\in X_k^0,\text{ } k=1,\ldots,K$.

\medskip
\noindent\emph{Step 4: Extension to a globally admissible family.}
We extend the improved potentials by a negative constant outside their
compact domains. Let $
M=(K-1)C+\|c\|_\infty$.
For every \(k=1,\ldots,K\), define
\[
f_k(x_k)
=
\begin{cases}
\bar f_k^{\,0}(x_k), & x_k\in X_k^0,\\
-M, & x_k\in X_k\setminus X_k^0.
\end{cases}
\]
We first verify measurability. The functions
\(\bar f_k^{\,0}\) are continuous on their compact domains:
the initial functions \(\widetilde f_k^{\,0}\) are continuous on \(X_k^0\). At each improvement step, the expression inside the infimum is jointly
continuous on the compact product $X^0$. Its uniform continuity implies
continuity of the infimum over the fixed compact set
\(X_{\setminus k}^0\). Continuity, and hence Borel measurability, is
therefore preserved by induction.
Since \(X_k^0\) is compact in the Polish space \(X_k\), it is a Borel subset of
\(X_k\). Hence both \(X_k^0\) and \(X_k\setminus X_k^0\) are Borel sets. On
\(X_k^0\), the function \(f_k\) coincides with the measurable function
\(\bar f_k^{\,0}\), while on \(X_k\setminus X_k^0\), it is equal to the constant
\(-M\). Therefore \(f_k\) is measurable on \(X_k\).
Moreover, each \(f_k\) is bounded, so
\(f_k\in L_1(\mu_k)\).

It remains to check feasibility. Let \(\boldsymbol x\in X\). If
\(\boldsymbol x \in X^0\), then, since
\( \bar{\boldsymbol f}^{\,0}\) is admissible on \(X^0\), $
\sum_{k=1}^K f_k(x_k)
=
\sum_{k=1}^K \bar f_k^{\,0}(x_k)
\le c(\boldsymbol x)$.
If, on the other hand, at least one coordinate \(x_j\) does not belong to
\(X_j^0\), then \(f_j(x_j)=-M\), while \(f_k(x_k)\le C\) for every
\(k\ne j\). Hence $
\sum_{k=1}^K f_k(x_k)
\le
-M+(K-1)C
=
-\|c\|_\infty
\le
c(\boldsymbol x)$.
Thus $
\sum_{k=1}^K f_k(x_k)\le c(\boldsymbol x)
\text{ for all }\boldsymbol x\in X$.
Therefore \(\boldsymbol f \in\mathcal F_c\).

\medskip
\noindent\emph{Step 5: Comparison of the full and truncated dual values.}
We compare both dual objectives using the fixed optimal plan $\pi_*$.
Since its marginals are $\mu_k$,
\[
J(\boldsymbol f )
=
\sum_{k=1}^K \int_{X_k} f_k\,d\mu_k
=
\int_{X}
\left(\sum_{k=1}^K f_k(x_k)\right)\,d\pi_* .
\]
Split the integral over the disjoint decomposition $
X
=
X^0\cup (X^0)^c$.
Thus,
\[
J(\boldsymbol f )
=
\int_{X^0}
\left(\sum_{k=1}^K f_k(x_k)\right)d\pi_*
+
\int_{(X^0)^c}
\left(\sum_{k=1}^K f_k(x_k)\right)d\pi_*.
\]
On \(X^0\), we have \(x_k\in X_k^0\) for every \(k=1,\ldots,K\). By the
definition of the extension, \(f_k=\bar f_k^{\,0}\) on \(X_k^0\). Hence, on \(X^0\), $
\sum_{k=1}^K f_k(x_k)
=
\sum_{k=1}^K \bar f_k^{\,0}(x_k)$.
Consequently, $
\int_{X^0}
\left(\sum_{k=1}^K f_k(x_k)\right)d\pi_*
=
\int_{X^0}
\left(\sum_{k=1}^K \bar f_k^{\,0}(x_k)\right)d\pi_*$.
By definition of the normalized restriction \(\pi_*^0\), we have $
\pi_*^0
=
\frac{\mathbf 1_{X^0}}{\pi_*(X^0)}\,\pi_*$.
Equivalently, on \(X^0\), $
d\pi_*=\pi_*(X^0)\,d\pi_*^0$.
Therefore,
\[
\int_{X^0}
\left(\sum_{k=1}^K \bar f_k^{\,0}(x_k)\right)d\pi_*
=
\pi_*(X^0)
\int_{X^0}
\left(\sum_{k=1}^K \bar f_k^{\,0}(x_k)\right)d\pi_*^0.
\]
Combining the previous identities, we obtain
\[
J(\boldsymbol f )
=
\pi_*(X^0)
\int_{X^0}
\left(\sum_{k=1}^K \bar f_k^{\,0}(x_k)\right)d\pi_*^0
+
\int_{(X^0)^c}
\left(\sum_{k=1}^K f_k(x_k)\right)d\pi_* .
\]
Moreover,
\[
J^0( \bar{\boldsymbol f}^{\,0})
=
\int_{X^0}
\left(\sum_{k=1}^K \bar f_k^{\,0}(x_k)\right)d\pi_*^0 .
\]
Recall that $M\ge C$, since $K\ge3$, so
$|\bar f_k^{\,0}|\le C$ and $|f_k|\le M$. Subtracting the preceding
identity for $J^0$ from that for $J$ therefore gives
\[
\begin{aligned}
|J(\boldsymbol f)-J^0(\bar{\boldsymbol f}^{\,0})|
&\le (1-\pi_*(X^0))KC+\pi_*((X^0)^c)KM\\
&\le K^2(C+M)\varepsilon,
\end{aligned}
\]
where we used \eqref{5}. In particular, $
J(\boldsymbol f )
\ge
J^0( \bar{\boldsymbol f}^{\,0})
-
K^2(C+M)\varepsilon$.
Since $
J^0( \bar{\boldsymbol f}^{\,0})
\ge
J^0(\widetilde{\boldsymbol f}^{\,0})
\ge
\inf I^0-\varepsilon$
and $
\inf I^0
\ge
\inf I-2K\|c\|_\infty\,\varepsilon$,
we deduce that
$$
J(\boldsymbol f)
\ge
\inf I
-
\big(2K\|c\|_\infty+1+K^2(C+M)\big)\varepsilon .
$$

\medskip
\noindent\emph{Step 6: Passage to the limit and conclusion.}
Since \(\boldsymbol f\in\mathcal F_c\), we have $
\sup_{\boldsymbol f\in\mathcal F_c}J(\boldsymbol f)
\ge
J(\boldsymbol f)$.
Therefore, $
\sup_{\boldsymbol f\in\mathcal F_c}J(\boldsymbol f)
\ge
\inf I
-
\big(2K\|c\|_\infty+1+K^2(C+M)\big)\varepsilon$.
The constants $C$ and $M$ depend only on $K$ and $\|c\|_\infty$,
and hence are independent of $\varepsilon$. Letting
$\varepsilon\longrightarrow0$ gives $
\sup_{\boldsymbol f\in\mathcal F_c}J(\boldsymbol f)
\ge
\inf_{\pi\in\Pi(\boldsymbol \mu)} I(\pi)$.
Combining this inequality with the weak duality inequality proved in
Lemma~\ref{Lemma: weak duality}, we conclude that
$$
\inf_{\pi\in\Pi(\boldsymbol \mu)} I(\pi)
=
\sup_{\boldsymbol f\in\mathcal F_c}J(\boldsymbol f).
$$
This proves the Kantorovich duality formula in the non-compact setting.
\hfill\(\Box\)

\begin{remark}\label{remark: Theorem duality-noncompact}
To verify that the plan \(\widetilde{\pi}\) defined by $\widetilde{\pi}=\pi_*(X^0)\,\widetilde{\pi}^{\,0}+\mathbf{1}_{(X^0)^c}\,\pi_*$
belongs to \(\Pi(\boldsymbol \mu)\), it is enough to check that each of its marginals coincides with the corresponding $\mu_k$. Fix a measurable set \(A\subset X_k\). Using the definition of
\(\widetilde{\pi}\), we write
\begin{eqnarray*}
(\widetilde{\pi})_k(A)
&=& \widetilde{\pi}(X_1\times\cdots\times A\times\cdots\times X_K)
\\
&=& \pi_*(X^0)\,
 \widetilde{\pi}^{\,0}\!\big((X_1\times\cdots\times A\times\cdots\times X_K)\cap X^0\big)
\\
&&
 +\pi_*\!\big((X^0)^c\cap(X_1\times\cdots\times A\times\cdots\times X_K)\big).
\end{eqnarray*}
Since \(\widetilde{\pi}^{\,0}\) is supported on \(X^0\) and has \(k\)-th marginal \(\mu_k^0\),
the first term equals
\begin{eqnarray*}
\pi_*(X^0)\,\mu_k^0(A\cap X_k^0)
&=&\pi_*(X^0)\,
\frac{\pi_*\!\big((X_1^0\times\cdots\times(A\cap X_k^0)\times\cdots\times X_K^0)\big)}{\pi_*(X^0)}\\
&=&\pi_*\!\big((X_1\times\cdots\times A\times\cdots\times X_K)\cap X^0\big),
\end{eqnarray*}
where we used \eqref{uuu}. Therefore,
\begin{align*}
(\widetilde{\pi})_k(A)
&=
\pi_*\big((X_1\times\cdots\times A\times\cdots\times X_K)\cap X^0\big)
+
\pi_*\big((X^0)^c\cap(X_1\times\cdots\times A\times\cdots\times X_K)\big).
\end{align*}
The two sets in the last line form a disjoint partition of the cylinder
\(X_1\times\cdots\times A\times\cdots\times X_K\); hence
$$
(\widetilde{\pi})_k(A)
=\pi_*\big(X_1\times\cdots\times A\times\cdots\times X_K\big)
=\mu_k(A)
$$
because \(\pi_*\in\Pi(\boldsymbol \mu)\). Thus, each marginal of \(\widetilde{\pi}\) is
equal to \(\mu_k\), and we conclude that \(\widetilde{\pi}\in\Pi(\boldsymbol \mu)\).

\end{remark}

\subsubsection{Successive partial \texorpdfstring{$c$}{c}-transforms and maximality}

\subsubsection*{Proof of Lemma~\ref{Lemma: sequential-conjugation}}

We distinguish the order-reversing property of a partial transform from
its effect on a feasible family. The following three observations show
that every update is finite and feasible, and that one pass through the
components yields mutual conjugacy.

\emph{(P1) Antitonicity in the other components.}
If $h_l\le q_l$ for every $l\ne k$, then
\[
(T_k\boldsymbol h)_k
=\inf_{\boldsymbol x_{\setminus k}}\Bigl(c-\sum_{l\ne k}h_l\Bigr)
\ \ge\ \inf_{\boldsymbol x_{\setminus k}}\Bigl(c-\sum_{l\ne k}q_l\Bigr)
=(T_k\boldsymbol q)_k.
\]
This reverses the order of the updated components. It does not imply
that the full families $T_k\boldsymbol h$ and $T_k\boldsymbol q$ are
comparable.

\emph{(P2) Finiteness, increase, and preservation of feasibility.}
Let $\boldsymbol h$ be a real-valued feasible family. The inequality
$\sum_l h_l\le c$ gives $h_k\le (T_k\boldsymbol h)_k$.
For a fixed $x_k$, choose points $\bar x_l\in X_l$, $l\ne k$, and
evaluate the infimum at this complementary tuple. Then
\[
-\infty<h_k(x_k)\le (T_k\boldsymbol h)_k(x_k)
\le c(\bar x_1,\ldots,x_k,\ldots,\bar x_K)
-\sum_{l\ne k}h_l(\bar x_l)<\infty.
\]
Thus the transform is real-valued: feasibility provides a finite lower
bound, and evaluation at a fixed tuple provides a finite upper bound.
Moreover, the definition of the infimum gives
\[
(T_k\boldsymbol h)_k(x_k)+\sum_{l\ne k}h_l(x_l)
\le c(\boldsymbol x).
\]
Each update therefore preserves feasibility and increases the family in
the componentwise order. Applying this observation successively proves
these assertions for all intermediate families.

\emph{(P3) Previously updated components remain conjugate.}
Immediately after the $k$th update, the complementary components have
not changed during that update, so
\[
(T_k\boldsymbol h^{(k)})_k=h_k^{(k)}.
\]
All later updates increase only the other components. In particular,
$g_k=h_k^{(k)}$, and (P1) gives
\[
(T_k\boldsymbol g)_k
\le (T_k\boldsymbol h^{(k)})_k
=h_k^{(k)}=g_k.
\]
Since $\boldsymbol g$ is feasible, (P2) gives the reverse inequality.
Hence $T_k\boldsymbol g=\boldsymbol g$ for every $k$, proving mutual
$c$-conjugacy.

Finally, on $\mathcal E$, feasibility and
$\boldsymbol g\succeq\boldsymbol h$ yield
\[
c=\sum_k h_k\le\sum_k g_k\le c.
\]
Thus the contact equality is preserved. The argument requires only that
each component be updated exactly once, and therefore applies to any
permutation of the indices.
\hfill$\Box$

\hypertarget{example-four-marginal-sweep}{}
\subsubsection*{An explicit four-marginal construction}

\begin{example}[A four-marginal conjugation sweep]
\label{Example: four-marginal-sweep}
Let $K=4$, $X_k=\{0,1\}$ for $k=1,\ldots,4$, and consider the cost
\[
c(x_1,x_2,x_3,x_4)=x_1+x_2+x_3+x_4.
\]
Start from the family $\boldsymbol h$ defined by $h_k(x_k)=-x_k$ for
every $k$. It is feasible because
\[
\sum_{k=1}^4 h_k(x_k)
=-\sum_{k=1}^4 x_k
\le c(\boldsymbol x)
\qquad\text{for every }\boldsymbol x\in X.
\]
We update the components in the order $1,2,3,4$, using the current values
of the other components at each step.

The first update gives
\begin{align*}
P_c(\boldsymbol h_{\setminus 1})(x_1)
&=\min_{x_2,x_3,x_4\in\{0,1\}}
\Bigl\{c(\boldsymbol x)-\sum_{k=2}^4 h_k(x_k)\Bigr\}\\
&=\min_{x_2,x_3,x_4\in\{0,1\}}
\bigl\{x_1+2x_2+2x_3+2x_4\bigr\}
=x_1.
\end{align*}
Thus $\boldsymbol h^{(1)}=(x_1,-x_2,-x_3,-x_4)$.

For the second component,
\begin{align*}
P_c(\boldsymbol h^{(1)}_{\setminus 2})(x_2)
&=\min_{x_1,x_3,x_4\in\{0,1\}}
\bigl\{c(\boldsymbol x)-x_1+x_3+x_4\bigr\}\\
&=\min_{x_1,x_3,x_4\in\{0,1\}}
\bigl\{x_2+2x_3+2x_4\bigr\}
=x_2,
\end{align*}
so $\boldsymbol h^{(2)}=(x_1,x_2,-x_3,-x_4)$.

Similarly,
\[
P_c(\boldsymbol h^{(2)}_{\setminus 3})(x_3)
=\min_{x_1,x_2,x_4\in\{0,1\}}
\{x_3+2x_4\}=x_3,
\]
which gives $\boldsymbol h^{(3)}=(x_1,x_2,x_3,-x_4)$.
Finally,
\[
P_c(\boldsymbol h^{(3)}_{\setminus 4})(x_4)
=\min_{x_1,x_2,x_3\in\{0,1\}}x_4
=x_4.
\]
The successive families are therefore
\[
\begin{aligned}
\boldsymbol h^{(0)}
&=(-x_1,-x_2,-x_3,-x_4),\\
\boldsymbol h^{(1)}
&=( x_1,-x_2,-x_3,-x_4),\\
\boldsymbol h^{(2)}
&=( x_1, x_2,-x_3,-x_4),\\
\boldsymbol h^{(3)}
&=( x_1, x_2, x_3,-x_4),\\
\boldsymbol g=\boldsymbol h^{(4)}
&=( x_1, x_2, x_3, x_4).
\end{aligned}
\]
Each entry denotes a function of the corresponding coordinate.

At every step, feasibility is preserved and the updated component does
not decrease. The final family satisfies
\[
\sum_{k=1}^4 g_k(x_k)=c(\boldsymbol x)
\qquad\text{for every }\boldsymbol x\in X.
\]
Consequently, for every $k=1,\ldots,4$,
\[
P_c(\boldsymbol g_{\setminus k})(x_k)
=\min_{\boldsymbol x_{\setminus k}}
\Bigl\{c(\boldsymbol x)-\sum_{l\ne k}g_l(x_l)\Bigr\}
=x_k=g_k(x_k),
\]
which verifies mutual $c$-conjugacy directly.

The initial contact set is $\mathcal E=\{(0,0,0,0)\}$. At this point,
\[
\sum_{k=1}^4 h_k(0)=c(0,0,0,0)=0,
\qquad
\sum_{k=1}^4 g_k(0)=c(0,0,0,0)=0.
\]
Thus the original contact point is preserved, as asserted by
Lemma~\ref{Lemma: sequential-conjugation}.
\end{example}

\subsubsection*{Proof of Lemma~\ref{Lemma: usc-conjugation}}

For each fixed $\boldsymbol x_{\setminus k}$, the function
\[
x_k\longmapsto c(x_k,\boldsymbol x_{\setminus k})
-\sum_{l\ne k}h_l(x_l)
\]
is continuous, since the second term is constant in $x_k$. Hence, for
every $a\in\mathbb R$,
\[
\begin{aligned}
&\{x_k\in X_k:P_c(\boldsymbol h_{\setminus k})(x_k)<a\}\\
&\qquad=\bigcup_{\boldsymbol x_{\setminus k}\in X_{\setminus k}}
\Bigl\{x_k\in X_k:
 c(x_k,\boldsymbol x_{\setminus k})
-\sum_{l\ne k}h_l(x_l)<a\Bigr\}
\end{aligned}
\]
is open. This proves upper semicontinuity and therefore Borel
measurability. Every component of the final family in the successive
construction is obtained by such a transform, so the same conclusion
holds for that family.
\hfill$\Box$

\subsubsection*{Proof of Proposition~\ref{dm:proposition:maximality}}

Suppose first that $\boldsymbol f$ is mutually $c$-conjugate, and let
$\boldsymbol g\in\mathcal F_c^b$ satisfy
$\boldsymbol g\succeq\boldsymbol f$. Feasibility and antitonicity of
the partial transform give
\[
g_k\le P_c(\boldsymbol g_{\setminus k})
\le P_c(\boldsymbol f_{\setminus k})=f_k,
\qquad k=1,\ldots,K.
\]
Together with $\boldsymbol g\succeq\boldsymbol f$, this implies
$\boldsymbol g=\boldsymbol f$. Thus $\boldsymbol f$ is Pareto maximal.

Conversely, suppose that $\boldsymbol f$ is maximal in
$\mathcal F_c^b$, and fix $k$. Since $c$ and all components of
$\boldsymbol f$ are bounded, the transform
$P_c(\boldsymbol f_{\setminus k})$ is bounded and real-valued.
Lemma~\ref{Lemma: usc-conjugation} also makes it upper semicontinuous,
and hence Borel measurable. Feasibility gives
$f_k\le P_c(\boldsymbol f_{\setminus k})$, and replacing $f_k$ by this
transform preserves feasibility. The resulting family therefore lies
in $\mathcal F_c^b$ and dominates $\boldsymbol f$. Maximality forces
\[
f_k=P_c(\boldsymbol f_{\setminus k}).
\]
As $k$ was arbitrary, $\boldsymbol f$ is mutually $c$-conjugate.
\hfill$\Box$

\subsubsection{Uniform estimates and dual optimizers on compact spaces}

\subsubsection*{Proof of Lemma~\ref{Lemma: equicontinuity}}

Fix $x_k,x_k'\in X_k$. For each complementary tuple
$\boldsymbol x_{\setminus k}$, define
\[
g_{\boldsymbol x_{\setminus k}}(x_k)
=c(\boldsymbol x)-\sum_{l\ne k}f_l(x_l).
\]
Since the second term does not depend on $x_k$,
\begin{align*}
\bigl|g_{\boldsymbol x_{\setminus k}}(x_k)
-g_{\boldsymbol x_{\setminus k}}(x_k')\bigr|
&=\bigl|c(x_1,\ldots,x_k,\ldots,x_K)
-c(x_1,\ldots,x_k',\ldots,x_K)\bigr|\\
&\le\omega(d_k(x_k,x_k')).
\end{align*}
In particular,
\[
g_{\boldsymbol x_{\setminus k}}(x_k)
\le g_{\boldsymbol x_{\setminus k}}(x_k')
+\omega(d_k(x_k,x_k')).
\]
Taking the infimum over $\boldsymbol x_{\setminus k}$ and using
$f_k=P_c(\boldsymbol f_{\setminus k})$ gives
\[
f_k(x_k)\le f_k(x_k')+\omega(d_k(x_k,x_k')).
\]
Exchanging $x_k$ and $x_k'$ yields the reverse inequality, and hence
\[
|f_k(x_k)-f_k(x_k')|\le\omega(d_k(x_k,x_k')).
\]
\hfill$\Box$

\subsubsection*{Proof of Lemma~\ref{uniform-bounds}}

\emph{Oscillation bounds.}
Since $\boldsymbol f^n\in\mathcal G_c$,
Lemma~\ref{Lemma: equicontinuity} gives
\[
|f_k^n(x)-f_k^n(y)|\le\omega(d_k(x,y)),
\qquad x,y\in X_k,\quad k=1,\ldots,K.
\]
Set $D_k=\omega(\operatorname{diam}(X_k))$. Then
\[
\operatorname{osc}_{X_k}(f_k^n)
=\sup_{x\in X_k}f_k^n(x)-\inf_{x\in X_k}f_k^n(x)
\le D_k
\]
for every $n$, and $D_k$ does not depend on $n$.

\emph{Normalization and its invariances.}
Choose reference points $x_k^0\in X_k$, $k=1,\ldots,K$, independent
of $n$. Define
\[
a_k^n=-f_k^n(x_k^0)\quad(k<K),
\qquad a_K^n=-\sum_{k=1}^{K-1}a_k^n,
\qquad \widetilde f_k^n=f_k^n+a_k^n\quad(k=1,\ldots,K).
\]
The zero-sum condition $\sum_{k=1}^Ka_k^n=0$ preserves feasibility and
the dual objective. It also preserves mutual $c$-conjugacy, since
\begin{align*}
P_c(\widetilde{\boldsymbol f}^{\,n}_{\setminus k})(x_k)
&=\inf_{\boldsymbol x_{\setminus k}}
\Bigl\{c(\boldsymbol x)-\sum_{l\ne k}(f_l^n(x_l)+a_l^n)\Bigr\}\\
&=P_c(\boldsymbol f^n_{\setminus k})(x_k)-\sum_{l\ne k}a_l^n\\
&=f_k^n(x_k)+a_k^n=\widetilde f_k^n(x_k).
\end{align*}
Thus $\widetilde{\boldsymbol f}^{\,n}\in\mathcal G_c$ and
$J(\widetilde{\boldsymbol f}^{\,n})=J(\boldsymbol f^n)$.
For $k<K$, the condition $\widetilde f_k^n(x_k^0)=0$ and the
oscillation bound immediately give
\[
\|\widetilde f_k^n\|_\infty\le D_k,
\qquad k=1,\ldots,K-1.
\]

\emph{Control of the last component.}
Set $t_n=\widetilde f_K^n(x_K^0)$ and $M=\|c\|_\infty$.
Feasibility at the reference tuple gives
\[
t_n=\sum_{k=1}^K\widetilde f_k^n(x_k^0)
\le c(x_1^0,\ldots,x_K^0)\le M.
\]
To obtain a lower bound, let
\[
S=\sup_{\boldsymbol f\in\mathcal F_c}J(\boldsymbol f).
\]
Weak duality yields $S\le M$, while the feasible constant family
$(-M,0,\ldots,0)$ gives $S\ge-M$. Thus $S\in\mathbb R$.
Since the translated sequence remains maximizing, there is $N\ge1$
such that
\[
J(\widetilde{\boldsymbol f}^{\,n})\ge S-1
\qquad(n\ge N).
\]
The oscillation estimate for the last component gives
\[
|\widetilde f_K^n(x_K)-t_n|\le D_K,
\qquad x_K\in X_K,
\]
and therefore $\widetilde f_K^n(x_K)\le t_n+D_K$.
Integrating this inequality and the bounds for the first $K-1$
components against the probability marginals yields
\[
\int_{X_k}\widetilde f_k^n\,d\mu_k\le D_k\quad(k<K),
\qquad
\int_{X_K}\widetilde f_K^n\,d\mu_K\le t_n+D_K.
\]
Consequently, for $n\ge N$,
\[
\begin{aligned}
S-1
&\le J(\widetilde{\boldsymbol f}^{\,n})
=\sum_{k=1}^K\int_{X_k}\widetilde f_k^n\,d\mu_k\\
&\le\sum_{k=1}^{K-1}D_k+t_n+D_K
=t_n+\sum_{k=1}^K D_k.
\end{aligned}
\]
Combining this lower bound for $t_n$ with the upper bound above gives
\[
S-1-\sum_{k=1}^K D_k\le t_n\le M,
\qquad n\ge N.
\]
The remaining terms $t_1,\ldots,t_{N-1}$ are finite. Hence there is a
constant $T<\infty$, independent of $n$, such that $|t_n|\le T$ for
all $n\ge1$. The oscillation estimate then gives
\[
|\widetilde f_K^n(x_K)|
\le|\widetilde f_K^n(x_K)-t_n|+|t_n|
\le D_K+T,
\]
so $\|\widetilde f_K^n\|_\infty\le D_K+T$ for every $n\ge1$.
Together with the bounds for $k<K$, this proves that there are finite
constants $C_k$, independent of $n$, such that
\[
\|\widetilde f_k^n\|_\infty\le C_k,
\qquad k=1,\ldots,K,\quad n\ge1.
\]
The translated sequence remains maximizing in $\mathcal G_c$, as
required.
\hfill$\Box$

\subsubsection*{Proof of Theorem~\ref{Theorem: compact c-conjugate}}

\noindent\emph{Finiteness of the dual value.}
Since $c$ is continuous on the compact product $X$, it is bounded.
Set $M=\|c\|_\infty$ and
$V_c=\sup_{\boldsymbol f\in\mathcal F_c}J(\boldsymbol f)$.
For any $\boldsymbol f\in\mathcal F_c$, integrating
$\sum_k f_k\le c$ against the product measure gives
\[
J(\boldsymbol f)=\sum_{k=1}^K\int_{X_k}f_k\,d\mu_k
\le\int_X c\,d(\mu_1\otimes\cdots\otimes\mu_K)\le M.
\]
Thus $V_c\le M$. The constant family $(-M,0,\ldots,0)$ is feasible
because $-M\le c$, so $V_c\ge-M$. In particular, $V_c\in\mathbb R$.

\medskip
\noindent\emph{Step 1: construction of a maximizing sequence in
$\mathcal G_c$.}
Theorem~\ref{theorem: duality-compact}, together with the constant-shift
observation following Theorem~\ref{theorem: duality-noncompact}, provides
continuous feasible families approaching $V_c$. More precisely, for
every $n\ge1$, choose
\[
\boldsymbol u^n\in\mathcal F_c\cap\prod_{k=1}^K\mathcal C(X_k)
\quad\text{such that}\quad
V_c-\frac1n<J(\boldsymbol u^n)\le V_c.
\]
These families form a maximizing sequence, but need not be mutually
$c$-conjugate. Apply Lemma~\ref{Lemma: sequential-conjugation} to each
$\boldsymbol u^n$, updating every component once in the order
$1,\ldots,K$, and set
\[
\boldsymbol g^n=T_K\cdots T_1\boldsymbol u^n.
\]
The lemma ensures that all intermediate families are real-valued and
feasible, that each update increases the corresponding component, and
that the final family is mutually $c$-conjugate. Thus
\[
\boldsymbol g^n\succeq\boldsymbol u^n,
\qquad g_k^n=P_c(\boldsymbol g_{\setminus k}^n),
\qquad k=1,\ldots,K.
\]
By Lemma~\ref{Lemma: equicontinuity}, every $g_k^n$ is continuous.
Compactness of $X_k$ makes it bounded and therefore an element of
$L_1(\mu_k)$. Together with feasibility and conjugacy, this gives
$\boldsymbol g^n\in\mathcal G_c$.

The pointwise domination and monotonicity of integration imply
$J(\boldsymbol u^n)\le J(\boldsymbol g^n)$. Since
$\mathcal G_c\subseteq\mathcal F_c$, the definition of $V_c$ also gives
$J(\boldsymbol g^n)\le V_c$. Therefore
\[
V_c-\frac1n<J(\boldsymbol u^n)
\le J(\boldsymbol g^n)\le V_c,
\]
and $J(\boldsymbol g^n)\to V_c$. We have obtained a maximizing
sequence contained in $\mathcal G_c$. This construction uses the
componentwise increase and scalar order reversal in
Lemma~\ref{Lemma: sequential-conjugation}; it does not require the maps
$T_k$ to preserve the order between distinct feasible families.

\medskip
\noindent\emph{Step 2: equicontinuity, normalization, and uniform bounds.}
Lemma~\ref{Lemma: equicontinuity} gives each family $(g_k^n)_{n\ge1}$
a common modulus of continuity depending only on $c$.
By Lemma~\ref{uniform-bounds}, there are constants
$a_1^n,\ldots,a_K^n$ with $\sum_{k=1}^K a_k^n=0$ such that
\[
\begin{gathered}
\widetilde{\boldsymbol g}^{\,n}
=(g_1^n+a_1^n,\ldots,g_K^n+a_K^n)\in\mathcal G_c,\\
\|g_k^n+a_k^n\|_\infty\le C_k<\infty,
\end{gathered}
\]
where $C_k$ is independent of $n$. These translations preserve the
moduli of continuity. They also preserve the objective, since the
constants sum to zero and each marginal is a probability measure:
$J(\widetilde{\boldsymbol g}^{\,n})=J(\boldsymbol g^n)$.
Relabeling the translated sequence as $(\boldsymbol g^n)_{n\ge1}$,
we may therefore assume that it remains maximizing in $\mathcal G_c$
and satisfies
\[
\|g_k^n\|_\infty\le C_k,
\qquad n\ge1,\quad k=1,\ldots,K.
\]

\medskip
\noindent\emph{Step 3: passage to the limit by Arzel\`a--Ascoli.}
For each $k$, the sequence $(g_k^n)_{n\ge1}$ is uniformly bounded and
equicontinuous on the compact metric space $X_k$. Apply the
Arzel\`a--Ascoli theorem successively for $k=1,\ldots,K$. Along a common
subsequence, still indexed by $n$, there are continuous functions
$f_k:X_k\to\mathbb R$ such that
\[
g_k^n\longrightarrow f_k\quad\text{uniformly on }X_k,
\qquad k=1,\ldots,K.
\]
We verify feasibility, optimality, and mutual conjugacy of the limit.

\emph{Feasibility.}
Passing to the limit in $\sum_k g_k^n\le c$ gives
$\sum_k f_k\le c$. Since the continuous functions $f_k$ are bounded
on $X_k$, they are integrable, and thus $\boldsymbol f\in\mathcal F_c$.

\emph{Optimality.}
Uniform convergence implies
\[
\int_{X_k}g_k^n\,d\mu_k\longrightarrow
\int_{X_k}f_k\,d\mu_k
\qquad(k=1,\ldots,K).
\]
Hence $J(\boldsymbol g^n)\to J(\boldsymbol f)$, and the maximizing
property yields
\[
J(\boldsymbol f)=V_c
=\sup_{\boldsymbol h\in\mathcal F_c}J(\boldsymbol h).
\]

\emph{Mutual $c$-conjugacy.}
Fix $k$. Since $\boldsymbol g^n\in\mathcal G_c$,
\[
g_k^n(x_k)
=\inf_{\boldsymbol x_{\setminus k}}
\Bigl\{c(\boldsymbol x)-\sum_{l\ne k}g_l^n(x_l)\Bigr\}
=\inf_{\boldsymbol x_{\setminus k}}H_n(\boldsymbol x),
\]
where $H_n=c-\sum_{l\ne k}g_l^n$. Set
$H=c-\sum_{l\ne k}f_l$. Uniform convergence of the components gives
$H_n\to H$ uniformly on $X$. For bounded functions $\varphi,\psi$
on the same set,
\[
|\inf\varphi-\inf\psi|\le\|\varphi-\psi\|_\infty.
\]
Applying this inequality to the infima over the complementary
coordinates at each fixed $x_k$ gives
\[
\Bigl\|\inf_{\boldsymbol x_{\setminus k}}H_n
-\inf_{\boldsymbol x_{\setminus k}}H\Bigr\|_{\infty,\,x_k}
\le\|H_n-H\|_\infty\longrightarrow0.
\]
Thus the transforms converge uniformly to
$\inf_{\boldsymbol x_{\setminus k}}H$. Since they equal $g_k^n$, whose
uniform limit is $f_k$, the two limits coincide:
\[
f_k(x_k)=\inf_{\boldsymbol x_{\setminus k}}
\Bigl\{c(\boldsymbol x)-\sum_{l\ne k}f_l(x_l)\Bigr\},
\qquad x_k\in X_k.
\]
This holds for every $k$, so $T_k\boldsymbol f=\boldsymbol f$ for all
$k$. By Definition~\ref{definition: Gc space},
$\boldsymbol f\in\mathcal G_c$.

\medskip
\noindent\emph{Conclusion.}
The family $\boldsymbol f\in\mathcal G_c\subseteq\mathcal F_c$
realizes the supremum over $\mathcal F_c$. The inclusion gives one
inequality between the two maxima, and the same maximizer gives the
reverse inequality. Therefore
\[
\max_{\boldsymbol h\in\mathcal G_c}J(\boldsymbol h)
=\max_{\boldsymbol h\in\mathcal F_c}J(\boldsymbol h)=V_c.
\]
Finally, Theorem~\ref{theorem: duality-compact} identifies
$V_c=\mathsf{MOT}_c(\boldsymbol\mu)$, so
\[
\mathsf{MOT}_c(\boldsymbol\mu)
=\max_{\boldsymbol f\in\mathcal F_c}J(\boldsymbol f)
=\max_{\boldsymbol f\in\mathcal G_c}J(\boldsymbol f).
\]
\hfill$\Box$

\subsubsection{Splitting families and dual optimizers on Polish spaces}

\subsubsection*{Proof of Proposition~\ref{Proposition: splitting}}

Let $\boldsymbol f=(f_1,\ldots,f_K)$ be the given real-valued
splitting family for $\Gamma$. Thus
\[
\sum_{k=1}^K f_k(x_k)\le c(\boldsymbol x)
\qquad\text{for every }\boldsymbol x\in X,
\]
with equality on $\Gamma$. We first apply successive conjugation to
this feasible family, then verify that the splitting equality is
preserved.

We apply Lemma~\ref{Lemma: sequential-conjugation}
to $\boldsymbol f$, updating each component
exactly once. More precisely, set
\[
\boldsymbol h^{(0)}=\boldsymbol f,
\qquad
\boldsymbol h^{(k)}
=T_k\boldsymbol h^{(k-1)},
\quad k=1,\ldots,K,
\]
and define $\boldsymbol g=\boldsymbol h^{(K)}$.

The initial family is real-valued and feasible, so the lemma ensures
that every partial $c$-transform takes finite values. Each update
preserves feasibility and does not decrease the updated component.
Consequently,
\[
\boldsymbol f=\boldsymbol h^{(0)}
\preceq\boldsymbol h^{(1)}
\preceq\cdots
\preceq\boldsymbol h^{(K)}=\boldsymbol g.
\]
Moreover, the final family is mutually
$c$-conjugate, that is,
\[
g_k=P_c(\boldsymbol g_{\setminus k}),
\qquad k=1,\ldots,K.
\]

To verify the splitting equality, fix $\boldsymbol x\in\Gamma$.
The initial family satisfies $\sum_k f_k(x_k)=c(\boldsymbol x)$.
Since $\boldsymbol g\succeq\boldsymbol f$ and $\boldsymbol g$ remains
feasible on $X$, we obtain
\[
c(\boldsymbol x)
=\sum_{k=1}^K f_k(x_k)
\le\sum_{k=1}^K g_k(x_k)
\le c(\boldsymbol x).
\]
The two inequalities force equality, giving
\[
\sum_{k=1}^K g_k(x_k)=c(\boldsymbol x),
\qquad \boldsymbol x\in\Gamma.
\]

Hence $\boldsymbol g$ is a real-valued, mutually $c$-conjugate
splitting family for $\Gamma$ and dominates $\boldsymbol f$.
Only pointwise feasibility and the order properties of partial
$c$-transforms enter the argument. Neither measurability nor
compactness is required.
\par\nobreak\noindent\hfill$\Box$

\subsubsection*{Proof of Proposition~\ref{Proposition: bounded-splitting} and Remark~\ref{Remark: bounded-conjugate-regularity}}
Let $\boldsymbol g$ be a real-valued mutually $c$-conjugate family.
Fix $\bar{\boldsymbol x}=(\bar x_1,\ldots,\bar x_K)\in X$ and set
$M=\|c\|_\infty$. All values $g_l(\bar x_l)$ are finite.
We first obtain an upper bound for each component and then use these
bounds to obtain lower bounds. For $k\in\{1,\ldots,K\}$ and
$x_k\in X_k$, mutual conjugacy gives
\[
g_k(x_k)
=
\inf_{\boldsymbol x_{\setminus k}}
\Big\{
c(\boldsymbol x)
-\sum_{l\ne k}g_l(x_l)
\Big\}.
\]
Evaluating the infimum at
$x_l=\bar x_l$ for $l\ne k$ gives
\[
\begin{aligned}
g_k(x_k)
\le c(\bar x_1,\ldots,x_k,\ldots,\bar x_K)
-\sum_{l\ne k}g_l(\bar x_l)
\le M-\sum_{l\ne k}g_l(\bar x_l)
=B_k<\infty.
\end{aligned}
\]
The constant $B_k$ is finite and independent of $x_k$, so every
component is bounded above. Having obtained these bounds for all
components, we use $g_l(x_l)\le B_l$ for $l\ne k$ in the conjugacy
identity to deduce
\[
\begin{aligned}
g_k(x_k)
=
\inf_{\boldsymbol x_{\setminus k}}
\Big\{
c(\boldsymbol x)
-\sum_{l\ne k}g_l(x_l)
\Big\}
\ge -M-\sum_{l\ne k}B_l.
\end{aligned}
\]
Consequently,
\[
-M-\sum_{l\ne k}B_l\le g_k(x_k)\le B_k
\qquad\text{for every }x_k\in X_k.
\]
Both bounds are finite, proving boundedness of each component.

If the spaces $X_k$ are topological
and $c$ is continuous,
Lemma~\ref{Lemma: usc-conjugation} implies
that each $g_k$ is upper semicontinuous,
and therefore Borel measurable.
Mutual conjugacy also implies the pointwise feasibility inequality
$\sum_{k=1}^K g_k(x_k)\le c(\boldsymbol x)$.
Boundedness and Borel measurability give $g_k\in L_1(\mu_k)$ for
every Borel probability measure $\mu_k$ on $X_k$. Hence
$\boldsymbol g\in\mathcal G_c$ for every choice of Borel probability
marginals.
\hfill\(\Box\)

\subsubsection*{Proof of Theorem~\ref{Theorem: noncompact c-conjugate}}

Set $M=\|c\|_\infty<\infty$. The proof constructs a dual-admissible
family from the assumed splitting family and then uses equality on the
optimal support to identify its objective value.

\medskip
\noindent\emph{Step 0: construction from the splitting hypothesis.}
By assumption, there exists an optimal plan
$\pi^*\in\Pi(\boldsymbol\mu)$ such that
$\Gamma=\operatorname{supp}(\pi^*)$
admits a real-valued splitting family
$\boldsymbol h$, which need not be measurable.
An optimal plan exists by Lemma~\ref{Lemma: compact space} and the
continuity of $\pi\mapsto\int_X c\,d\pi$ under weak convergence.
The splitting hypothesis concerns the support of the selected optimal
plan $\pi^*$. In particular,
\[
I(\pi^*)=\min_{\pi\in\Pi(\boldsymbol\mu)}\int_X c\,d\pi.
\]
Proposition~\ref{Proposition: splitting} constructs a real-valued mutually $c$-conjugate splitting
family $\boldsymbol g\ge\boldsymbol h$ in $K$ successive
updates. Lemma~\ref{Lemma: usc-conjugation} makes every $g_k$ upper semicontinuous
and hence Borel measurable. By construction, $\boldsymbol g$ satisfies
\begin{align*}
\sum_{k=1}^K g_k(x_k)&\le c(\boldsymbol x), &&\forall\,\boldsymbol x\in X,\\
\sum_{k=1}^K g_k(x_k)&= c(\boldsymbol x), &&\forall\,\boldsymbol x\in\Gamma,
\end{align*}
and, being mutually $c$-conjugate (Definition~\ref{definition: c-conjugate}), for every
$k\in\{1,\dots,K\}$ and every $x_k\in X_k$,
\begin{equation*}
g_k(x_k)=\inf_{\boldsymbol x_{\setminus k}\in X_{\setminus k}}
\Bigl\{c(\boldsymbol x)-\sum_{l\ne k}g_l(x_l)\Bigr\}.
\end{equation*}
It remains to establish integrability, and hence membership in
$\mathcal G_c\subseteq\mathcal F_c$, before integrating the splitting
equality. We will then show that $\boldsymbol g$ maximizes both dual
problems and that their common value equals the primal minimum.

\medskip
\noindent\emph{Step 1: boundedness and integrability
of the potentials.}
By Proposition~\ref{Proposition: splitting}, the family
$\boldsymbol g$
is real-valued and mutually $c$-conjugate.
Since $c$ is bounded and continuous,
Proposition~\ref{Proposition: bounded-splitting} and
Remark~\ref{Remark: bounded-conjugate-regularity}
imply that each $g_k$ is bounded,
upper semicontinuous, and therefore
Borel measurable.
As $K$ is finite, there exists $B<\infty$
such that $
\|g_k\|_\infty\le B,
\text{ for all }k=1,\ldots,K$.
Since each $\mu_k$ is a probability measure, $
\int_{X_k}|g_k|\,d\mu_k
\le B<\infty$.
Thus $g_k\in L_1(\mu_k)$ for every $k$.

Together with the global admissibility
condition~\eqref{eq:dual-feasibility} for $\boldsymbol g$, this yields
$\boldsymbol g\in\mathcal F_c$.
Moreover, the mutual conjugacy
identities~\eqref{eq:mutual-conjugacy} for $\boldsymbol g$ imply $
\boldsymbol g\in\mathcal G_c$.

\medskip
\noindent\emph{Step 2: equality of the constructed dual value and the
primal minimum.}
Because $\pi^*\in\Pi(\boldsymbol \mu)$, its $k$-th marginal is $\mu_k$, so for
each $k$, $\int_{X_k}g_k\,d\mu_k=\int_X g_k(x_k)\,d\pi^*(\boldsymbol x)$; summing,
\begin{equation*}
J(\boldsymbol g)
=\sum_{k=1}^K\int_{X_k}g_k\,d\mu_k
=\int_X\Bigl(\sum_{k=1}^K g_k(x_k)\Bigr)\,d\pi^*(\boldsymbol x).
\end{equation*}
To use the splitting equality in this integral, define
$\Phi(\boldsymbol x)=c(\boldsymbol x)-\sum_{k=1}^K g_k(x_k)$.
This function is Borel measurable and nonnegative by
\eqref{eq:dual-feasibility} for $\boldsymbol g$. It vanishes on $\Gamma$
because the construction preserves the contact equality
\eqref{eq:contact-equality}. Since $X$ is Polish,
it is second countable, and the support of a Borel probability measure
carries full mass. Thus $\pi^*(X\setminus\Gamma)=0$, so $\Phi=0$
$\pi^*$-almost everywhere. It follows that
\[
\int_X\Phi\,d\pi^*=0,
\qquad\text{i.e.}\qquad
\int_X\Bigl(\sum_{k=1}^K g_k(x_k)\Bigr)\,d\pi^*
=\int_X c\,d\pi^*=I(\pi^*).
\]
Applying the marginal identity~\eqref{eq:marginal-dual-identity} to
$\boldsymbol g$ and $\pi^*$, together with optimality of $\pi^*$, gives
\begin{equation}\label{eq:nc-Jg-min}
J(\boldsymbol g)=I(\pi^*)
=\min_{\pi\in\Pi(\boldsymbol \mu)}\int_X c\,d\pi.
\end{equation}

\medskip
\noindent\emph{Step 3: weak duality.}
For every $\boldsymbol f=(f_1,\dots,f_K)\in\mathcal F_c$ and every
$\pi\in\Pi(\boldsymbol \mu)$, admissibility
$\sum_k f_k(x_k)\le c(\boldsymbol x)$ and the marginal constraints give
\[
J(\boldsymbol f)=\sum_{k=1}^K\int_{X_k}f_k\,d\mu_k
=\int_X\Bigl(\sum_{k=1}^K f_k(x_k)\Bigr)\,d\pi
\le\int_X c\,d\pi=I(\pi).
\]
Taking the supremum over $\boldsymbol f\in\mathcal F_c$ and the
minimum over $\pi\in\Pi(\boldsymbol\mu)$ yields
\begin{equation*}
\sup_{\boldsymbol f\in\mathcal F_c}J(\boldsymbol f)
\le\min_{\pi\in\Pi(\boldsymbol \mu)}\int_X c\,d\pi.
\end{equation*}
This is Lemma~\ref{Lemma: weak duality}, applied here to the bounded
continuous cost and the attained primal minimum.

\medskip
\noindent\emph{Step 4: existence of a maximizer over $\mathcal F_c$.}
Because $\boldsymbol g\in\mathcal F_c$, we have
$\sup_{\boldsymbol f\in\mathcal F_c}J(\boldsymbol f)\ge J(\boldsymbol g)$.
Conversely, \eqref{eq:nc-Jg-min} and weak duality
\eqref{eq:weak-duality} give
\[
J(\boldsymbol g)
=\min_{\pi\in\Pi(\boldsymbol \mu)}\int_X c\,d\pi
\ge\sup_{\boldsymbol f\in\mathcal F_c}J(\boldsymbol f).
\]
The two inequalities therefore give
\begin{equation}\label{eq:nc-attain-Fc}
J(\boldsymbol g)=\sup_{\boldsymbol f\in\mathcal F_c}J(\boldsymbol f)
=\max_{\boldsymbol f\in\mathcal F_c}J(\boldsymbol f),
\end{equation}
the supremum being attained at $\boldsymbol g$.

\medskip
\noindent\emph{Step 5: existence of a maximizer in the class $\mathcal G_c$.}
Since $\mathcal G_c\subset\mathcal F_c$, restricting the supremum can only
decrease it: $
\sup_{\boldsymbol f\in\mathcal G_c}J(\boldsymbol f)
\le\sup_{\boldsymbol f\in\mathcal F_c}J(\boldsymbol f)$.
Conversely, $\boldsymbol g\in\mathcal G_c$ (Step~1), so by~\eqref{eq:nc-attain-Fc}, $
\sup_{\boldsymbol f\in\mathcal G_c}J(\boldsymbol f)
\ge J(\boldsymbol g)
=\sup_{\boldsymbol f\in\mathcal F_c}J(\boldsymbol f)$.
The two inequalities give
\begin{equation}\label{eq:nc-attain-Gc}
\sup_{\boldsymbol f\in\mathcal G_c}J(\boldsymbol f)
=\sup_{\boldsymbol f\in\mathcal F_c}J(\boldsymbol f)
=J(\boldsymbol g).
\end{equation}
The family $\boldsymbol g$ consequently attains the supremum over
$\mathcal G_c$ as well. In particular,
$\max_{\boldsymbol f\in\mathcal F_c}J(\boldsymbol f)
=\max_{\boldsymbol f\in\mathcal G_c}J(\boldsymbol f)$.

\medskip
\noindent\emph{Conclusion.}
Combining~\eqref{eq:nc-Jg-min}, \eqref{eq:nc-attain-Fc}
and~\eqref{eq:nc-attain-Gc} with the definition of the primal value gives
\[
\mathsf{MOT}_c(\boldsymbol \mu)
=\min_{\pi\in\Pi(\boldsymbol \mu)}\int_{X} c\,d\pi
=\max_{\boldsymbol f\in\mathcal F_c}J(\boldsymbol f)
=\max_{\boldsymbol f\in\mathcal G_c}J(\boldsymbol f).
\]
The primal minimum is attained at $\pi^*$, and both dual maxima
are attained at the constructed family $\boldsymbol g\in\mathcal G_c$.
Its components are bounded and upper semicontinuous, hence Borel
measurable, as shown in Step~1.
\hfill\(\Box\)

\numberwithin{theorem}{section}
\numberwithin{definition}{section}
\numberwithin{lemma}{section}


\subsubsection{Compactness and dual maximization on noncompact spaces}
\label{dm:appendix:noncompact-class}

\begin{proof}[Proof of Proposition~\ref{dm:proposition:noncompact-class}]
Let $M=\|c\|_\infty$. We first show that total boundedness of the
coordinate cost profiles gives continuity of conjugate potentials in
the original topology. We then combine normalization with finite nets
to prove compactness in the global uniform norm. Finally, a maximizing
sequence yields the dual representation. Throughout the proof, uniform
norms are taken on the full marginal spaces.

\medskip
\noindent\emph{Step 1: Cost pseudometrics and continuity of conjugate
potentials.}
For each $k$, recall that
\[
\Delta_k^c(x_k,y_k)
=\sup_{\boldsymbol z_{\setminus k}\in X_{\setminus k}}
\big|c(x_k,\boldsymbol z_{\setminus k})
-c(y_k,\boldsymbol z_{\setminus k})\big|.
\]
The bound $\Delta_k^c\le2M$ makes this function finite. Symmetry
and the zero diagonal are immediate, while the triangle inequality
follows by applying the triangle inequality before taking the
supremum. Thus $\Delta_k^c$ is a pseudometric. Distinct points may
have zero distance when their cost sections coincide.

Define the section map
\[
\mathscr C_k:X_k\longrightarrow\mathcal C_b(X_{\setminus k}),
\qquad
\mathscr C_k(x_k)(\boldsymbol z_{\setminus k})
=c(x_k,\boldsymbol z_{\setminus k}).
\]
The codomain is a Banach space under the supremum norm, and
\[
\|\mathscr C_k(x_k)-\mathscr C_k(y_k)\|_\infty
=\Delta_k^c(x_k,y_k).
\]
By Assumption~\ref{dm:assumption:cost-totally-bounded}, the image
$\mathscr C_k(X_k)$ is therefore totally bounded in this norm. Its
closure is compact because the ambient space is complete.

To relate this compactness to the original topology, suppose that
$x_k^n\to x_k$. We claim that $\Delta_k^c(x_k^n,x_k)\to0$.
Every subsequence of $\mathscr C_k(x_k^n)$ has a further subsequence
converging uniformly to some $v\in\mathcal C_b(X_{\setminus k})$.
For each fixed $\boldsymbol z_{\setminus k}$, continuity of $c$
also gives
\[
c(x_k^n,\boldsymbol z_{\setminus k})
\longrightarrow c(x_k,\boldsymbol z_{\setminus k}).
\]
Consequently, every such uniform limit equals $\mathscr C_k(x_k)$.
If the full sequence did not converge to $\mathscr C_k(x_k)$ in
norm, a subsequence would stay a positive distance from it, contradicting
the existence and identification of a uniformly convergent further
subsequence. This proves the claim. Since $X_k$ is metrizable, it also
proves continuity of $\mathscr C_k$ for the original topology.

Now let $\boldsymbol f$ be a real-valued mutually $c$-conjugate
family. Its conjugacy identities express finite real values as infima.
For any fixed complementary tuple, changing $x_k$ to $y_k$ changes
the expression under the infimum by at most $\Delta_k^c(x_k,y_k)$.
Comparison of the two infima therefore gives
\[
|f_k(x_k)-f_k(y_k)|\le\Delta_k^c(x_k,y_k)\le2M.
\]
Fixing one point $y_k\in X_k$ shows that $f_k$ is bounded, since
$f_k(y_k)$ is finite. The preceding claim then proves that $f_k$
is continuous in the original topology. In particular, it is integrable
against every Borel probability measure on $X_k$.
Mutual conjugacy also implies the pointwise feasibility inequality.
Thus the class of real-valued mutually conjugate families coincides
with $\mathcal G_c$ for every choice of the marginals, and
$\mathcal G_c\subseteq\prod_k\mathcal C_b(X_k)$.

\medskip
\noindent\emph{Step 2: Nonemptiness and uniform bounds after
normalization.}
The constant family $(-M,0,\ldots,0)$ is real-valued and feasible.
Lemma~\ref{Lemma: sequential-conjugation} turns it into a real-valued
mutually conjugate family. By Step~1, its components are bounded and
continuous. Lemma~\ref{dm:lemma:normalization} then produces an
element of $\mathcal N_c$ using the fixed reference points.
This proves nonemptiness without using existence of a dual optimizer.

For an arbitrary mutually conjugate family, the same normalization
lemma shows that zero-sum constant shifts preserve conjugacy and the
dual objective for every probability marginal tuple. The bounds in
that lemma use only $\Delta_k^c\le2M$ and evaluation at the fixed
reference points; compactness of the original spaces is not needed.
Hence \eqref{dm:eq:uniform-bounds} applies here as well. In particular,
with $C=(2K-1)M$, every $\boldsymbol f\in\mathcal N_c$ satisfies
\[
\|\boldsymbol f\|_\infty\le C,
\qquad
|f_k(x_k)-f_k(y_k)|\le\Delta_k^c(x_k,y_k)
\quad(k=1,\ldots,K).
\]
Neither these estimates nor the normalization conditions depend on
the marginal measures.

\medskip
\noindent\emph{Step 3: Total boundedness and closedness in the global
uniform norm.}
Fix $\varepsilon>0$. For each $k$, choose a finite
$\Delta_k^c$-net $x_{k,1},\ldots,x_{k,m_k}$ of radius
$\varepsilon$: every $x_k\in X_k$ satisfies
$\Delta_k^c(x_k,x_{k,j})<\varepsilon$ for some $j$.
Partition $[-C,C]$ into finitely many sets of diameter at most
$\varepsilon$. Record, for each $\boldsymbol f\in\mathcal N_c$,
the partition cell containing each value $f_k(x_{k,j})$, for
$1\le k\le K$ and $1\le j\le m_k$. This gives a finite list of
cells, with only finitely many possible lists. For every list that
occurs, choose one representative family in $\mathcal N_c$.

If $\boldsymbol f,\boldsymbol g$ have the same list, then
$|f_k(x_{k,j})-g_k(x_{k,j})|\le\varepsilon$ at every selected point.
For arbitrary $x_k$, choose $j$ with
$\Delta_k^c(x_k,x_{k,j})<\varepsilon$. The estimates from Step~2 give
\[
\begin{aligned}
|f_k(x_k)-g_k(x_k)|
&\le |f_k(x_k)-f_k(x_{k,j})|
 +|f_k(x_{k,j})-g_k(x_{k,j})|
 +|g_k(x_{k,j})-g_k(x_k)|\\
&\le2\Delta_k^c(x_k,x_{k,j})+\varepsilon
<3\varepsilon.
\end{aligned}
\]
The same estimate holds for every coordinate and every point of
its marginal space. Every family is therefore at uniform distance
at most $3\varepsilon$ from one of the finitely many representatives.
Since $\varepsilon>0$ is arbitrary, $\mathcal N_c$ is totally bounded
in the product uniform norm.

For closedness, suppose that $\boldsymbol f^n\in\mathcal N_c$
converges uniformly to
$\boldsymbol f\in\prod_k\mathcal C_b(X_k)$.
The transform estimate \eqref{dm:eq:transform-stability} is valid
for bounded costs and bounded families on arbitrary nonempty spaces.
Indeed, its proof compares the functions under the infima, whose
values are finite under these boundedness assumptions. Applying the
estimate with the same cost on both sides yields
\[
\begin{aligned}
\|f_k-P_c(\boldsymbol f_{\setminus k})\|_\infty
&\le\|f_k-f_k^n\|_\infty
 +\|P_c(\boldsymbol f^n_{\setminus k})
       -P_c(\boldsymbol f_{\setminus k})\|_\infty\\
&\le\sum_{\ell=1}^K\|f_\ell^n-f_\ell\|_\infty
\longrightarrow0.
\end{aligned}
\]
Hence $\boldsymbol f$ is mutually conjugate. Uniform convergence
also preserves $f_k(x_k^\circ)=0$ for $k<K$, so
$\boldsymbol f\in\mathcal N_c$. Thus $\mathcal N_c$ is closed in the complete space
$\prod_k\mathcal C_b(X_k)$ and is itself complete. Together with
total boundedness, this proves compactness in the product uniform
norm.

\medskip
\noindent\emph{Step 4: Dual maximization for arbitrary marginals.}
Fix $\boldsymbol\mu\in\prod_k\mathcal P(X_k)$ and write
$V=\mathsf{MOT}_c(\boldsymbol\mu)$. Since $c$ is bounded,
$V\in[-M,M]$. Theorem~\ref{theorem: duality-noncompact} gives
\[
V=\sup_{\boldsymbol u\in\mathcal F_c}
J_{\boldsymbol\mu}(\boldsymbol u).
\]
Choose a sequence $\boldsymbol u^n\in\mathcal F_c$ with
$J_{\boldsymbol\mu}(\boldsymbol u^n)\to V$.
The components of each initial family are real-valued and
integrable, as required by the definition of $\mathcal F_c$.
They need not be bounded or continuous.

Apply Lemma~\ref{Lemma: sequential-conjugation} to each
$\boldsymbol u^n$. The resulting family $\boldsymbol g^n$ is
real-valued, feasible, mutually conjugate, and satisfies
$\boldsymbol g^n\succeq\boldsymbol u^n$.
By Step~1, its components are bounded and continuous. Consequently,
\[
J_{\boldsymbol\mu}(\boldsymbol u^n)
\le J_{\boldsymbol\mu}(\boldsymbol g^n)\le V,
\]
where the second inequality is weak duality.
Normalize $\boldsymbol g^n$ by zero-sum constant shifts to obtain
$\boldsymbol f^n\in\mathcal N_c$. Normalization preserves the
objective, so
\[
J_{\boldsymbol\mu}(\boldsymbol f^n)
=J_{\boldsymbol\mu}(\boldsymbol g^n)\longrightarrow V.
\]
By compactness, a subsequence converges uniformly to some
$\boldsymbol f^*\in\mathcal N_c$. For any two bounded families,
the probability marginal conditions imply
\[
|J_{\boldsymbol\mu}(\boldsymbol f)
-J_{\boldsymbol\mu}(\boldsymbol g)|
\le\sum_{k=1}^K\|f_k-g_k\|_\infty
\le K\|\boldsymbol f-\boldsymbol g\|_\infty.
\]
It follows that $J_{\boldsymbol\mu}(\boldsymbol f^*)=V$.
Every family in $\mathcal N_c$ is admissible for these marginals,
so weak duality bounds its objective above by $V$. The family
$\boldsymbol f^*$ attains that bound, proving
\[
\mathsf{MOT}_c(\boldsymbol\mu)
=\max_{\boldsymbol f\in\mathcal N_c}
J_{\boldsymbol\mu}(\boldsymbol f).
\]
The argument applies to every probability marginal tuple.
It constructs a maximizing family without assuming a splitting
family on an optimal support. Its starting point is duality over
$\mathcal F_c$; continuity, compactness, and existence of a dual
maximizer are obtained within the proof, rather than presupposed
through a duality formula restricted to continuous potentials.
\end{proof}

\subsection{Supporting results}

We recall the functional-analytic and measure-theoretic results used
in the preceding proofs.

\begin{theorem}[Fenchel--Rockafellar duality \cite{villani2003}]
\label{SuppRes:Theo:Fenchel}
Let $E$ be a normed vector space with topological dual $E^*$, and let
$\Theta,\Xi:E\longrightarrow\mathbb R\cup\{+\infty\}$ be convex
functions. Denote their Legendre--Fenchel transforms by $\Theta^*$
and $\Xi^*$. Suppose that there exists $z_0\in E$ such that
\[
\Theta(z_0)<+\infty,\qquad \Xi(z_0)<+\infty,
\]
and that $\Theta$ is continuous at $z_0$. If
$\inf_{z\in E}[\Theta(z)+\Xi(z)]$ is finite, then
\[
\inf_{z\in E}\big[\Theta(z)+\Xi(z)\big]
=\max_{z^*\in E^*}
\big[-\Theta^*(-z^*)-\Xi^*(z^*)\big].
\]
\end{theorem}

\begin{theorem}[Riesz representation theorem \cite{knapp2007}]
\label{SuppRes:Theo:Riez}
Let $X_1,\ldots,X_K$ be compact metric spaces, set
$X=\prod_{k=1}^K X_k$, and equip $E=\mathcal C(X)$ with the
supremum norm. The topological dual $E^*$ is isometrically identified
with the space $\mathcal M(X)$ of finite signed regular Borel measures,
equipped with the total variation norm. The correspondence associates
to each $\mu\in\mathcal M(X)$ the functional
\[
\mathcal L_\mu(u)=\int_X u(\boldsymbol x)\,d\mu(\boldsymbol x),
\qquad u\in E.
\]
\end{theorem}

\begin{theorem}[Arzel\`a--Ascoli \cite{santambrogio2015}]
\label{SuppRes:Theo:Ascoli}
Let $X$ be a compact metric space, and let $(f_n)_{n\in\mathbb N}$
be an equicontinuous and uniformly bounded sequence of real-valued
functions on $X$. Here equicontinuity means that, for every
$\varepsilon>0$, there exists $\delta>0$ such that
\[
d(x,y)<\delta
\quad\Longrightarrow\quad |f_n(x)-f_n(y)|<\varepsilon
\qquad\text{for all }x,y\in X\text{ and }n\in\mathbb N.
\]
Then some subsequence $(f_{n_k})_{k\in\mathbb N}$ converges
uniformly on $X$ to a continuous function $f:X\longrightarrow\mathbb R$.
\end{theorem}

\begin{theorem}[Prokhorov's theorem \cite{villani2008}]
\label{SuppRes:Theo:Tightness}
Let $X$ be a Polish space. A family $P\subset\mathcal P(X)$ is
relatively compact for the weak topology, meaning that its closure
is compact, if and only if it is tight.
\end{theorem}

\begin{definition}[Tightness \cite{villani2008}]
\label{SuppRes:Def:Tightness}
Let $X$ be a Polish space. A family $P\subset\mathcal P(X)$ is
\emph{tight} if, for every $\varepsilon>0$, there exists a compact
set $\mathcal K_\varepsilon\subset X$ such that
\[
\mu(X\setminus\mathcal K_\varepsilon)\le\varepsilon
\qquad\text{for every }\mu\in P.
\]
\end{definition}

\begin{definition}[Weak convergence \cite{Billingsley}]
\label{SuppRes:Def:WeakConv}
Let $S$ be a metric space and let $P_n,P\in\mathcal P(S)$.
The sequence $(P_n)$ converges weakly to $P$, written
$P_n\rightharpoonup P$, if
\[
\int_S f\,dP_n\longrightarrow\int_S f\,dP
\qquad\text{for every }f\in\mathcal C_b(S).
\]
\end{definition}

\begin{lemma}[Regularity of Borel probability measures \cite{knapp2007}]
Let $X$ be a Polish space with Borel $\sigma$-algebra $\mathcal B(X)$,
and let $\mu$ be a Borel probability measure on $X$.
For every Borel set $A\subset X$,
\[
\mu(A)
=\sup\{\mu(K):K\subset A,\ K\text{ compact}\}
=\inf\{\mu(O):A\subset O,\ O\text{ open}\}.
\]
\end{lemma}

\bibliographystyle{plainnat}
\bibliography{biblio}
\end{document}